\documentclass[11pt]{article}
\usepackage{amsfonts,amssymb,amsthm,amsmath}
\usepackage{graphicx}
\usepackage{caption2}
\usepackage{amsthm}
\usepackage{graphicx}
\usepackage{float}
\numberwithin{equation}{section}

\newtheorem{theorem}{Theorem}[section]
\newtheorem{lemma}[theorem]{Lemma}
\newtheorem{definition}[theorem]{Definition}
\newtheorem{corollary}[theorem]{Corollary}
\newtheorem{proposition}[theorem]{Proposition}
\newtheorem{remark}[theorem]{Remark}
\newtheorem{example}[theorem]{Example}

\newcommand{\Ent}{\mbox{\rm Ent}}
\newcommand{\supp}{\mbox{\rm supp}}

\newcommand{\Gaps}{\mbox{\rm gaps}}
\def\eI{[0,1]}
\def\1{\,{\makebox[0pt][c]{\normalfont    1}
\makebox[2.5pt][c]{\raisebox{3.5pt}{\tiny {$\|$}}}
\makebox[-2.5pt][c]{\raisebox{1.7pt}{\tiny {$\|$}}}
\makebox[2.5pt][c]{} }}
\def\one{\1 }

\newcommand{\Dom}{{\it{Dom}}}

\newcommand{\R}{{\mathbb{R}}}

\newcommand{\G}{{\mathcal{G}}}
\newcommand{\E}{{\mathcal{E}}}
\newcommand{\C}{{\mathcal{C}}}

\newcommand{\varh}{{h}}

\newcommand{\Pe}{{\mathcal{P}}}

\DeclareSymbolFont{AMSb}{U}{msb}{m}{n}
\newcommand{\N}{{\mathbb{N}}}
\newcommand{\Z}{{\mathbb{Z}}}
\newcommand{\Q}{{\mathbb{Q}}}
\newcommand{\Pp}{{\mathbb{P}}}

\newcommand{\D}{{\mathbb{D}}}
\newcommand{\T}{{\mathbb{T}}}
\newcommand{\DD}{\overline{\mathbb{D}}}
\newcommand{\LL}{{\mathbb{L}}}
\newcommand{\EE}{{\mathbb{E}}}

\newcommand{\leb}{{\mbox{Leb}}}
\newcommand{\Cyl}{{\mathfrak{C}}}
\newcommand{\Syl}{{\mathfrak{S}}}
\newcommand{\Zyl}{{\mathfrak{Z}}}
\def\smint{{\mbox{$\int$}}}

\begin{document}

\title{Entropic Measure and Wasserstein Diffusion}

\author{Max-K. von Renesse, Karl-Theodor Sturm}

\date{}

\maketitle

\begin{abstract}
We construct a new random probability measure on the sphere and on the unit
interval which in both cases has a Gibbs structure with the relative entropy
functional as Hamiltonian. It satisfies a quasi-invariance formula with
respect to  the action of smooth diffeomorphism of the sphere and the
interval respectively. The associated integration by parts formula is used to
construct two classes of diffusion processes on  probability measures (on the sphere or the
unit interval) by Dirichlet form methods. The first one is
closely related to Malliavin's Brownian motion on the homeomorphism group.
The second one is a
probability valued stochastic perturbation of
the heat flow, whose intrinsic metric is  the quadratic Wasserstein
distance. It may be regarded as the canonical diffusion process on the Wasserstein space.
\end{abstract}

\begin{section}{Introduction}
{\bf (a)}\ Equipped with the $L^2$-Wasserstein distance $d_W$ (cf. (\ref{dw})),
the space $\Pe(M)$ of probability measures on an Euclidean or Riemannian space
$M$ is itself a rich object of geometric interest. Due to
the fundamental works of Y.~Brenier, R.~McCann, F.~Otto, C.~Villani
and many others (see e.g. \cite{MR1100809,MR98e:82003,MR2002k:58038,MR2002j:35180,MR2001k:58076,MR1964483})  there are well
understood and powerful concepts of geodesics, exponential maps,
tangent spaces $T_\mu\Pe(M)$ and gradients $Du(\mu)$ of functions  on this space. In a
certain sense,  $\Pe(M)$ can be regarded as an infinite dimensional
Riemannian manifold, or at least as an infinite dimensional Alexandrov space with nonnegative lower curvature bound if the base manifold $(M,d)$ has nonnegative sectional curvature.

A central role is played by the {\em relative entropy} $:
\Pe(M)\to\R\cup\{+\infty\}$ with respect to the Riemannian volume
measure $dx$ on $M$
$$\Ent(\mu)=\left\{\begin{array}{ll}\int_M \rho\log\rho\, dx,\qquad&\mbox{if }d\mu(x)\ll dx \ \mbox{ with }\rho(x)=\frac{d\mu(x)}{dx}\\
+\infty,&\mbox{else}.\end{array}\right.$$
The relative entropy as a function on the geodesic space
$(\Pe(M),d_W)$ is $K$-convex for a given number $K\in\R$ if and only
if the {\em Ricci curvature} of the underlying manifold $M$ is
bounded from below by $K$, \cite{MR2142879,MR2237206}. The gradient flow for the
relative entropy in the geodesic space $(\Pe(M),d_W)$ is given by
the  {\em heat
equation} $\frac{\partial}{\partial t}\mu=\Delta\mu$ on $M$,
\cite{MR1617171}.
More generally, a large class of evolution equations can be treated as gradient flows for suitable free energy functionals
$S:\Pe(M) \to \R$, \cite{MR1964483}.

\medskip

What is missing until now, is a natural 'Riemannian volume measure'
$\Pp$ on $\Pe(M)$. The basic requirement will be an {\em integration by parts formula}
for the gradient. This will imply the {\em closability} of the pre-Dirichlet form
$$\EE(u,v)=\int_{\Pe(M)} \langle Du(\mu),Dv(\mu)\rangle_{T_\mu} \, d\Pp(\mu)$$
in $L^2(\Pe(M),\Pp)$, -- which in turn will be the key tool in order
to develop an analytic and stochastic calculus on $\Pe(M)$. In
particular, it will allow us to construct a kind of {\em Laplacian}
and a kind of {\em Brownian motion} on $\Pe(M)$.
Among others, we intend to use the powerful machinery
of Dirichlet forms to study stochastically perturbed gradient flows
on $\Pe(M)$ which -- on the level of the underlying spaces $M$ --
will lead to a new concept of SPDEs (preserving probability by
construction).

Instead of constructing a  'uniform distribution' $\Pp$ on
$\Pe(M)$, for various reasons, we prefer to construct a probability measure $\Pp^\beta$ on
$\Pe(M)$ formally given as
\begin{equation}\label{formal}d\Pp^\beta(\mu)=\frac1{Z_\beta} \,
e^{-\beta\cdot \Ent(\mu)}\, d\Pp(\mu) \end{equation} for $\beta>0$ and  some normalization constant  $Z_\beta$. (In the language of
statistical mechanics, $\beta$ is the 'inverse temperature' and
$Z_\beta$ the 'partition function' whereas the entropy plays the
role of a Hamiltonian.)

\medskip\bigskip

{\bf (b)}\  One of the basic results of this paper is the rigorous construction
of such a {\em entropic measure} $\Pp^\beta$ in the one-dimensional case, i.e.
$M=S^1$ or $M=\eI$. We will essentially make use of the
representation of probability measures by their inverse distributions function $g_\mu$.
It allows to transfer the problem of
constructing a measure $\Pp^\beta$ on the space of probability measures
$\Pe(\eI)$ (or $\Pe(S^1)$) into the problem of constructing a measure
$\Q^\beta_0$ (or $\Q^\beta$) on the space $\G_0$ (or $\G$, resp.) of nondecreasing functions from $\eI$ (or $S^1$, resp.) into itself.

In terms of the measure $\Q^\beta_0$ on $\G_0$, for instance, the
formal characterization (\ref{formal}) then reads as follows
\begin{equation}\label{formal-q-intro}
d\Q_0^\beta(g)=\frac1{Z_\beta} \, e^{-\beta\cdot S(g)}\, d\Q_0(g).
\end{equation}
Here $\Q_0$ denotes some 'uniform distribution' on $\G_0\subset
L^2(\eI)$  and $S:\G_0\to [0,\infty]$ is the {\em entropy functional}
$$S(g):=\Ent(g_*\leb)=-\int_0^1 \log g'(t)\,dt.$$
This representation  is reminiscent
of Feynman's heuristic picture of the Wiener measure, --- now with the energy  $$H(g)=\int_0^1 g'(t)^2dt$$ of a path replaced by its entropy.
$\Q_0^\beta$  will turn out to be (the law of) the {\em Dirichlet process} or {\em normalized Gamma process}.

\medskip\bigskip

{\bf (c)}\  The key result here is the {\em quasi-invariance}  -- or in other words a {\em change of variable formula} -- for the measure $\Pp^\beta$ (or $\Pp^\beta_0$)
under push-forwards $\mu\mapsto h_*\mu$ by means of smooth diffeomorphisms $h$ of $S^1$ (or $\eI$, resp.).
This is equivalent to the quasi-invariance of the measure $\Q^\beta$ under translations $g\mapsto h\circ g$ of the
semigroup $\G$ by smooth
$h\in \G$.
The density
$$\frac{d\Pp^\beta(h_*\mu)}{d\Pp^\beta(\mu)}=X_h^\beta\cdot Y_h^0(\mu)$$
consists of two terms. The first one
$$X_h^\beta (\mu)=\exp\left(\beta\int_{S^1}\log h'(t)d\mu(t)\right)$$
can be interpreted as $\exp(-\beta\Ent(h_*\mu)) / \exp(-\beta\Ent(\mu))$ in accordance with our formal interpretation
(\ref{formal}).
The second one
$$Y_h^0(\mu)=
\prod_{I\in\Gaps(\mu)}\frac{\sqrt{h'(I_-)\cdot h'(I_+)}}{|h(I)|/|I|}
$$
can be interpreted as the change of variable formula for the (non-existing) measure $\Pp$.
Here $\Gaps(\mu)$ denotes the set of intervals $I=]I_-,I_+[\subset S^1$ of maximal length with $\mu(I)=0$.
Note that $\Pp^\beta$ is concentrated on the set of $\mu$ which have no atoms and not absolutely continuous parts and whose
supports have Lebesgue measure 0.

\medskip\bigskip

{\bf (d)}\
The tangent space at a given point $\mu$ in $\Pe=\Pe(S^1)$ (or in $\Pe_0=\Pe(\eI)$)
will be an appropriate
completion of the space $\C^\infty(S^1,\R)$ (or $\C^\infty(\eI,\R)$, resp.).  The action of a tangent vector $\varphi$
on $\mu$ ('exponential map') is given by the push forward
$\varphi_*
\mu$.
This leads to the notion of the directional derivative
$$D_\varphi u(\mu)=\lim_{t\to0}\frac1t\left[
u((Id+t\varphi)_*\mu)-u(\mu)\right]$$
for functions $u:\Pe\to\R$.
The quasi-invariance of the measure $\Pp^\beta$ implies an integration by parts formula (and thus the closability)
$$D^*_\varphi u=-D_\varphi u- V_\varphi\cdot u$$
with drift $V_\varphi=\lim_{t\to0}\frac1t(Y^\beta_{Id+t\varphi}-1)$.

The subsequent construction
will strongly depend on the choice of the norm on the tangent spaces
$T_\mu\Pe$. Basically, we will encounter two important cases.

\medskip\bigskip

{\bf (e)}\
Choosing
 $T_\mu\Pe=H^s(S^1,\leb)$ for
some $s>1/2$ --- independent of $\mu$ --- leads to
a regular, local, recurrent Dirichlet form $\E$ on $L^2(\Pe,\Pp^\beta)$ by
$$\E(u,u)=\int_{\Pe}\sum_{k=1}^\infty |D_{\varphi_k} u(\mu)|^2\, d\Pp^\beta(\mu).$$
where $\{\varphi_k\}_{k\in\N}$ denotes some complete orthonormal system in  the Sobolev space $H^s(S^1)$.
According to the theory of Dirichlet forms on locally compact spaces \cite{MR96f:60126}, this form is   associated with a continuous Markov process on $\Pe(S^1)$
which is reversible with respect to the measure $\Pp^\beta$. Its generator is given by
\begin{equation}\label{gener}
\frac12\sum_k D_{\varphi_k}D_{\varphi_k}+\frac12\sum_k V_{\varphi_k}\cdot D_{\varphi_k}.
\end{equation}

\medskip

This process
$(g_t)_{t\ge0}$ is closely related to the stochastic processes on the diffeomorphism group of $S^1$ and to the
'Brownian motion' on the homeomorphism group of $S^1$, studied by Airault, Fang, Malliavin, Ren, Thalmaier and others \cite{MR2082490,AM06, AR02, Fang02,MR2091355,MR1713340}.
These are processes
with generator $\frac12\sum_k D_{\varphi_k}D_{\varphi_k}$.
Hence, one advantage of our approach is to identify a probability measure $\Pp^\beta$ such that these processes --- after adding a suitable drift --- are reversible.

Moreover,  previous approaches are restricted to $s\ge 3/2$ whereas our construction applies to all  cases $s>1/2$.

\medskip\bigskip

{\bf (f)}\
Choosing
 $T_\mu\G=L^2(\eI,\mu)$
 leads to the {\em Wasserstein Dirichlet form}
$$\EE(u,u)=\int_{\Pe_0} \|D u(\mu)\|_{L^2(\mu)}^2\, d\Pp_0^\beta(\mu)$$
on $L^2(\Pe_0,\Pp_0^\beta)$.
Its square field operator is the squared norm of the Wasserstein gradient and its intrinsic distance (which governs the short time asymptotic of the process)
coincides with the $L^2$-Wasserstein metric.
The associated  continuous Markov process $(\mu_t)_{t\ge0}$ on $\Pe(\eI)$, which we shall call  {\em Wasserstein diffusion}, is reversible w.r.t. the entropic measure $\Pp_0^\beta$. It can be regarded as a stochastic
perturbation of the Neumann heat flow on $\Pe(\eI)$ with small time Gaussian behaviour measured in terms of kinetic   energy.
\end{section}

\begin{section}{Spaces of Probability Measures and Monotone Maps}

The goal of this paper is to study stochastic dynamics on spaces
$\Pe(M)$ in case  $M$ is the unit interval
$[0,1]$ or the unit circle $S^1$.

\begin{subsection}{The Spaces $\Pe_0=\Pe(\eI)$ and $\G_0$}

Let us collect some basic facts for the space $\Pe_0=\Pe([0,1])$  of
probability measures on
 the unit interval $[0,1]$ the proofs of which
can be found in
 the monograph \cite{MR1964483}.
Equipped with the {\em $L^2$-Wasserstein distance} $d_W$, it is
 a compact metric space. Recall that
\begin{equation}\label{dw}
d_W(\mu,\nu) := \inf_{\gamma} \left(\iint_{\eI^2} |x-y|^2
\gamma(dx,dy)\right)^{1/2}, \end{equation}
 where the infimum is
taken over all probability measures $\gamma \in \Pe(\eI^2)$ having
marginals $\mu$ and $\nu$ \ (i.e. $\gamma(A\times M)=\mu(A)$ and
$\gamma(M\times B)=\nu(B)$ for all $A,B\subset M$).

Let $\G_0$ denote the space of all right continuous nondecreasing
maps $g: [0,1[\,\to\eI$
 equipped with the $L^2$-distance
$$\|g_1-g_2\|_{L^2}=\left(\int_0^1 |g_1(t)-g_2(t)|^2dt\right)^{1/2}.$$
Moreover, for notational convenience each  $g \in \G_0$ is extended to the full interval $\eI$ by  $g(1):=1$.  The map
$$\chi: \ \G_0\to\Pe_0, \ g\mapsto
g_*\leb$$ (= push forward of the Lebesgue measure on $\eI$ under the
map $g$) establishes an isometry between $(\G_0, \|.\|_{L^2})$ and
$(\Pe_0,d_W)$. The inverse map $\chi^{-1}: \ \Pe_0\to\G_0, \
\mu\mapsto g_\mu$ assigns to each probability measure  $\mu\in\Pe_0$
its {\em inverse distribution function} defined by
\begin{equation}\label{invdist}g_\mu(t):= \inf \{ s\in\eI: \ \mu[0,s]  > t\}
\end{equation}
with $\inf\emptyset:=1$. In particular, for all $\mu,\nu\in\Pe_0$
\begin{equation}\label{L2equ}d_W(\mu,\nu)=\|g_\mu-g_\nu\|_{L^2}.\end{equation}

\medskip

For each $g\in\G_0$  the {\em generalized inverse}  $g^{-1}\in\G_0$ is
defined by $g^{-1}(t)=\inf\{s\ge0: \  g(s) >t\}$. Obviously,
\begin{equation}\label{L1equ}\|g_1-g_2\|_{L^1}=\|g_1^{-1}-g_2^{-1}\|_{L^1}
\end{equation} (being simply
the area between the graphs) and $(g^{-1})^{-1} =g$. Moreover,
$g^{-1}( g(t))= t$ for all $t$ provided  $g^{-1}$ is continuous.
(Note that under the measure $\Q_0^\beta$ to be constructed below the
latter will be satisfied for a.e. $g\in\G_0$.)

On $\G_0$, there exist various canonical topologies: the
$L^2$-topology of $\G_0$ regarded as subset of $L^2(\eI,\R$); the image
of the weak topology on $\Pe_0$ under the map $\chi^{-1}:\
\mu\mapsto g_\mu$ (= inverse distribution function); the image of
the weak topology on $\Pe_0$ under the map $ \mu\mapsto g_\mu^{-1}$
(= distribution function). All these -- and several other --
topologies coincide.

\begin{proposition}\label{g-top}
For each sequence $(g_n)_n\subset\G_0$, each $g\in\G_0$ and each
$p\in[1,\infty[$ the following are equivalent:
 \begin{enumerate}
 \item
$g_n(t)\to g(t)$ for each $t\in\eI$ in which $g$ is continuous;
\item
$g_n\to g$ in $L^p(\eI)$;
\item
$g_n^{-1}\to g^{-1}$ in $L^p(\eI)$;
\item
$\mu_{g_n}\to\mu_g$ weakly;
\item
$\mu_{g_n}\to\mu_g$ in $d_W$.
\end{enumerate}
In particular, $\G_0$ is  compact.
\end{proposition}

Let us briefly sketch the main arguments of the
\begin{proof}
Since all the functions $g_n$ and $g_n^{-1}$ are bounded, properties
(ii) and (iii) obviously are independent of $p$. The equivalence of
(ii) and (iii) for $p=1$ was already stated in (\ref{L1equ}) and the
equivalence between (ii) for $p=2$ and (v) was stated in
(\ref{L2equ}). The equivalence of (iv) and (v) is the well known
fact that the Wasserstein distance metrizes the weak topology.
Another well known characterization of weak convergence states that
(iv) is equivalent to (i'): $g_n^{-1}(t)\to g^{-1}(t)$ for each
$t\in\eI$ in which $g^{-1}$ is continuous. Finally,
$(i')\Leftrightarrow(i)$ according to the equivalence
$(ii)\Leftrightarrow(iii)$ which allows to pass from convergence of
distribution functions $g_n^{-1}$ to convergence of inverse
distribution functions $g_n$. The last assertion follows from the compactness of $\Pe_0$ in the weak topology.
\end{proof}
\end{subsection}

\begin{subsection}{The Spaces $\G$, $\G_1$ and $\Pe=\Pe(S^1)$}

 Throughout this paper, $S^1=\R/\Z$ will always
denote the circle of length 1. It inherits the group operation +
from $\R$ with neutral element 0. For each $x,y\in S^1$  the
positively oriented segment from $x$ to $y$ will be denoted by
$[x,y]$ and its length by $|[x,y]|$. If no ambiguity is possible,
the latter will also be denoted by $y-x$. In contrast to that,
$|x-y|$ will denote the $S^1$-distance between $x$ and $y$. Hence,
in particular, $|[y,x]|=1-|[x,y]|$ and
$|x-y|=\min\{|[y,x]|,|[x,y]|\}$.
 A family of points $t_1,\ldots, t_N\in S^1$ is called an
'ordered family' if $\sum_{i=1}^N |[t_{i},t_{i+1}]|=1$ with
$t_{N+1}:=t_1$ (or in other words if all the open segments $]t_{i},
t_{i+1}[$ are disjoint).

Put
$$\G(\R)=\{g:\R\to\R \ \mbox{right continuous nondecreasing with }
g(x+1)=g(x)+1 \ \mbox{for all }x\in\R\}.$$  Due to the required
equivariance with respect to the group action of $\Z$,
 each map $g\in\G(\R)$ induces uniquely a map $\pi(g): S^1\to S^1$. Put $\G:=\pi(\G(\R))$.
The monotonicity of the functions in $\G(\R)$ induces also a kind of
monotonicity of maps in $\G$: each continuous $g\in\G$ will be order
preserving and homotopic to the identity map. In the sequel,
however, we often will have to deal with discontinuous $g\in\G$. The
elements $g\in\G$ will be called {\em monotone maps} of $S^1$. $\G$
is a compact subspace of the $L^2$-space of maps from $S^1$ to $S^1$
with metric $\|g_1-g_2\|_{L^2}=\left(\int_ {S^1}
|g_1(t)-g_2(t)|^2dt\right)^{1/2}$.

With the composition $\circ$ of maps, $\G$ is a semigroup. Its
neutral element $e$ is the identity map.  Of particular interest in
the sequel will be the  semigroup $\G_1=\G/S^1$  where functions
$g,h\in\G$ will be identified if
  $g(.)=h(.+a)$ for some $a\in S^1$.

\begin{proposition}
The map
 $$\chi:\G_1\to\Pe, \ g\mapsto
g_*\leb$$  (= push forward of the Lebesgue measure on $S^1$
under the map $g$) and its inverse
$\chi^{-1}:\Pe\to\G_1, \ \mu\mapsto g_\mu$ (with $g_\mu$ as
defined  in (\ref{invdist}))
 establish an isometry between
 the space $\G_1$ equipped with the induced
$L^2$-distance
$$\|g_1-g_2\|_{\G_1}=\left(\inf_{s\in S^1}\int_ {S^1} |g_1(t)-g_2(t+s)|^2dt\right)^{1/2}$$
and
 the space $\Pe$ of
probability measures on $S^1$ equipped with the $L^2$-Wasserstein
distance.
In particular,  $\G_1$ is compact.
\end{proposition}

\begin{proof} The bijectivity of $\chi$ and $\chi^{-1}$ is clear.
It remains to prove that
\begin{equation}\label{L2-S1}d_W(\mu,\nu)=\|g_\mu-g_\nu\|_{\G_1}\end{equation}
for all $\mu,\nu\in\Pe$. Obviously, it suffices to prove this for
all absolutely continuous $\mu,\nu$ (or equivalently for strictly
increasing $g_\mu, g_\nu$) since the latter  are dense in $\Pe$ (or
in $\G_1$, resp.). For such a pair of measures, there exists a map
$F:S^1\to S^1$ ('transport map') which minimizes the transportation
costs \cite{MR1964483}. Fix any point in $S^1$, say 0, and put
$s=F(0)$. Then the map $F$ is a transport map for the mass $\mu$ on
the segment $]0,1[$ onto the mass $\nu$ on the segment $]s,s+1[$.
Since these segments are isometric to the interval $]0,1[$, the
results from the previous subsection imply that the minimal  cost
for such a transport is given by $\int_ {S^1}
|g_1(t)-g_2(t+s)|^2dt$. Varying over $s$ finally proves the claim.
\end{proof}

\end{subsection}

\end{section}

\begin{section}{Dirichlet Process and Entropic Measure}

\begin{subsection}{Gibbsean Interpretation and Heuristic Derivation of the Entropic Measure}

One of the basic results of this paper is the rigorous construction
of a measure $\Pp^\beta$ formally given as (\ref{formal}) in the one-dimensional case, i.e.
$M=S^1$ or $M=\eI$. We will essentially make use of the isometries
$\chi:\G_1\to\Pe=\Pe(S^1), g\mapsto g_*\leb$ and
$\chi:\G_0\to\Pe_0=\Pe(\eI)$. They allow to transfer the problem of
constructing measures $\Pp^\beta$ on  spaces of probability measures
$\Pe$ (or $\Pe_0$) into the problem of constructing  measures
$\Q^\beta$ (or $\Q^\beta_0$) on  spaces of functions $\G_1$ (or
$\G_0$, resp.).
In terms of the measure $\Q^\beta_0$ on $\G_0$, for instance, the
formal characterization (\ref{formal}) then reads as follows
\begin{equation}\label{formal-q}
\Q_0^\beta(dg)=\frac1{Z_\beta} \, e^{-\beta\cdot S(g)}\, \Q_0(dg).
\end{equation}
Here $\Q_0$ denotes some 'uniform distribution' on $\G_0\subset
L^2(\eI)$  and $S:\G_0\to [0,\infty]$ is the {\em entropy
functional} $S(g):=\Ent(g_*\leb)$. If $g$ is absolutely continuous
then $S(g)$ can be expressed explicitly as
$$S(g)=-\int_0^1 \log g'(t)\,dt.$$
 The representation (\ref{formal-q}) is reminiscent
of Feynman's heuristic picture of the Wiener measure. Let us briefly
recall the latter and try to use it as a guideline for our
construction of the measure $\Q^\beta_0$.

\medskip

According to this heuristic picture, the Wiener measure
$\mathbf{P}^\beta$ with diffusion constant $\sigma^2=1/\beta$ should
be interpreted (and could be constructed) as
\begin{equation}\label{Feynman}
\mathbf{P}^\beta(dg)=\frac1{Z_\beta}\, e^{-\beta\cdot H(g)}\,
\mathbf{P}(dg)
\end{equation} with the energy functional $H(g)=\frac12\int_0^1
g'(t)^2dt$. Here $\mathbf{P}(dg)$ is assumed to be the 'uniform
distribution' on the space $\G^*$ of all continuous paths
$g:\eI\to\R$ with $g(0)=0$. Even if such a uniform distribution existed, typically almost all paths $g$ would have infinite energy. Nevertheless, one can overcome this difficulty as follows.

Given any finite partition $\{0=t_0<t_1< \dots< t_N=1\}$ of $\eI$,
one should replace the energy $H(g)$ of the path $g$ by the energy
of the piecewise linear interpolation of $g$
$$H_N(g)=\inf\left\{ H(\tilde g): \ \tilde g\in\G^*, \ \tilde g(t_i)=g(t_i) \
\forall i\right\}=
\sum_{i=1}^N\frac{|g(t_i)-g(t_{i-1})|^2}{2(t_i-t_{i-1})} .
$$
Then (\ref{Feynman}) leads to the following explicit representation
for the finite dimensional distributions
\begin{equation}\label{Feynman2}
\mathbf{P}^\beta\left(g_{t_1}\in dx_1,\ldots,g_{t_N}\in dx_N\right)=
\frac1{Z_{\beta,N}}\exp\left(-\frac\beta2
\sum_{i=1}^N\frac{|x_i-x_{i-1}|^2}{t_i-t_{i-1}} \right)\, p_N(dx_1,
\ldots,x_N).
\end{equation}
Here $p_N(dx_1, \ldots,x_N)=\mathbf{P}\left(g_{t_1}\in
dx_1,\ldots,g_{t_N}\in dx_N\right)$ should be a 'uniform
distribution' on $\R^N$ and $Z_{\beta,N}$ a normalization constant.
Choosing $p_N$ to be the $N$-dimensional Lebesgue measure makes the
RHS of (\ref{Feynman2})  a projective family of probability
measures. According to Kolmogorov's extension theorem this family
has a unique projective limit, the Wiener measure $\mathbf{P}^\beta$
on $\G^*$ with diffusion constant $\sigma^2=1/\beta$.

\medskip

Now let us try to follow this procedure with the entropy functional
$S(g)$ replacing the energy functional $H(g)$. Given any finite
partition $\{0=t_0<t_1< \dots< t_N<t_{N+1}=1\}$ of $\eI$, we will
replace the entropy $S(g)$ of the path $g$ by the entropy of the
piecewise linear interpolation of $g$
$$S_N(g)=
\inf\left\{ S(\tilde g): \ \tilde g\in\G_0, \ \tilde g(t_i)=g(t_i) \
\forall i\right\}=
-\sum_{i=1}^{N+1}\log\frac{g(t_i)-g(t_{i-1})}{t_i-t_{i-1}}\cdot
(t_i-t_{i-1}).
$$
This leads to the following expression for the finite dimensional
distributions
\begin{eqnarray}\label{q-approx} \nonumber\lefteqn{
\Q_0^\beta\left( g_{t_1}\in dx_1,\ldots, g_{t_N}\in
dx_N\right)}\\
&=&\frac1{Z_{\beta,N}}\exp\left(\beta\sum_{i=1}^{N+1}\log\frac{x_i-x_{i-1}}{t_i-t_{i-1}}\cdot
(t_i-t_{i-1}) \right)\ q_N(dx_1\ldots dx_N)
\end{eqnarray}
where $q_N(dx_1, \ldots,x_N)=\Q_0\left(g_{t_1}\in
dx_1,\ldots,g_{t_N}\in dx_N\right)$ is a 'uniform distribution' on
the simplex $\Sigma_N=\left\{(x_1,\ldots,x_N)\in \eI^N: \
0<x_1<x_2\ldots<x_N<1\right\}$ and $x_0:=0$, $x_{N+1}:=1$.

\medskip

What is a 'canonical' candidate for $q_N$? A natural requirement
will be the invariance property
\begin{eqnarray}\label{q-inv}\nonumber
q_N(dx_1,\ldots, dx_N)&=&\left[(\Xi^{x_{i-1},x_{i+k}})_*\,
q_k(dx_i,\ldots, dx_{i+k-1})\right]\\
&&\qquad\quad dq_{N-k}(dx_1,\ldots,dx_{i-1},dx_{i+k},\ldots, dx_N)
\end{eqnarray}  for all $1\le k\le N$ and all $1\le i\le
N-k+1$ with the convention $x_0=0, x_{N+1}=1$ and the rescaling map
$\Xi^{a,b}:\,]0,1[^{k}\to\,\R^k, y_j \mapsto  y_j ({b-a})+a$ for $j=1, \cdots, {k}$.

If the $q_N$, $N\in\N$, were probability measures then the
invariance property  admits the following interpretation: under
$q_N$, the distribution of the $(N-k)$-tuple
$(x_1,\ldots,x_{i-1},x_{i+k},\ldots, x_N)$ is nothing but $q_{N-k}$;
and under $q_N$, the distribution of the $k$-tuple $(x_i,\ldots,
x_{i+k-1})$ of points in the interval $]x_{i-1},x_k[$ coincides ---
after rescaling of this interval --- with $q_k$.
 Unfortunately, no family of {\em probability} measures $q_N, N\in\N$ with
 property (\ref{q-inv}) exists. However, there is a family of measures
 with this property.

 By iteration of the invariance property (\ref{q-inv}), the choice of
the measure $q_1$ on the interval $\Sigma_1=\,]0,1[$ will determine
all the measures $q_N$, $N\in\N$. Moreover, applying (\ref{q-inv})
for $N=2$, $k=1$ and both choices of $i$ yields
\begin{equation}\label{q1-inv}
\left[(\Xi^{0,x_1})_*\, q_1(dx_2)\right]\, dq_{1}(dx_1)=
\left[(\Xi^{x_2,1})_*\, q_1(dx_1)\right]\, dq_{1}(dx_2)
\end{equation}
for all $0<x_1<x_2<1$. This reflects the intuitive requirement that
there should be no difference whether we first choose randomly
$x_1\in\,]0,1[$ and then $x_2\in\,]x_1,1[$ or the other way round,
first $x_2\in\,]0,1[$ and then $x_1\in\,]0,x_2[$.

\begin{lemma}
A family of measures $q_N, N\in\N$, with continuous densities
satisfies property (\ref{q-inv}) if and only if
\begin{equation}\label{qn}
q_N(dx_1,\ldots, dx_N)=C^N \frac{dx_1\ldots
dx_N}{x_1\cdot(x_2-x_1)\cdot\ldots\cdot(x_N-x_{N-1})\cdot (1-x_N)}
\end{equation}
for some constant $C\in\R_+$.
\end{lemma}

\begin{proof}
If $q_1(dx)=\rho(x)dx$ then (\ref{q1-inv}) is equivalent to
\begin{equation*}\label{q2-inv}
\rho(y)\cdot\rho\left(\frac{x}{y}\right)\cdot\frac1y=\rho(x)\cdot\rho\left(\frac{y-x}{1-x}\right)\cdot\frac1{1-x}
\end{equation*}
for all $0<x<y<1$. For continuous $\rho$ this implies that there
exists a constant $C\in\R_+$ such that $\rho(x)= \frac C{x(1-x)}$
 for all $0<x<1$. Iterated inserting this into
(\ref{q-inv}) yields the claim.
\end{proof}

Let us come back to our attempt to give a meaning to the heuristic
formula (\ref{formal-q}). Combining (\ref{q-approx}) with the choice
(\ref{qn}) of the measure $q_N$ finally yields
\begin{eqnarray} \nonumber\lefteqn{
\Q_0^\beta\left( g_{t_1}\in dx_1,\ldots, g_{t_N}\in
dx_N\right)}\\
&=&\frac1{Z_{\beta,N}}\prod_{i=1}^{N+1}(x_i-x_{i-1})^{\beta(t_i-t_{i_1})}\
\frac{dx_1\ldots dx_N}{x_1\cdot(x_2-x_1)\cdot\ldots\cdot (1-x_N)}
\end{eqnarray}
with appropriate normalization constants $Z_{\beta,N}$. Now the RHS
of this formula indeed turns out to define a consistent family of
probability measures. Hence, by Kolmogorov's extension theorem it
admits a projective limit $\Q_0^\beta$ on the space $\G_0$. The push
forward of this measure under the canonical identification
$\chi:\G_0\to\Pe_0, g\mapsto g_*\leb$  will be the {\em entropic
measure} $\Pp^\beta_0$ which we were looking for.

The details of the {\em rigorous construction} of this measure as
well as various properties of it will be presented in the following
sections.
\end{subsection}

\begin{subsection}{The Measures $\Q^\beta$ and $\mathbb{P}^\beta$ }
The basic object to be studied in this section is the probability
measure $\Q^\beta$ on the space $\G$.

\begin{proposition}
For each real number $\beta>0$ there exists a unique probability
measure $\Q^\beta$ on $\G$, called {\em Dirichlet process}, with the
property that for each $N\in\N$ and for each ordered family of
points $t_1,t_2,\ldots,t_N\in S^1$
\begin{equation}\label{Dir-S1}
\Q^\beta\left(g_{t_1}\in dx_1,\ldots,g_{t_N}\in dx_N\right)=
\frac{\Gamma(\beta)}{\prod_{i=1}^N\Gamma(\beta (t_{i+1}-t_i))}
\prod_{i=1}^N (x_{i+1}-x_i)^{\beta(t_{i+1}-t_i)-1} dx_1\ldots dx_N.
\end{equation}
\end{proposition}
The precise meaning of (\ref{Dir-S1}) is that for all bounded
measurable $u: (S^1)^N\to\R$
\begin{eqnarray*}\label{Dir-S1-prec}\lefteqn{\int_\G u\left(g_{t_1},\ldots,g_{t_N}\right)\, d\Q^\beta(g)}&&\\
&=&\frac{\Gamma(\beta)}{\prod_{i=1}^N\Gamma(\beta\cdot
|[t_i,t_{i+1}]|)}\int_{\Sigma_N} u(x_1,\ldots,x_N) \prod_{i=1}^N
|[x_i,x_{i+1}]|^{\beta\cdot |[t_i,t_{i+1}]|-1} dx_1\ldots dx_N.
\end{eqnarray*}
with $\Sigma_N=\left\{(x_1,\ldots,x_N)\in (S^1)^N: \ \sum_{i=1}^N
|[x_i,x_{i+1}]|=1\right\}$ and $x_{N+1}:=x_1$, $t_{N+1}:=t_1$.
In particular, with $N=1$ this means $\int_{\G}u(g_t)d\Q^\beta(g)=\int_{S^1}u(x)dx$ for each $t\in S^1$.

\begin{proof}
It suffices to prove that (\ref{Dir-S1}) defines a consistent family
of finite dimensional distributions. The existence of $\Q^\beta$ (as
a 'projective limit') then follows from Kolmogorov's extension
theorem. The required  consistency means that
\begin{eqnarray*}\lefteqn{\frac{\Gamma(\beta)}
{\prod_{i=1}^N\Gamma(\beta\cdot |[t_i,t_{i+1}]|)} \int_{\Sigma_N}
 \prod_{i=1}^N
|[x_i,x_{i+1}]|^{\beta\cdot |[t_i,t_{i+1}]|-1} u(x_1,\ldots,x_N)\, dx_1\ldots dx_N}\\
&=& \frac{\Gamma(\beta)}{\Gamma(\beta\cdot
|[t_1,t_{2}]|)\cdot\ldots\cdot\Gamma(\beta\cdot
|[t_{k-1},t_{k+1}]|)\cdot \ldots\cdot\Gamma(\beta\cdot|[t_N,t_{1}]|)
}\\
&\cdot& \int_{\Sigma_{N-1}} |[x_1,x_{2}]|^{\beta\cdot
|[t_1,t_{2}]|-1}\cdot\ldots\cdot |[x_{k-1},x_{k+1}]|^{\beta\cdot
|[t_{k-1},t_{k+1}]|-1}\cdot\ldots\cdot |[x_{N},x_{1}]|^{\beta\cdot
|[t_{N},t_{1}]|-1}\\
&&\qquad\qquad\cdot v(x_1,\ldots,x_{k-1},x_k\ldots,x_N)\,
 dx_1\ldots dx_{k-1}dx_{k+1}\ldots
dx_N
\end{eqnarray*}
whenever $u(x_1,\ldots,x_N)=v(x_1,\ldots,x_{k-1},x_k\ldots,x_N)$ for
all $(x_1,\ldots x_N)\in\Sigma_N$. The latter is an immediate
consequence of the well-known fact ({\it Euler's beta integral})
that
\begin{eqnarray*}\lefteqn{\int_{[x_{k-1},x_k+1]}|[x_{k-1},x_{k}]|^{\beta\cdot
|[t_{k-1},t_{k}]|-1}\cdot|[x_{k},x_{k+1}]|^{\beta\cdot
|[t_{k},t_{k+1}]|-1}\, dx_k}\\
& =& \frac{ \Gamma(\beta\cdot |[t_{k-1},t_{k}]|)\Gamma(\beta\cdot
|[t_{k},t_{k+1}]|)}{\Gamma(\beta\cdot |[t_{k-1},t_{k+1}]|)}
|[x_{k-1},x_{k+1}]|^{\beta\cdot
|[t_{k-1},t_{k+1}]|-1}.\end{eqnarray*}
\end{proof}

For $s\in S^1$ let $\hat\theta_s:\G\to\G, g\mapsto g\circ \theta_s$
be the isomorphism of $\G$ induced by the rotation $\theta_s: S^1\to
S^1, t\mapsto t+s$. Obviously, the measure $\Q^\beta$ on $\G$ is
invariant under each of the maps $\hat\theta_s$. Hence, $\Q^\beta$
induces a probability measure $\Q^\beta_1$ on the quotient spaces
$\G_1=\G/S^1$.

Recall the definition of the map $\chi:\G\to\Pe, g\mapsto g_*\leb$. Since
$(g\circ\theta_s)_*\leb=g_*\leb$ this canonically extends to a map
$\chi:\G_1\to\Pe$. (As mentioned before, the latter is even an
isometry.)

\begin{definition} The {\em entropic measure} $\mathbb{P}^\beta$ on $\Pe$ is
defined as the push forward of the Dirichlet process $\Q^\beta$ on
$\G$ (or equivalently, of the measure $\Q_1^\beta$ on $\G_1$) under
the map $\chi$. That is, for all bounded measurable $u:\Pe\to\R$
$$\int_\Pe u(\mu)\, d\mathbb{P}^\beta(\mu)=\int_\G u(g_*\leb)\,
d\Q^\beta(g).$$
\end{definition}

\end{subsection}

\begin{subsection}{The Measures $\Q_0^\beta$ and $\mathbb{P}_0^\beta$ }
The subspaces  $\{g\in\G: \,
g(0)=0\}$ and  $\{g\in\G_0: \,
g(0)=0\}$ can obviously be identified.
Conditioning the probability measure $\Q^\beta$ onto this event thus will
define a probability measure $\Q_0^\beta$ on $\G_0$.
However, we  prefer to give the direct construction of $\Q^\beta_0$.

\begin{proposition}
For each real number $\beta>0$ there exists a unique probability
measure $\Q_0^\beta$ on $\G_0$, called {\em Dirichlet process}, with
the property that for each $N\in\N$ and each  family
$0=t_0<t_1<t_2<\ldots<t_N<t_{N+1}=1$
\begin{equation}\label{Dir-I}
\Q^\beta_0\left(g_{t_1}\in dx_1,\ldots,g_{t_N}\in dx_N\right)=
\frac{\Gamma(\beta)}{\prod_i\Gamma(\beta\cdot (t_{i+1}-t_i))}
\prod_i (x_{i+1}-x_i)^{\beta\cdot (t_{i+1}-t_i)-1} dx_1\ldots dx_N.
\end{equation}

\end{proposition}
The precise meaning of (\ref{Dir-I}) is that for all bounded
measurable $u: \eI^N\to\R$
\begin{eqnarray*}\label{Dir-I-prec}\lefteqn{\int_{\G_0} u\left(g_{t_1},\ldots,g_{t_N}\right)\, d\Q_0^\beta(g)}&&\\
&=&\frac{\Gamma(\beta)}{\prod_{i=1}^N\Gamma(\beta\cdot
(t_{i+1}-t_i))}\int_{\Sigma_N} u(x_1,\ldots,x_N) \prod_{i=1}^N
(x_{i+1}-x_i)^{\beta\cdot(t_{i+1}-t_i)-1} dx_1\ldots dx_N.
\end{eqnarray*}
with $\Sigma_N=\left\{(x_1,\ldots,x_N)\in \eI^N: \
0<x_1<x_2\ldots<x_n<1\right\}$ and $x_{N+1}:=x_1$, $t_{N+1}:=t_1$.

\begin{remark}\rm According to these explicit formulae, it is easy to
calculate the moments of the Dirichlet process. For instance,
$$\mathbb{E}_0^\beta(g_t):=\int_{\G_0} g_t\, d\Q^\beta_0(g)=t$$
and
$$\mbox{Var}_0^\beta(g_t):=\int_{\G_0} (g_t-t)^2\, d\Q^\beta_0(g)=\frac1{1+\beta}t(1-t)$$
for all $\beta>0$ and all $t\in\eI$. \end{remark}

\begin{definition}
The {\em entropic measure} $\mathbb{P}_0^\beta$ on
$\Pe_0=\Pe(\eI)$ is defined as the push forward of the Dirichlet
process $\Q_0^\beta$ on $\G_0$ under the map $\chi$. That is, for
all bounded measurable $u:\Pe_0\to\R$
$$\int_{\Pe_0} u(\mu)\, d\mathbb{P}_0^\beta(\mu)=\int_{\G_0} u(g_*\leb)\,
d\Q_0^\beta(g).$$
\end{definition}

\begin{remark}\rm\label{q-q0}
{\bf (i)} According to the above construction $\Q_0^\beta(\, .
\,)=\Q^\beta(\ .  \   |g(0)=0)$ and
$$\int_{\G_0}u(g)\,d\Q^\beta_0(g)=\int_\G u(g-g(0))\, d\Q^\beta(g),$$
$$\int_\G u(g)\, d\Q^\beta(g)=\int_0^1\int_{\G_0}u(g+x)\,
d\Q^\beta_0(g)\, dx.$$

{\bf (ii)} Analogously, the entropic measures on the sphere and on
the are linked as follows
$$\int_\Pe u(\mu)\, d\Pp^\beta(\mu)=\int_0^1 \int_{\Pe_0} u((\theta_x)_*\mu)d\Pp_0^\beta(\mu)\, dx$$
or briefly
$$d\Pp^\beta=\int_0^1\left[ (\hat\theta_x)_*d\Pp^\beta_0\right]dx$$
where $\theta_x:S^1\to S^1, y\mapsto x+y$ and
$\hat\theta_x:\Pe\to\Pe: \mu\mapsto (\theta_x)_*\mu$. We would like
to emphasize, however, that $\Pp^\beta\not=\Pp^\beta_0$. For
instance, consider $u(\mu):=\int f\,d\mu$ for some  $f: S^1\to\R$
(which may be identified with $f: \eI\to\R$). Then
$$\int_{\Pe(S^1)} u(\mu)\, d\Pp^\beta(\mu)=\int_{S^1} f(x)\, dx$$
whereas
$$\int_{\Pe(\eI)} u(\mu)\, d\Pp_0^\beta(\mu)=\int_{\eI} f(x)\rho_\beta(x)\, dx$$
with $\rho_\beta(x)=\frac{\Gamma(\beta)}{\Gamma(\beta
t)\Gamma(\beta(1-t))}\int_0^1 x^{\beta t-1}(1-x)^{\beta(1-t)-1}\,
dt$.
\end{remark}

\medskip

According to the last remark, it suffices to study in detail one of
the four measures $\Q^\beta$, $\Q^\beta_0$,  $\Pp^\beta$, and
$\Pp_0^\beta$ . We will concentrate in the rest of this chapter on
the measure $\Q^\beta_0$ which seems to admit the most easy
interpretations.

\end{subsection}

\begin{subsection}{The Dirichlet Process as Normalized Gamma
Process}

We start recalling some basic facts about the real valued Gamma
processes. For $\alpha>0$ denote by  $G(\alpha)$ the absolutely
continuous probability measure on $\R_+$ with density   $\frac 1
{\Gamma(\alpha)} x^{\alpha -1} e^ {-x}$.

\begin{definition}A real valued Markov process
$(\gamma_t)_{t\geq 0}$ starting in zero is called standard Gamma
process  if its increments $\gamma_t - \gamma_s$ are independent and
distributed according to    $G(t-s)$ for $0\leq s< t$. Without loss
of generality we may assume that almost surely the function $t\to
\gamma_t$ is right continuous and nondecreasing.
\end{definition}

\smallskip

\smallskip
Alternatively the   Gamma-Process may be defined as the unique pure
jump Levy process with Levy measure $\Lambda (dx) =\one_{x>0}
\frac{e^{-x}}{x} dx $. The connection between pure jump Levy and
Poisson point processes gives rise to several other equivalent
representations of the Gamma process \cite{MR1207584, MR1746300}.
For instance, let $\Pi=\{p=(p_x,p_y)\in \R^2\}$ be the Poisson point
process on $\R_+\times \R_+$ with intensity measure $dx \times
\Lambda (dy)$ with $\Lambda$ as above, then a Gamma process is
obtained by
\begin{equation}
\gamma_t:= \sum\limits_{p\in \Pi: p_x\leq t} p_y. \label{poissonrep}
\end{equation}

For  $\beta >0$ the process $\gamma_{t\cdot \beta}$ is a Levy
process with Levy measure $\Lambda_\beta(dx) = \beta \cdot
\one_{x>0} \frac{e^{-x}}{x} dx$. Its increments are distributed
according to
\begin{equation}
P( \gamma_{\beta\cdot t } -\gamma_{\beta \cdot s} \in dx ) = \frac
1{\Gamma(\beta\cdot (t-s))} x^{\beta \cdot (t-s) -1} e^ {-x} dx.
\tag*{$\Box$}
\end{equation}

\begin{proposition}\label{dir=gamma}
For each $\beta >0$, the law of the process   $(\frac{\gamma_{t\cdot
\beta}}{\gamma_\beta})_{t \in \eI}$  is the Dirichlet process
$\Q^\beta_0$.
\end{proposition}

\begin{proof}This well-known fact is easily obtained from  Lukacs' characterization of the Gamma distribution \cite{MR2074696}.\end{proof}

\end{subsection}

\begin{subsection}{Support Properties}

\begin{proposition}{\bf (i)} For each $\beta>0$, the measure $\Q^\beta_0$ has full support on $\G_0$.

{\bf (ii)}
 $\Q^\beta_0$-almost surely the function
$t\mapsto g(t)$ is strictly increasing but increases only by jumps
(that is, the jumps heights add up to 1 and the jump locations are
dense in $\eI$).

{\bf (iii)} For each fixed $t_0\in\eI$,  $\Q^\beta_0$-almost surely
the function $t\mapsto g(t)$ is continuous at $t_0$.
\end{proposition}

\begin{proof} (i)
Let  $g \in \G \subset L^2(\eI,dx)$  and $\epsilon >0$ then we have
to show   $ \Q^\beta(B_\epsilon(g) ) >0 $ where
$B_\epsilon(g)=\{h\in\G_0: \, \|h-g\|_{L^2(\eI)}<\epsilon\}$. For
this choose finitely many points $t_i \in \eI$ together with $\delta
_ i
>0$ such that the set $ S:=
 \{ f \in \G\, \bigl | \, |f(t_i)-g(t_i)|\leq \delta_i \quad \forall i \}$
 is contained in $ B_\epsilon (g)$. Clearly, from  (\ref{Dir-I})  $\Q^\beta(S)>0$  which proves the claim.

(ii) (\ref{Dir-I}) implies that  $\Q^\beta_0$-almost surely
$g(s)<g(t)$ for each given pair $s<t$. Varying over all such
rational pairs $s<t$, it follows that a.e. $g$ is strictly
increasing on $\R_+$.

In terms of the probabilistic representation (\ref{dir=gamma}), it
is obvious that $g$ increases only by jumps.

(iii) This also follows easily from the representation as a
normalized gamma process (\ref{dir=gamma}).
\end{proof}

Restating the previous property (ii) in terms of the entropic
measure yields that $\mathbb{P}_0^\beta$-a.e. $\mu\in\Pe_0$ is
'Cantor like'. More precisely,
\begin{corollary}\label{support-P}
$\mathbb{P}_0^\beta$-almost surely the measure $\mu\in\Pe_0$ has no
absolutely continuous part and no discrete part. The topological
support of $\mu$ has Lebesgue measure 0. Moreover,
\begin{equation}\Ent(\mu)=+\infty.\end{equation}
\end{corollary}
\begin{proof} The assertion on the entropy of $\mu$ is an immediate consequence of
the statement on the support of $\mu$. The second claim follows
from the fact that the jump heights of $g$ add up to 1.
\end{proof}

In terms of the measure $\Q^\beta_0$, the last assertion of the
corollary states that $S(g)=+\infty$ for $\Q^\beta_0$-a.e.
$g\in\G_0$.
\end{subsection}

\begin{subsection}{Scaling and Invariance Properties}
The Dirichlet process $\Q^\beta_0$ on $\G_0$ has the following {\em
Markov property:} the distribution of $g|_{[s,t]}$ depends on
$g_{[0,1]\setminus [s,t]}$ only via  $g(s), g(t)$.

And the Dirichlet process $\Q^\beta_0$ on $\G_0$ has a remarkable
{\em self-similarity property}: if we restrict the functions $g$
onto a given interval $[s,t]$ and then linearly rescale their domain
and range in order to make them again elements of $\G_0$ then this
new process is distributed according to $\Q^{\beta'}_0$ with
$\beta'=\beta\cdot|t-s|$.
\begin{proposition}\label{ss-m}
For each $\beta>0$, and each $s,t\in\eI$, $s<t$
\begin{equation}\label{Markov}
\Q^\beta_0\left( g|_{[s,t]}\in . \ \big| \ g_{[0,1]\setminus
[s,t]}\right) = \Q^\beta_0\left( g|_{[s,t]}\in . \ \big| \ g(s),
g(t)\right)
\end{equation}
 and
\begin{equation}\label{self-similar}
(\Xi^{s,t})_*\Q^\beta_0=\Q_0^{\beta\cdot |t-s|} \end{equation}
 where $\Xi^{s,t}: \G_0\to \G_0$ with
$\Xi^{s,t}(g)(r)=\frac{g((1-r)s+rt)-g(s)}{g(t)-g(s)}$ for
$r\in\eI$.
\end{proposition}

\begin{proof} Both properties follow immediately from the
representation in Proposition \ref{Dir-I}.
\end{proof}

\begin{corollary} The probability measures $\Q^\beta_0$, $\beta>0$
on $\G_0$ are uniquely characterized by the self-similarity property
(\ref{self-similar}) and the distributions of $g_{1/2}$:
$$\Q^\beta_0(g_{1/2}\in
dx)=\frac{\Gamma(\beta)}{\Gamma(\beta/2)^2}\cdot
[x(1-x)]^{\beta/2-1}dx.$$
\end{corollary}

\begin{proposition}
{\bf (i)} \ For $\beta\to0$ the measures $\Q^\beta_0$ weakly
converge to a measure $\Q^0_0$ defined as the uniform distribution
on the set $\{1_{[t,1]}:\ t\in\,]0,1]\}\subset\G_0$.

Analogously,
 the measures $\Q^\beta$ weakly
converge for $\beta\to0$ to a measure $\Q^0$ defined as the uniform
distribution on the set of constant maps $\{t:\ t\in\
S^1\}\subset\G$.

{\bf (ii)} \ For $\beta\to\infty$ the measures $\Q^\beta_0$ (or
$\Q^\beta$) weakly converge to the Dirac mass $\delta_e$ on the
identity map $e$ of $\eI$ (or $S^1$, resp.).

\end{proposition}
\begin{proof}
(i) Since the space $\G_0$ (equipped with the $L^2$-topology) is
compact, so is $\Pe(\G_0)$ (equipped with the weak topology). Hence
the family $\Q^\beta_0$, $\beta>0$ is pre-compact. Let $\Q^0_0$
denote the limit of any converging subsequence of $\Q^\beta_0$ for
$\beta\to0$. According to the formula for the one-dimensional
distributions, for each $t\in\, ]0,1[$
\begin{eqnarray*}
\Q^\beta_0(g_{t}\in dx)&=& \frac{\Gamma(\beta)}{\Gamma(\beta
t)\Gamma(\beta(1-t))}\cdot x^{\beta t-1}(1-x)^{\beta(1-t)-1}dx \\
&\longrightarrow & (1-t)\delta_{\{ 0 \}}(dx)+t\delta_{\{ 1 \}}(dx)
\end{eqnarray*}
 as $\beta\to0$. Hence, $\Q^0_0$ is the uniform distribution on the set
$\{1_{[t,1]}:\ t\in\,]0,1]\}\subset\G_0$.

(ii) Similarly, $\Q^\beta_0(g_{t}\in dx)\to\delta_t(dx)$ as
$\beta\to\infty$. Hence, $\delta_e$ with $e:t\mapsto t$ will be the
unique accumulation point of $\Q^\beta_0$ for $\beta\to\infty$.
\end{proof}

Restating the previous results in terms of the entropic measures,
yields that the entropic measures $\Pp^\beta_0$ converge weakly to
 the uniform distribution $\Pp^0_0$ on the
set $\{ (1-t)\delta_{\{ 0 \}}+t\delta_{\{ 1 \}}: \
t\in[0,1]\}\subset\Pe_0$; and the measures $\Pp^\beta$ converge
weakly to the  uniform distribution $\Pp^0$  on the set $\{
\delta_{\{t \}}: \ t\in S^1\}\subset\Pe$ whereas for
$\beta\to\infty$ both,  $\Pp^\beta_0$ and $\Pp^\beta$, will converge
to $\delta_\leb$, the Dirac mass on the uniform distribution of
$\eI$ or $S^1$, resp.

The assertions of Proposition \ref{ss-m}  imply the following Markov
property and self-similarity property of the entropic measure.

\begin{proposition} For each each $x,y\in\eI$, $x<y$
$$\Pp^\beta_0\left(\mu|_{[x,y]}\in . \ \big|\mu|_{[0,1]\setminus [x,y]}\right)=
\Pp^\beta_0\left(\mu|_{[x,y]}\in . \ \big|\mu([x,y]\right)$$ and
$$
\Pp^\beta_0\left(\mu|_{[x,y]}\in . \ \big|\mu([x,y])=\alpha\right)
=\Pp_0^{\beta\cdot\alpha}\left(\mu_{x,y}\in . \,\right)$$ with
$\mu_{x,y}\in\Pe_0$ ('rescaling of $\mu|_{[x,y]}$') defined by
$\mu_{x,y}(A)=\frac1{\mu([x,y])}\mu(x+(y-x)\cdot A)$ for
$A\subset\eI$.
\end{proposition}

\end{subsection}

\begin{subsection}{Dirichlet Processes on General Measurable
Spaces}

Recall Ferguson's notion of a Dirichlet process on a general
measurable space $M$ with parameter measure $m$ on $M$. This is a
probability measure $\Q^m_{\Pe(M)}$ on $\Pe(M)$, uniquely defined by
the fact that for any finite measurable partition $M =\dot
\bigcup_{i=1}^{N+1} M_i$ and $\sigma_i := m(M_i)$.
\begin{eqnarray*}\lefteqn{
\Q^m_{\Pe(M)}\left(\mu:\ \mu(M_1)\in dx_1,\dots, \mu(M_{N})\in dx_{N}\right)} \\
 &=& \frac{\Gamma(m(M))}{\prod_{i=1}^{N+1} \Gamma(\sigma_i)} x_1^{\sigma_1-1}\cdots x_{N}^{\sigma_{N}-1}
\bigl(1-\sum_{i=1}^N x_i\bigr)^{\sigma_{N+1}-1} dx_1\cdots dx_{N},
\end{eqnarray*}

If a map $h:M\to M$ leaves the parameter measure $m$ invariant then
obviously the induced map $\hat h:\Pe(M)\to\Pe(M), \mu\mapsto
h_*\mu$ leaves the Dirichlet process $\Q^m_{\Pe(M)}$ invariant.

\medskip

In the particular case $M=\eI$ and $m=\beta\cdot\leb$, the Dirichlet
process $\Q^m_{\Pe(M)}$ can be obtained as push forward of the
measure $\Q^\beta_0$ (introduced before) under the isomorphism $\zeta:
\G_0\to\Pe(\eI)$ which assigns to each $g$ the induced
Lebesgue-Stieltjes measure $dg$ (the inverse $\zeta^{-1}$ assigns to
each probability measure its distribution function):

\begin{equation}
\Q^m_{\Pe(\eI)}=\zeta_*\Q^\beta_0.
\end{equation}

Note that the support properties of the measure  $\Q^m_{\Pe(\eI)}$
 are {completely different} from those of the
measure $\mathbb{P}_0^\beta$. In particular,
$\Q^m_{\Pe(\eI)}$-almost every $\mu\in\Pe(\eI)$ is discrete and has
full topological support, cf. Corollary \ref{support-P}. The invariance properties of $\Q^m_{\Pe(\eI)}$ under push forwards
by means of measure preserving transformations of $\eI$ seems to
have no intrinsic interpretation in terms of $\Q^\beta_0$.
\end{subsection}

\end{section}

\begin{section}{The Change of Variable Formula for the Dirichlet Process and for the Entropic Measure}
Our main result in this chapter will be a change of variable formula
for the Dirichlet process. To motivate this formula, let us first present an heuristic derivation based on the formal representation (\ref{formal-q}).

\begin{subsection}{Heuristic Approaches to Change of Variable Formulae}
Let us  have a look on the change of variable formula for the Wiener measure. On a formal level, it easily follows
from Feynman's heuristic interpretation
$$d{\mathbf{P}}^\beta(g)=\frac1Z e^{-\frac{\beta}{2} \int_0^1 g'(t)^2dt}\, d\mathbf{P}(g)$$
with the  (non-existing) 'uniform distribution' $\mathbf{P}$.
Assuming that the latter is 'translation invariant' (i.e. invariant
under additive changes of variables, -- at least in 'smooth'
directions $h$) we immediately obtain
\begin{eqnarray*}
d\mathbf{P}^\beta(\varh+g)&=& \frac1Z e^{-\frac{\beta}{2} \int_0^1
(\varh+g)'(t)^2dt}\, d\mathbf{P}(\varh+g)\\
&=& \frac1Z e^{-\frac{\beta}{2} \int_0^1 \varh'(t)^2dt -
\beta\int_0^1 \varh'(t)g'(t)dt}\cdot
e^{-\frac{\beta}{2} \int_0^1 g'(t)^2dt}\, d\mathbf{P}(g)\\
&=& e^{-\frac{\beta}{2} \int_0^1 \varh'(t)^2dt - \beta\int_0^1
\varh'(t)dg(t) } \, d\mathbf{P}^\beta(g).
\end{eqnarray*}
 If we interpret $\int_0^1\varh'(t)d g(t)$ as the
Ito integral of $\varh'$ with respect to the Brownian path $g$ then
this is indeed the  famous Cameron-Martin-Girsanov-Maruyama theorem.

\medskip

In the case of the entropic measure, the starting point for a
similar argumentation is the heuristic interpretation
$$d\Q_0^\beta(g)=\frac1Z e^{\beta \int_0^1\log g'(t)dt}\, d\Q_0(g),$$
again with a  (non-existing) 'uniform distribution'
$\Q_0$ on $\G_0$. The natural concept of 'change of variables', of course,
will be based on the semigroup structure of the underlying space
$\G_0$; that is, we will study transformations of $\G_0$ of the form  $g\mapsto \varh\circ g$ for some (smooth) element $\varh\in\G_0$.
It turns out that  $\Q_0$ should not be assumed to be
invariant under translations but merely quasi-invariant:
$$d\Q_0(\varh\circ g)=Y_\varh^0(g)\, d\Q_0(g)$$
with some density $Y_\varh$. This immediately implies the following change of variable formula for
$\Q_0^\beta$:
\begin{eqnarray*}
d\Q_0^\beta(\varh\circ g)&=& \frac1Z e^{\beta \int_0^1\log (\varh\circ
g)'(t)dt}\, d\Q_0(\varh\circ g)\\
&=& \frac1Z e^{\beta \int_0^1 \log\varh'(g(t))dt}\cdot
e^{\beta \int_0^1 \log g'(t)dt}  \cdot
Y^0_\varh(g)\, d\Q_0( g)\\
&=&   e^{ \beta\int_0^1 \log g'(t)dt } \cdot Y^0_\varh(g)\,
d\Q^\beta_0(g) .
\end{eqnarray*}
This is the heuristic derivation of the change of variables formula.
Its rigorous derivation (and the identification of the density $Y_\varh$) is the main result of this chapter.
\end{subsection}

\begin{subsection}{The Change of Variables Formula on the Sphere}
For $g,h \in \mathcal{G}$  with $h \in \mathcal{C}^2$ we put
\begin{equation}\label{Y0-def}
Y^0_{\varh}(g):=\prod_{a \in J_g}  \frac{ {\sqrt{\varh'
\left(g(a-)\right) \cdot \varh'\left(g(a+)\right)}}}{
 \frac{\delta
(\varh\circ g)}{\delta g}\left(a\right)} ,
\end{equation}
 where $J_g \subset S^1$ denotes the set of jump locations of $g$ and
\begin{equation*}\label{delta-def}
\frac{\delta (\varh\circ g)}{\delta
g}\left(a\right):=\frac{\varh\left(g(a+)\right) -
\varh\left(g(a-)\right)}{g(a+) - g(a-)} \ .
\end{equation*}
To simplify notation, here and in the sequel (if no ambiguity seems
possible), we write $y-x$ instead of $|[x,y]|$ to denote the length
of the positively oriented segment from $x$ to $y$ in $S^1$. We will
see below that the infinite product in the definition of $Y^0_h(g)$
converges for $\Q^\beta$-a.e. $g\in\G$. Moreover, for $\beta>0$ we
put
\begin{equation}
\label{Y-def}
X_{\varh}^\beta(g):=\exp\left(\beta\int^1_0 \log \varh'
\left(g(s)\right) ds\right),\qquad
Y^{\beta}_{\varh}(g):=X_{\varh}^{\beta}(g) \cdot Y^0_{\varh}(g).
\end{equation}

\begin{theorem}\label{CoV}
Each $\mathcal{C}^2$-diffeomorphism $h\in\G$ induces a bijective map $\tau_h:\G\to\G, \ g\mapsto h\circ g$
which leaves the measure $\Q^\beta$ quasi-invariant:
\begin{equation*}\label{CoV-formula}
d\Q^{\beta}(\varh \circ g) = Y^{\beta}_{\varh}(g) \ d\Q^{\beta}(g).
\end{equation*}
In other words,
the push forward of $\Q^\beta$ under the map  $\tau_h^{-1}=\tau_{h^{-1}}$ is absolutely continuous w.r.t.  $\Q^\beta$ with density
$Y^{\beta}_{\varh}$:
\begin{equation*}\label{CoV-formula2}
\frac{d (\tau_{h^{-1}})_*\Q^{\beta}(g)}{ d\Q^{\beta}(g)} = Y^{\beta}_{\varh}(g).
\end{equation*}
The function $Y^\beta_h$ is  bounded from above and below (away from 0) on $\G$.
\end{theorem}

By means of the canonical isometry $\chi:\G\to\Pe, \ g\mapsto g_*\leb$,
Theorem \ref{CoV} immediately implies

\begin{corollary}
For each $\C^2$-diffeomorphism $h\in\G$ the entropic measure $\Pp^\beta$ is quasi-invariant under
the transformation $\mu\mapsto h_*\mu$ of the space $\Pe$:
$$d\Pp^\beta(h_*\mu)=Y_h^\beta(\chi^{-1}(\mu))\ d\Pp^\beta(\mu).$$
The density $Y_h^\beta(\chi^{-1}(\mu))$ introduced in (\ref{Y-def}) can be expressed as follows
$$Y_h^\beta(\chi^{-1}(\mu))=
\exp\left[\beta\int_0^1 \log h'(s) \, \mu(ds)\right]\cdot
\prod_{I\in\Gaps(\mu)}\frac{\sqrt{h'(I_-)\cdot h'(I_+)}}{|h(I)|/|I|}
$$
where $\Gaps(\mu)$ denotes the set of segments $I= \,
]I_-,I_+[\,\subset S^1$ of maximal length with $\mu(I)=0$ and $|I|$
denotes the length of such a segment.
\end{corollary}
\end{subsection}

\begin{subsection}{The Change of Variables Formula on the Interval}
From the representation of $\Q^\beta$ as a product of $\Q^\beta_0$
and $\leb$ (see Remark \ref{q-q0}) and the change of variable
formulae for $\Q^\beta$ and $\leb$, one can deduce a change of
variable formula for $\Q^\beta_0$ similar to that of Theorem
\ref{CoV} but containing an additional factor $\frac1{h'(0)}$. In
this case, one has to restrict to  translations by means of
$\C^2$-diffeomorphisms $h\in\G$ with $h(0)=0$.

More generally, one might be interested in translations of $\G_0$ by
means of $\C^2$-diffeomorphisms $h\in\G_0$. In contrast to the
previous situation, it now may happen that $h'(0)\not= h'(1)$.

\medskip

 For
$g \in \mathcal{G}_0$   and   $\mathcal{C}^2$-ismorphism $h : \eI \to \eI$   we put
\begin{equation}\label{Y*-def}
Y^\beta_{\varh,0}(g):=X_\varh^\beta(g)\cdot Y_{\varh,0}(g)
\end{equation}
with
\[ Y_{\varh,0}(g)= \frac1{\sqrt {h'(0)\cdot h'(1)}} \cdot
Y^0_{\varh}(g)
\]
and  $X_\varh^\beta(g)$ and $  Y^0_\varh(g)$ defined
as before in (\ref{Y0-def}), (\ref{Y-def}). Note that here and in the sequel by a $\mathcal C^2$-isomorphism $h \in \mathcal G_0$ we understand an increasing homeomorphism $h: \eI \to \eI$ such that $h $ and $h^{-1}$ are bounded in $\mathcal C^2(\eI)$, which in particular implies $h'>0$.

\begin{theorem}\label{CoV-I}
Each translation $\tau_h:\G_0\to\G_0, \ g\mapsto h\circ g$ by means
of a $\mathcal{C}^2$-isomorphism $h\in\G_0$ leaves the measure
$\Q^\beta_0$ quasi-invariant:
\begin{equation*}\label{CoV-I-formula}
d\Q^{\beta}_0(\varh \circ g) =  Y^{\beta}_{\varh,0}(g) \
d\Q_0^{\beta}(g)
\end{equation*}
or, in other words,
\begin{equation*}\label{CoV-I-formula2} \frac{d
(\tau_{h^{-1}})_*\Q_0^{\beta}(g)}{ d\Q_0^{\beta}(g)} =
Y^{\beta}_{\varh,0}(g).
\end{equation*}
The function $Y^\beta_{h,0}$ is  bounded from above and below (away
from 0) on $\G_0$.
\end{theorem}

\begin{corollary}
For each $\C^2$-isomorphism $h\in\G_0$ the entropic measure
$\Pp^\beta_0$ is quasi-invariant under the transformation
$\mu\mapsto h_*\mu$ of the space $\Pe_0$:
$$\frac{d\Pp_0^\beta(h_*\mu)}{d\Pp_0^\beta(\mu)}=
\exp\left[\beta\int_0^1 \log h'(s) \,
\mu(ds)\right]\cdot\frac1{\sqrt {h'(0)\cdot h'(1)}} \cdot
\prod_{I\in\Gaps(\mu)}\frac{\sqrt{h'(I_-)\cdot h'(I_+)}}{|h(I)|/|I|}
$$
where $\Gaps(\mu)$ denotes the set of intervals
$I=\,]I_-,I_+[\,\subset \eI$ of maximal length with $\mu(I)=0$ and
$|I|$ denotes the length of such an interval.
\end{corollary}

\begin{remark}\rm Theorem \ref{CoV-I} seems to be unrelated to the quasi-invariance of the measure $\Q^m_{\Pe(\eI)}$ under the transformation $dg \to h \cdot dg/{\langle h,dg\rangle} $ shown in \cite{MR1902190}. Nor is it anyhow implied by the quasi-ivarariance formula for the general measure valued gamma process as in \cite{MR1853759} with respect to a similar  transformation. In our present case the latter would correspond to the mapping $d\gamma \to h\cdot d\gamma$ of the (measure valued) Gamma process $d\gamma$.
\end{remark}
\end{subsection}

\begin{subsection}{Proofs for the Sphere Case}
\begin{lemma}\label{X-ident}
For each $\mathcal{C}^2$-diffeomorphism  $h\in\G$
\begin{equation}\label{X-formula}
X_{\varh}^\beta(g)= \lim_{k \to \infty}
\prod^{k-1}_{i=0}\left[\frac{\varh \left(g(t_{i+1})\right) - \varh
\left(g(t_i)\right)}{g(t_{i+1}) - g(t_i)}\right]^{\beta(t_{i+1} -
t_i)}
\end{equation}
Here $t_i=\frac{i}{k}$ for $i=0, 1, \dots, k-1$  and $t_k=0$. Thus
$t_{i+1} - t_i:=|[t_i,t_{i+1}]|=\frac{1}{k}$ for all $i$.
\end{lemma}

\begin{proof} Without restriction, we may assume $\beta=1$.
According to Taylor's formula
\begin{equation*}
\varh \left(g(t_{i+1})\right)=\varh \left(g(t_i)\right) + \varh'
\left(g(t_i)\right) \cdot \left(g(t_{i+1}) - g(t_i)\right) +
\tfrac{1}{2} \varh''(\gamma_i) \cdot \left(g(t_{i+1}) -
g(t_i)\right)^2
\end{equation*}
 for some $\gamma_i \in \left[g(t_i), g(t_{i+1})\right]$. Hence,
\begin{align*}
&\lim_{k \to \infty} \prod^{k-1}_{i=0}\left[\frac{\varh
\left(g(t_{i+1})\right) - \varh \left(g(t_i)\right)}{g(t_{i+1}) -
g(t_i)}\right]^{t_{i+1} - t_i}=\\
&=\lim_{k \to \infty} \prod^{k-1}_{i=0} \left[\varh' \left(g(t_i)\right) + \tfrac{1}{2} \varh'' (\gamma_i) \cdot \left(g(t_{i+1}) - g(t_i)\right)\right]^{t_{i+1} - t_i}\\
&=\lim_{k \to \infty} \ \exp \left(\sum^{k-1}_{i=0}
\left\{\left[\log \varh' \left(g(t_i)\right) + \log \left(1 +
\tfrac{1}{2}\frac{\varh''(\gamma_i)}{\varh' \left(g(t_i)\right)}
\left(g
(t_{i+1}) - g(t_i)\right)\right)\right] \cdot \left(t_{i+1} - t_i\right)\right\}\right)\\
&\overset{(\star)}{=} \exp \left(\lim_{k \to \infty} \sum^{k-1}_{i=0} \left\{ \log \varh'\left(g(t_i)\right) \cdot \left(t_{i+1} - t_i\right)\right\}\right)\\
&=\exp \left(\int^1_0 \log \varh' \left(g(t)\right)
dt\right)=X_{\varh}^1(g).
\end{align*}
Here $(\star)$ follows from the fact that
\begin{align*}
1 + \tfrac{1}{2}\frac{\varh''(\gamma_i)}{\varh' \left(g(t_i)\right)}
\cdot \left(g(t_{i+1}) - g(t_i)\right)
&=\frac{\varh \left(g(t_{i+1})\right) - \varh \left(g(t_i)\right)}{g(t_{i+1}) - g(t_i)} \cdot \frac{1}{\varh' \left(g(t_i)\right)}\\
&=\varh' (\eta_i) \cdot \frac{1}{\varh' \left(g(t_i)\right)} \\
&\geq \ \varepsilon \ > \ 0
\end{align*}
for some $\eta_i \in \left[g(t_i), g(t_{i+1})\right]$ and some
$\varepsilon >0$, independent of $i$ and $k$. Thus
\begin{align*}
\sum^{k-1}_{i=0}&\left|\log \left[1 + \tfrac{1}{2}\frac{\varh''(\gamma_i)}{\varh' \left(g(t_i)\right)} \cdot \left(g(t_{i+1}) - g(t_i)\right)\right]\right| \cdot \left(t_{i+1} - t_i\right)\\
&\leq C_1 \cdot \sum^{k-1}_{i=0}\tfrac{1}{2}\left|\frac{\varh''(\gamma_i)}{\varh' \left(g(t_i)\right)}\right| \cdot \left(g(t_{i+1}) - g(t_i)\right) \cdot \left(t_{i+1} - t_i\right)\\
&\leq C_2 \cdot \sum^{k-1}_{i=0} \left(g(t_{i+1}) - g(t_i)\right) \cdot \left(t_{i+1} - t_i\right)\\
&\leq C_3 \cdot \tfrac{1}{k} \ .
\end{align*}
\end{proof}

\begin{lemma}\label{Y-ident}
For each $\mathcal{C}^3$-diffeomorphism $h\in\G$
\begin{equation}\label{Y-formula}
Y^0_{\varh}(g):=\lim_{k \to \infty} \prod^{k-1}_{i=0}\left[\varh'
\left(g(t_i)\right) \cdot \frac{g(t_{i+1}) - g(t_i)}{\varh
\left(g(t_{i+1})\right) - \varh \left(g(t_i)\right)}\right]
\end{equation}
where  $t_i=\frac{i}{k}$ for $i=0, 1, \dots, k-1$  and $t_k=0$.
\end{lemma}

\begin{proof}
Let $\varh$ and $g$ be given. Depending on some $\varepsilon > 0$
let us choose $l \in \N$ large enough (to be specified in the
sequel) and let $a_1, \dots, a_l$ denote the $l$ largest jumps of
$g$. Put $J^*_g=J_g \setminus \{a_1, \dots, a_l\}$ and for
simplicity $a_{l+1}:=a_1$. For $k$ very large (compared with $l$)
and $j=1, \dots, l$ let $k_j$ denote the index $i \in \{0, 1, \dots,
k-1\}$, for which $a_j \in \left[t_i, t_{i+1} \right[$. Then again
by Taylor's formula
\begin{align*}
\prod^{k_{j+1}-1}_{i=k_j+1} &\left[\varh'\left(g(t_i)\right) \cdot \frac{g(t_{i+1}) - g(t_i)}{\varh\left(g(t_{i+1})\right) - \varh\left(g(t_i)\right)}\right]^{-1}\\
&=\prod^{k_{j+1}-1}_{i=k_j+1} \left[1 + \tfrac{1}{2}\frac{\varh'' \left(g(t_i)\right)}{\varh'\left(g(t_i)\right)} \cdot \left(g(t_{i+1}) - g(t_i)\right) + \tfrac{1}{6}\frac{\varh''' (\eta_i)}{\varh'\left(g(t_i)\right)} \cdot \left(g(t_{i+1}) - g(t_i)\right)^2\right]\\
&\overset{(1a)}{\leq} \exp\left(\sum^{k_{j+1}-1}_{i=k_j+1}\log\left[1 + \left\{\tfrac{1}{2}\left(\log \varh'\right)'\left(g(t_i)\right) + \tfrac{\varepsilon}{l}\right\} \cdot \left(g(t_{i+1}) - g(t_i)\right)\right]\right)\\
&\overset{(1b)}{\leq} e^{\varepsilon/l} \cdot
\exp\left(\tfrac{1}{2}\sum^{k_{j+1}-1}_{i=k_j+1}\left(\log
\varh'\right)'\left(g(t_i)\right) \cdot \left(g(t_{i+1}) -
g(t_i)\right)\right),
\end{align*}
provided $l$ and $k$ are chosen so large that
\begin{equation*}
\left|g(t_{i+1}) - g(t_i)\right|\leq\frac{\varepsilon}{C_1 \cdot l}
\end{equation*}
for all $i \in \{0, \dots, k-1\} \setminus \{k_1, \dots, k_l\}$, where $C_1= \underset{x, y}{\sup}\frac{\left|\varh'''(x)\right|}{6 \cdot \varh'(y)}$.\\
On the other hand,
\begin{align*}
\lefteqn{\sqrt{\frac{\varh'\left(g(t_{k_{j+1}})\right)}{\varh'\left(g(t_{k_j+1})\right)}}\ = \
\exp\left(\int^{g(t_{k_{j+1}})}_{g(t_{k_j+1})} \left(\tfrac{1}{2}\log\varh'\right)'(s) \ ds\right)}\\
&=\exp\left(\sum^{k_{j+1}-1}_{i=k_j+1}\left[\left(\tfrac{1}{2}\log\varh'\right)'\left(g(t_i)\right) \cdot \left(g(t_{i+1}) - g(t_i)\right) + \left(\tfrac{1}{2}\log\varh'\right)''(\gamma_i) \cdot \tfrac{1}{2}\left(g(t_{i+1}) - g(t_i)\right)^2\right]\right)\\
&\overset{(2)}{\geq} e^{-\varepsilon/l} \cdot
\exp\left(\tfrac{1}{2}\sum^{k_{j+1}-1}_{i=k_j+1}\left(\log
\varh'\right)'\left(g(t_i)\right) \cdot \left(g(t_{i+1}) -
g(t_i)\right)\right),
\end{align*}
 provided $l$ and $k$ are chosen so large that
\begin{equation*}
\left|g(t_{i+1}) - g(t_i)\right|\leq\frac{\varepsilon}{C_2 \cdot l}
\end{equation*}
for all $i \in \{0, 1, \dots, k-1\} \setminus \{k_1, \dots, k_l\}$, where $C_2=\underset{x}{\sup}\left|\left(\tfrac{1}{2}\log \varh'\right)''(x)\right|$.\\

Therefore,
\begin{align*}
\prod_{i \in \{0, 1, \dots, k-1\} \setminus \{k_1, \dots, k_l\}}&\left[\varh'\left(g(t_i)\right) \cdot \frac{g(t_{i+1}) - g(t_i)}{\varh\left(g(t_{i+1})\right) - \varh\left(g(t_i)\right)}\right]^{-1}\\
&\leq e^{2\varepsilon} \cdot
\prod^l_{j=1}\sqrt{\frac{\varh'\left(g(t_{k_{j+1}})\right)}{\varh'\left(g(t_{k_j+1})\right)}}
= (\text{I}).
\end{align*}
In order to derive the corresponding lower estimate, we can proceed
as before in (1a) and (2) (replacing $\varepsilon$ by $-\varepsilon$
and $\leq$ by $\geq$ and vice versa). To proceed as in (1b) we have
to argue as follows
\begin{align*}
\exp&\left(\sum^{k_{j+1}-1}_{i=k_j+1}\log\left[1 + \left\{\left(\tfrac{1}{2}\log \varh'\right)'\left(g(t_i)\right) - \tfrac{\varepsilon}{l}\right\} \cdot \left(g(t_{i+1}) - g(t_i)\right)\right]\right)\\
&\overset{(1c)}{\geq}e^{-\varepsilon/l} \cdot
\exp\left(\sum^{k_{j+1}-1}_{i=k_j+1}(1-\varepsilon) \cdot
\left(\tfrac{1}{2}\log \varh'\right)'\left(g(t_i)\right) \cdot
\left(g(t_{i+1}) - g(t_i)\right)\right),
\end{align*}
provided $l$ and $k$ are chosen so large that
\begin{equation*}
\log\left(1 + C_3 \cdot \left(g(t_{i+1}) - g(t_i)\right)\right)\geq (1 - \varepsilon) \cdot C_3 \cdot \left(g(t_{i+1}) - g(t_i)\right)
\end{equation*}
for all $i \in \{0, 1, \dots, k-1\} \setminus \{k_1, \dots, k_l\}$, where $C_3=\underset{x}{\sup}\left|\left(\tfrac{1}{2}\log \varh'\right)'(x)\right|$.\\
Thus we obtain the following lower estimate
\begin{align*}
\prod_{i \in \{0, 1, \dots, k-1\} \setminus \{k_1, \dots, k_l\}}&\left[\varh'\left(g(t_i)\right) \cdot \frac{g(t_{i+1}) - g(t_i)}{\varh\left(g(t_{i+1})\right) - \varh\left(g(t_i)\right)}\right]^{-1}\\
&\geq e^{-2\varepsilon} \cdot \left[\prod^l_{j=1}\sqrt{\frac{\varh'\left(g(t_{k_{j+1}})\right)}{\varh'\left(g(t_{k_j+1})\right)}}\right]^{1 - \varepsilon}\\
&\geq e^{-2\varepsilon} \cdot C^{-\varepsilon/2}_3 \cdot
\prod^l_{j=1}\sqrt{\frac{\varh'\left(g(t_{k_{j+1}})\right)}{\varh'\left(g(t_{k_j+1})\right)}}
= (\text{II}),
\end{align*}
since
\begin{align*}
\left[\prod^l_{j=1}\sqrt{\frac{\varh'\left(g(t_{k_{j+1}})\right)}{\varh'\left(g(t_{k_j+1})\right)}}\right]^{\varepsilon}
&=\exp\left(\tfrac{\varepsilon}{2} \sum^l_{j=1} \left[\log \varh'\left(g(t_{k_{j+1}})\right) - \log \varh'\left(g(t_{k_j+1})\right)\right]\right)\\
&\leq \exp\left(\tfrac{\varepsilon}{2} \sum^l_{j=1} C_3 \cdot \left[g(t_{k_{j+1}}) - g(t_{k_j+1})\right]\right)\\
&\leq \exp\left(\tfrac{\varepsilon}{2}C_3\right),
\end{align*}
where $C_3=\underset{x}{\sup}\left|\left(\log \varh'\right)'(x)\right|$.\\
 Now for fixed $l$ as $k \to \infty$ the bound (I) converges to
\begin{equation*}
(\text{I}')=e^{2\varepsilon} \cdot
\prod^l_{j=1}\sqrt{\frac{\varh'\left(g(a_{j+1}-)\right)}{\varh'\left(g(a_j+)\right)}}
\end{equation*}
and the bound (II) to
\begin{equation*}
(\text{II}')=e^{-2\varepsilon} \cdot C^{-\varepsilon/2}_3 \cdot
\prod^l_{j=1}\sqrt{\frac{\varh'\left(g(a_{j+1}-)\right)}{\varh'\left(g(a_j+)\right)}}
\ .
\end{equation*}
Finally, it remains to consider
\begin{equation*}
\prod_{i \in \{k_1, \dots, k_l\}} \left[\varh'\left(g(t_i)\right)
\cdot \frac{g(t_{i+1}) - g(t_i)}{\varh\left(g(t_{i+1})\right) -
\varh\left(g(t_i)\right)}\right]^{-1} = (\text{III}).
\end{equation*}
Again for fixed $l$ and $k \to \infty$ this obviously converges to
\begin{equation*}
(\text{III}')=\prod^l_{j=1}\left[\frac{1}{\varh'\left(g(a_j-)\right)}
\cdot \frac{\delta (\varh\circ g)}{\delta g}\left(a_j\right)\right].
\end{equation*}
Putting together these estimates and letting $l \to \infty$, we
obtain the claim.
\end{proof}

\begin{lemma}\label{finiteness}
(i) For all $g,h\in\G$ with $h\in\C^2$ strictly increasing, the
infinite product in the definition of $Y^0_h(g)$ converges. There
exists a constant $C=C(\beta, \varh)$ such that $\forall g \in
\mathcal{G}$
\begin{equation*}
\tfrac{1}{C} \leq Y^{\beta}_{\varh}(g) \leq C.
\end{equation*}

(ii)
If $h_n\to h$ in $\C^2$ then $Y^0_{h_n}(g)\to Y^0_h(g)$.

(iii) Let $Y^0_{\varh, k}, X_{\varh, k}^\beta, Y^{\beta}_{\varh,k}$
denote the sequences used in Lemma \ref{X-ident} and \ref{Y-ident}
to approximate $Y^0_{\varh}, X_{\varh}^\beta, Y^{\beta}_{\varh}$.
Then there exists a constant $C=C(\beta, \varh)$ such that $\forall
g \in \mathcal{G}$, $\forall k \in \N$
\begin{equation*}
\tfrac{1}{C} \leq Y^{\beta}_{\varh, k}(g) \leq C.
\end{equation*}
\end{lemma}

\begin{proof}
(i) Put $C=\sup|(\log h')'|$. Given $g\in\G$  and $\epsilon>0$, we choose $k$ large enough such that
$\sum_{a\in J_g(k)}|g(a+)-g(a-)|\le\epsilon$ where $J_g(k)=J_g\setminus\{a_1,a_2,\ldots,a_k\}$ denotes the
'set of small jumps' of $g$. Here we
enumerate the jump locations $a_1,a_2,\ldots\in J_g$ according to the size of the respective jumps.
Then with suitable $\xi_a\in [g(a-), g(a+)]$
\begin{eqnarray*}
&&\sum_{a\in J_g(k)}
\left|\log \frac{\sqrt{h'(g(a-))}\sqrt{h'(g(a+))}}{\frac{\delta(h\circ g)}{\delta g}(a)}\right|\\
&\le&\sum_{a\in J_g(k)}\left|\frac12\log h'(g(a-))+\frac12\log h'(g(a-))-\log h'(\xi(a))\right|\\
&\le&\sum_{a\in J_g(k)}\left|C\cdot(g(a+)-g(a-))\right|=C\cdot\epsilon.
\end{eqnarray*}
Hence, the infinite sum
$$\sum_{a\in J_g}
\log \frac{\sqrt{h'(g(a-))}\sqrt{h'(g(a+))}}{\frac{\delta(h\circ g)}{\delta g}(a)}
=\lim_{k\to\infty}\sum_{a\in J_g(k)}
\log \frac{\sqrt{h'(g(a-))}\sqrt{h'(g(a+))}}{\frac{\delta(h\circ g)}{\delta g}(a)}
$$ is absolutely convergent
and thus also infinite product in the definition of $Y^0_h(g)$ converges.
The same arguments immediately yield
\begin{eqnarray}\label{absch-via-h''}
\left|\log Y^0_h(g)\right|
&\le&\sum_{a\in J_g}\left|\frac12\log h'(g(a-))+\frac12\log h'(g(a-))-\log h'(\xi(a))\right|\le C.
\end{eqnarray}

(ii) In order to prove the convergence $Y^0_{h_n}(g)\to Y^0_h(g)$, for given $g\in\G$ we split the product over all jumps into a finite product over the big jumps
and an infinite product over all small jumps. Obviously, the finite products will converge (for any choice of $k$)
$$\prod_{a\in\{a_1,\ldots,a_k\}}\frac{\sqrt{h_n'(g(a-))}\sqrt{h_n'(g(a+))}}{\frac{\delta(h_n\circ g)}{\delta g}(a)}
\longrightarrow\prod_{a\in\{a_1,\ldots,a_k\}}\frac{\sqrt{h'(g(a-))}\sqrt{h'(g(a+))}}{\frac{\delta(h\circ g)}{\delta g}(a)}$$
as $n\to\infty$ provided $h_n\to h$ in $\C^2$.
Now let  $C=\sup_n\sup_x|(\log h_n')'(x)|$ and choose $k$ as before.
Then uniformly in $n$
$$\left|\log\prod_{a\in J_g\setminus\{a_1,\ldots,a_k\}}\frac{\sqrt{h_n'(g(a-))}
\sqrt{h_n'(g(a+))}}{\frac{\delta(h_n\circ g)}{\delta g}(a)}\right|\le C\cdot\epsilon.$$
(iii)
Let $C_1=\underset{x}{\sup}|\varh'(x)|$ and
$C_2=\underset{x}{\sup}\left|\left(\log \varh'\right)'(x)\right|$.
Then for all $g$ and $k$:
\begin{equation*}
X_{\varh, k}(g)=\prod ^{k-1}_{i=0} \varh'(\eta_i)^{t_{i+1}-t_i} \leq
C_1
\end{equation*}
and
\begin{eqnarray*}
Y^0_{\varh, k}(g) \ = \ \prod^{k-1}_{i=0} \frac{\varh'\left(g(t_i)\right)}{\varh'(\gamma_i)}&
= & \exp\left[\sum^{k-1}_{i=0}\left(\log \varh'\right)'(\zeta_i) \cdot \left(g(t_i) - \gamma_i\right)\right]\\
&\leq& \exp \left[C_2 \cdot \sum^{k-1}_{i=0}\left|g(t_i) - \gamma_i\right|\right] \ \leq  \ \exp(C_2)
\end{eqnarray*}
(with suitable $\gamma_i, \eta_i \in \left[g(t_i), g(t_{i+1})\right]$ and $\zeta_i \in \left[g(t_i), \gamma_i\right]$). Analogously, the lower estimates follow.
\end{proof}

\begin{proof}[Proof of Theorem \ref{CoV}]
 In order to prove the equality of the two measures under
consideration, it suffices to prove that all of their finite
dimensional distributions coincide. That is, for each $m \in \N$,
each ordered family $t_1,\dots,t_m$ of points in $S^1$ and each
bounded continuous $u: (S^1)^m \longrightarrow \R$ one has to verify
that
\begin{align*}
\int_{\mathcal{G}} &u\left(\varh^{-1}\left(g(t_1)\right), \varh^{-1}\left(g(t_2)\right), \dots, \varh^{-1}\left(g(t_m)\right)\right) \ d\Q^{\beta}(g)\\
&=\int_{\mathcal{G}} u\left(g(t_1), g(t_2), \dots, g(t_m)\right)
\cdot Y^{\beta}_{\varh}(g) \ d\Q^{\beta}(g).
\end{align*}
Without restriction, we may restrict ourselves to equidistant
partitions, i.e. $t_i=\tfrac{i}{m}$ for $i=1, \dots, m$. Let us fix
$m \in \N$, $u$ and $\varh$. For simplicity, we first assume that
$h$ is $\C^3$. Then by Lemmas \ref{X-ident} - \ref{finiteness} and
Lebesgue's theorem
\begin{align*}
&\int_{\mathcal{G}} u\left(g\left(\tfrac{1}{m}\right), \dots, g\left(1\right)\right) \cdot Y^{\beta}_{\varh}(g) \ d\Q^{\beta}(g)\\
&=\int_{\mathcal{G}} u\left(g\left(\tfrac{1}{m}\right), \dots, g\left(1\right)\right) \cdot \lim_{k\to\infty}
Y^{\beta}_{\varh,k}(g) \ d\Q^{\beta}(g)\\
&=\lim_{k \to \infty}\int_{\mathcal{G}}
u\left(g\left(\tfrac{1}{m}\right), \dots, g\left(1\right)\right)
\cdot
\prod^{mk-1}_{i=0}\left[\varh'\left(g\left(\tfrac{i}{km}\right)\right)
\cdot \frac{g\left(\tfrac{i+1}{km}\right) -
g\left(\tfrac{i}{km}\right)}{\varh\left(g\left(\tfrac{i+1}{km}\right)\right)
- \varh\left(g\left(\tfrac{i}{km}\right)\right)}
\right]\\
& \phantom{=\lim_{k \to \infty}\int_{\mathcal{G}} u\left(g\left(\tfrac{1}{m}\right), \dots, g\left(1\right)\right) \cdot}\cdot \prod^{mk-1}_{i=0}\left[\frac{\varh\left(g\left(\tfrac{i+1}{km}\right)\right) - \varh\left(g\left(\tfrac{i}{km}\right)\right)}{g\left(\tfrac{i+1}{km}\right) - g\left(\tfrac{i}{km}\right)}\right]^{\tfrac{\beta}{km}} d\Q^{\beta}(g)\\
&=\lim_{k \to \infty}  \frac{\Gamma(\beta)}{ [\Gamma ( \beta /km)]^{km}  }\int_{S^{mk}_1} u(x_k, x_{2k}, \dots, x_{mk}) \prod^{mk-1}_{i=0} \varh'(x_i) \cdot \prod^{mk-1}_{i=0}\left[\varh(x_{i+1}) - \varh(x_i)\right]^{\tfrac{\beta}{km}-1} dx_1\dots dx_{mk}\\
&=\lim_{k \to \infty}  \frac{\Gamma(\beta)}{ [\Gamma ( \beta /km)]^{km}  }\int_{S^{mk}_1} u(x_k, x_{2k}, \dots, x_{mk}) \cdot \prod^{mk-1}_{i=0}\left[\varh(x_{i+1}) - \varh(x_i)\right]^{\tfrac{\beta}{km}-1} \ d\varh(x_1)\dots d\varh(x_{mk})\\
&=\lim_{k \to \infty}  \frac{\Gamma(\beta)}{ [\Gamma ( \beta /km)]^{km}  } \int_{S^{mk}_1} u\left(\varh^{-1}(y_k), \varh^{-1}(y_{2k}), \dots, \varh^{-1}(y_{mk})\right) \cdot \prod^{mk-1}_{i=0}\left[y_{i+1} - y_i\right]^{\tfrac{\beta}{km}-1} \ dy_1\dots dy_{mk}\\
&=\int_{\mathcal{G}}
u\left(\varh^{-1}\left(g\left(\tfrac{1}{m}\right)\right),
\varh^{-1}\left(g\left(\tfrac{2}{m}\right)\right), \dots,
\varh^{-1}\left(g\left(1\right)\right)\right) \ d\Q^{\beta}(g).
\end{align*}
Now we treat the general case $h\in\C^2$. We choose a sequence of $\C^3$-functions $h_n\in\G$ with $h_n\to h$ in $\C^2$.
Then
\begin{eqnarray*}
&&\int_{\mathcal{G}} u\left(h^{-1}\left(g(t_1)\right), h^{-1}\left(g(t_2)\right),
 \dots, h^{-1}\left(g(t_m)\right)\right) \ d\Q^{\beta}(g)\\
&=&\lim_{n\to\infty}
\int_{\mathcal{G}} u\left(h_n^{-1}\left(g(t_1)\right), h_n^{-1}\left(g(t_2)\right),
 \dots, h_n^{-1}\left(g(t_m)\right)\right) \ d\Q^{\beta}(g)\\
&=&\lim_{n\to\infty}
\int_{\mathcal{G}} u\left(g(t_1), g(t_2), \dots, g(t_m)\right)
\cdot Y^{\beta}_{h_n}(g) \ d\Q^{\beta}(g)\\
&=&
\int_{\mathcal{G}} u\left(g(t_1), g(t_2), \dots, g(t_m)\right)
\cdot Y^{\beta}_{\varh}(g) \ d\Q^{\beta}(g).
\end{eqnarray*}
For the last equality, we have used the dominated convergence
$Y^{\beta}_{h_n}(g)\to Y^{\beta}_{h}(g)$
(due to Lemma \ref{finiteness}).
\end{proof}

\end{subsection}

\begin{subsection}{Proof for the Interval Case}
The proof of Theorem \ref{CoV-I} uses completely analogous arguments as in  the previous section. To simplify notation, for $\varh  \in \mathcal{C}^1(\eI)$, $k \in \N$ let
$X_{\varh,k}, Y^0_{\varh,k}: \G_0 \to \R$
be  defined by
\[
 \quad X_{\varh,k}(g) := \prod^{k-1}_{i=0}\left[\frac{\varh
\left(g(t_{i+1})\right) - \varh \left(g(t_i)\right)}{g(t_{i+1}) -
g(t_i)}\right]^{t_{i+1} - t_i} \]
and
\[ Y^0_{\varh,k}(g):=   \left[ \frac{g(t_{1}) - g(t_0)}{\varh
\left(g(t_{1})\right) - \varh \left(g(t_0)\right)} \right]
\prod^{k-1}_{i=1}\left[\varh'
\left(g(t_i)\right) \cdot \frac{g(t_{i+1}) - g(t_i)}{\varh
\left(g(t_{i+1})\right) - \varh \left(g(t_i)\right)}\right]
\]
where $t_i=\frac{i}{k}$ with $i=0, 1, \dots, k$. Similar to the proof of theorem \ref{CoV} the measure $\Q^\beta_0$ satisfies the following finite dimensional quasi-invariance formula.

For any $u : \eI^{m-1} \to \R$, $m, l \in \N$ and  $\mathcal C^1$-isomorphism
$\varh : \eI \to \eI$
\begin{align*}
\int_{\mathcal{G}_0} &u\left(\varh^{-1}\left(g(t_1)\right),
\varh^{-1}\left(g(t_2)\right), \dots,
\varh^{-1}\left(g(t_{m-1})\right)\right) \ d\Q^\beta_0(g)
\\&
~~~~~~ =\int_{\mathcal{G}_0} u\left(g(t_1), g(t_2), \dots, g(t_{m-1})\right)
\cdot X^{\beta}_{ \varh, {l \cdot  m}}(g)\cdot  Y^{0}_{ \varh,{l \cdot  m}}(g)
\ d\Q^\beta_0(g),
\end{align*}
where $t_i = \frac i m $, $ i= 1, \cdots, m-1$. The passage to the limit for
letting first $l$ and then  $m$   to infinity is based on the following
assertions.

\begin{lemma}
  (i) For each $\mathcal{C}^2$-isomorphism  $h \in \G_0$   and $g \in \G_0$
\[
X_{\varh}(g)= \lim_{k \to \infty} X_{\varh,k}(g).
\]

  (ii) For each $\mathcal{C}^3$-isomorphism $h \in \G_0$   and $g \in \G_0$
\begin{align*}
\lim_{k \to \infty} Y^0_{\varh,k}(g) & =
   \prod _{ a \in J _g  }\frac{\sqrt {h'(g(a+))\cdot h'(g(a-))}}{\frac{\delta (h\circ g)}{\delta g}(a)}
\\ & \quad
\times
\frac 1 { \sqrt { h'(g(0))\cdot h'(g(1-))}}
\cdot
\left\{
\begin{array}{ll}
1 & \mbox{ if }  g(1-)=g(1) \\
\frac {h'(g(1-)) }{ \frac{\delta (h\circ g)}{\delta g}(1)} & \mbox{ else},
\end{array}
\right.
\end{align*}
where $J_g\subset ]0,1[$  is the set of jump locations of $g $ on   $]0,1[$. In particular,
\[
\lim_{k \to \infty} Y^0_{\varh,k}(g) =   Y_ {\varh,0}(g) \quad \mbox{ for } \Q^\beta_0\mbox{-a.e.} g.\]
(iii) For all $g \in \G_0$ and $\mathcal C^2$-isomorphism $\varh \in \G_0$, the
infinite product in the definition of $Y_{\varh,0}(g)$ converges. There
exists a constant $C=C(\beta, \varh)$ such that $\forall g \in
\mathcal{G}_0$
\begin{equation*}
\tfrac{1}{C} \leq Y^{\beta}_{\varh,0}(g) \leq C.
\end{equation*}
(iv)
If $h_n\to h$ in $\C^2(\eI, \eI)$ with $\varh$ as above, then $Y^{0,h_n}(g)\to Y_{0,h}(g)$.\\
(v)  For each $\mathcal C^3$-isomorphism $\varh   \in \G_0$ there
 exists a constant $C=C(\beta, \varh)$ such that $\forall
g \in \mathcal{G}$, $\forall k \in \N$
\begin{equation*}
\tfrac{1}{C} \leq X^{\beta}_{\varh,k}(g) \cdot Y^{0}_{\varh,k}(g) \leq C.
\end{equation*}
\end{lemma}

\begin{proof}
The proofs of (i) and (iii)-(iv) carry over from  their respective counterparts on the sphere, lemmas \ref{X-ident} and \ref{finiteness} above. We sketch the proof of statement (ii) which needs most modification. For  $\varepsilon > 0$
choose $l \in \mathbb N$ large enough  and let $a_2, \dots, a_{l-1}$ denote the $l-2$ largest jumps of
$g $ on $ ]0,1[$. For $k$ very large (compared with $l$) we may assume that $a_2, \dots, a_{l-2}  \in ]\frac 2 k , 1 -\frac 2 k [$. Put $a_1 := \frac 1 k $, $a_{l} := 1- \frac 1 k$.
For $j=1, \dots, l$ let $k_j$ denote the index $i \in \{ 1, \dots,
k-1\}$, for which $a_j \in \left[t_i, t_{i+1} \right[$. In particular,   $k_1 =1 $ and $k_l=k-1$. Then using the same arguments as in lemma \ref{Y-ident} one obtains, for $k$ and $l$ sufficiently large, the two sided bounds
\begin{align*}
 (\text{I}) & = e^{2\varepsilon} \cdot
\prod^{l-1}_{j=1} \sqrt{\frac{\varh'\left(g(t_{k_{j+1}})\right)}{\varh'\left(g(t_{k_j+1})\right)}}
\\
  &~~~~~~~~ \geq
  \prod_{i \in \{ 1, \dots, k-1\} \setminus \{k_1, \dots, k_l\}}\left[\varh'\left(g(t_i)\right) \cdot \frac{g(t_{i+1}) - g(t_i)}{\varh\left(g(t_{i+1})\right) - \varh\left(g(t_i)\right)}\right]^{-1}\\
 & ~~~~~~~~ ~~~~~~~~   \geq  e^{-2\varepsilon} \cdot C^{-\varepsilon/2}_3 \cdot
\prod^{l-1}_{j=1}\sqrt{\frac{\varh'\left(g(t_{k_{j+1}})\right)}{\varh'\left(g(t_{k_j+1})\right)}}
= (\text{II})
\end{align*}
For fixed $l$ and $k \to \infty$ the bounds (I) and (II) converge to
\begin{equation*}
(\text{I}')=e^{2\varepsilon} \sqrt {\frac {h'(g(a_2-))}{h'(g(0))}}\cdot
\prod^{l-2}_{j=2}\sqrt{\frac{\varh'\left(g(a_{j+1}-)\right)}{\varh'\left(g(a_j+)\right)}} \cdot \sqrt {\frac {h'(g(1-))}{h'(g(a_{l-1}+))}}
\end{equation*}
and
\begin{equation*}
(\text{II}')=e^{-2\varepsilon} \cdot C^{-\varepsilon/2}_3 \cdot
 \sqrt {\frac {h'(g(a_2-))}{h'(g(0))}}\cdot
\prod^{l-2}_{j=2}\sqrt{\frac{\varh'\left(g(a_{j+1}-)\right)}{\varh'\left(g(a_j+)\right)}} \cdot \sqrt {\frac {h'(g(1-))}{h'(g(a_{l-1}+))}}.\end{equation*}
It remains to consider the three remaining terms
\begin{equation*}
 (\text{III})= \prod_{i \in \{k_2, \dots, k_{l-1}\}} \left[\varh'\left(g(t_i)\right)
\cdot \frac{g(t_{i+1}) - g(t_i)}{\varh\left(g(t_{i+1})\right) -
\varh\left(g(t_i)\right)}\right]^{-1},
\end{equation*}
which   for fixed $l$ and $k \to \infty$    converges to
\begin{equation*}
(\text{III}')=\prod^{l-1}_{j=2}\left[\frac{1}{\varh'\left(g(a_j-)\right)}
\cdot \frac{\delta (\varh\circ g)}{\delta g}\left(a_j\right)\right],
\end{equation*}
\[ (\text{IV})=  \left[\frac{g(\frac 1 k ) - g(0)}{\varh\left(g( \frac 1 k)\right) -
\varh\left(g(0)\right)}\right]^{-1}
 \cdot \left[\varh'\left(g(\frac 1 k)\right)
\cdot \frac{g({\frac 2 k }) - g(\frac 1 k)}{\varh\left(g( {\frac 2 k})\right) -
\varh\left(g(\frac 1 k)\right)}\right]^{-1},
\]
converging  by right continuity of $g$ to
\[ (\text{IV}') = h'(g(0))\]
and
\[ (\text{V})=   \left[\varh'\left(g(\frac {k-1} k)\right)
\cdot \frac{g({1 }) - g(\frac {k-1} k)}{\varh\left(g( {1})\right) -
\varh\left(g(\frac {k-1} k)\right)}\right]^{-1},
\]
which tends, also for $k \to \infty$, to
\[ (\text{V}')
= \left\{
\begin{array}{ll}
1 & \mbox{ if }  g \mbox{ continuous in } 1 \\
 \frac{\delta (h\circ g)}{\delta g}(1) \frac 1 {\varh '(g(1-))}
& \mbox{ else. }
\end{array}
\right.
\]
Combining  these estimates and letting $l \to \infty$, we
obtain the first claim. The second claim in statement (ii) follows from the fact that $g$ is continuous in $t=1$ $\mathbb Q^\beta_0$-almost surely.
\end{proof}

\end{subsection}

\end{section}

\begin{section}{The Integration by Parts Formula}
In order to construct Dirichlet forms and  Markov processes on $\G$,
we will consider it as an infinite dimensional manifold. For each
$g\in\G$, the tangent space $T_g\G$ will be an appropriate
completion of the space $\C^\infty(S^1,\R)$. The whole construction
will strongly depend on the choice of the norm on the tangent spaces
$T_g\G$. Basically, we will encounter two important cases:
\begin{itemize}
\item in Chapter 6 we will study the case $T_g\G=H^s(S^1,\leb)$ for
some $s>1/2$, independent of $g$; this approach is closely related
to the construction of stochastic processes on the diffeomorphism
group of $S^1$ and  Malliavin's Brownian motion on the homeomorphism group on
$S^1$, cf. \cite{MR1713340}.

\item in Chapters 7-9 we will assume $T_g\G=L^2(S^1, g_*\leb)$; in terms of the dynamics on the space $\Pe(S^1)$
of probability measures, this
will lead to a Dirichlet form and a stochastic process associated with the Wasserstein gradient
and with intrinsic metric given by the Wasserstein distance.
\end{itemize}
In this chapter, we develop the basic tools for the differential
calculus on $\G$. The main result will be an integration by parts
formula. These results will be independent of the choice of the norm
on the tangent space.

\begin{subsection}{The Drift Term}

For each $\varphi\in\C^\infty(S^1,\R)$, the {\em flow} generated by
$\varphi$ is the map $e_\varphi:\R\times S^1\to S^1$ where for each
$x\in S^1$ the function $e_\varphi(.,x): \R\to S^1, t\mapsto
e_\varphi(t,x)$ denotes the unique solution to the ODE
\begin{equation}\label{flow}
\frac{dx_t}{dt}=\varphi(x_t)
\end{equation}
with initial condition $x_0=x$. Since
$e_\varphi(t,x)=e_{t\varphi}(1,x)$ for all $\varphi,t,x$ under
consideration, we may simplify notation and write $e_{t\varphi}(x)$
instead of $e_\varphi(t,x)$.

Obviously, for each $\varphi\in\C^\infty(S^1,\R)$ the family
$e_{t\varphi}$, $t\in\R$ is a group of orientation preserving,
$\C^\infty$-diffeomorphism of $S^1$. (In particular,  $e_0$ is the
identity map $e$ on $S^1$, $e_{t\varphi}\circ
e_{s\varphi}=e_{(t+s)\varphi}$ for all $s,t\in\R$ and
$(e_\varphi)^{-1}=e_{-\varphi}$.)

Since $\frac{\partial}{\partial t}
e_{t\varphi}(x)|_{t=0}=\varphi(x)$  we obtain as a linearization for
small $t$
\begin{equation}\label{linear}
e_{t \varphi}(x)\approx x+t \varphi(x).
\end{equation}
More precisely,
$$|e_{t \varphi}(x)-(x+t \varphi(x))|\le C\cdot t^2$$
as well as
$$|\frac{\partial}{\partial x}e_{t \varphi}(x)-(1+t \frac{\partial}{\partial x}\varphi(x))|\le C\cdot t^2$$
uniformly in $x$ and $|t|\le1$.

\medskip

For $\varphi\in\C^\infty(S^1,\R)$ and $\beta>0$ we define functions
$V^\beta_\varphi:\G\to\R$ by
$$V^\beta_\varphi(g):=V^0_\varphi(g)+\beta\int_{S^1}\varphi'(g(x))dx$$
where
\begin{equation}\label{drift-0}
V^0_\varphi(g):=\sum_{a\in J_g}\left[
\frac{\varphi'(g(a+))+\varphi'(g(a-))}2-\frac{\varphi(g(a+))-\varphi(g(a-))}{g(a+)-g(a-)}\right].
\end{equation}

\begin{lemma}\label{drift-lemma}
{\bf (i)} \ The sum in (\ref{drift-0}) is absolutely convergent.
More precisely,
$$|V^0_\varphi(g)|\le \sum_{a\in J_g}\left|
\frac{\varphi'(g(a+))+\varphi'(g(a-))}2-\frac{\varphi(g(a+))-\varphi(g(a-))}{g(a+)-g(a-)}
\right|\le \frac12\int_{S^1}|\varphi''(x)|dx$$ and
$$|V^\beta_\varphi(g)|\le (1/2+\beta)\cdot\int_{S^1}|\varphi''(x)|dx.$$

{\bf (ii)} \ For each $\beta\ge0$
\begin{equation}\label{drift-ident}
V^\beta_\varphi(g)=\left.\frac{\partial}{\partial
t}Y^\beta_{e_{t\varphi}}(g)\right|_{t=0}
=\left.\frac{\partial}{\partial
t}Y^\beta_{e+{t\varphi}}(g)\right|_{t=0}.
\end{equation}
\end{lemma}

\begin{proof}
(i) \  According to Taylor's formula, for each $a\in J_g$
$$\frac{\varphi'(g(a+))+\varphi'(g(a-))}2-\frac{\delta(\varphi\circ
g)}{\delta g}(a)=\frac1{2(g(a+)-g(a-))}
\int_{g(a-)}^{g(a+)}\int_{g(a-)}^{g(a+)} \mbox{\rm
sgn}(y-x)\cdot\varphi''(y)dydx.$$ Hence,
\begin{eqnarray*}
\lefteqn{\sum_{a\in J_g}\left|
\frac{\varphi'(g(a+))+\varphi'(g(a-))}2-\frac{\delta(\varphi\circ
g)}{\delta g}(a)\right|}\\
&\le&\frac12\sum_{a\in J_g}\left|\frac1{(g(a+)-g(a-))}
\int_{g(a-)}^{g(a+)}\int_{g(a-)}^{g(a+)} \mbox{\rm
sgn}(y-x)\cdot\varphi''(y)dydx\right|\\
&\le&\frac12\sum_{a\in J_g}\int_{g(a-)}^{g(a+)}|\varphi''(y)|dy\ =\
\frac12\int_{S^1}|\varphi''(y)|dy.
\end{eqnarray*}
Finally,
$$|\int_{S^1}\varphi'(g(x))dx|\le \sup_{y\in
S^1}|\varphi'(y)|\le\int_{S^1}|\varphi''(y)|dy.$$

(ii) Let us first consider the case $\beta=0$.
\begin{eqnarray*}
\left.\frac{\partial}{\partial t}\log
Y^0_{e_{t\varphi}}(g)\right|_{t=0} &=&
\left.\frac{\partial}{\partial t} \sum_{a\in
J_g}\left[\frac12\log(\frac{\partial}{\partial x}
e_{t\varphi})(g(a+))+ \frac12\log(\frac{\partial}{\partial x}
e_{t\varphi})(g(a-))- \log\frac{\delta
(e_{t\varphi}\circ g)}{\delta g}(a) \right]\right|_{t=0}\\
&{=}& \sum_{a\in J_g} \left.\frac{\partial}{\partial t}
\left[\frac12\log(\frac{\partial}{\partial x} e_{t\varphi})(g(a+))+
\frac12\log(\frac{\partial}{\partial x} e_{t\varphi})(g(a-))-
\log\frac{\delta (e_{t\varphi}\circ g)}{\delta g}(a)
\right]\right|_{t=0}.
\end{eqnarray*}
In order to justify that we may interchange differentiation and
summation, we decompose (as we did several times before) the
infinite sum over all jumps in $J_g$ into a finite sum over big
jumps $a_1,\ldots,a_k$ and an infinite sum over small jumps in
$J_g(k)=J_g\setminus\{a_1,\ldots,a_k\}$. Of course, the finite sum
will make no problem. We are going to prove that the contribution of
the small jumps is arbitrarily small. Recall from Lemma
\ref{finiteness} that
$$\sum_{a\in
J_g(k)}\left[\frac12\log(\frac{\partial}{\partial x}
e_{t\varphi})(g(a+))+ \frac12\log(\frac{\partial}{\partial x}
e_{t\varphi})(g(a-))- \log\frac{\delta (e_{t\varphi}\circ g)}{\delta
g}(a) \right] \le C_t\cdot\sum_{a\in
J_g(k)}\left[g(a+)-g(a-)\right]$$ where
$C_t:=\sup_x\left|\frac{\partial}{\partial
x}\log(\frac{\partial}{\partial x} e_{t\varphi})(x)\right|$. Now
$C_t\le C\cdot |t|$ for all $|t|\le1$ and an appropriate constant
$C$. Thus for any given $\epsilon>0$
$$\left|\frac{\partial}{\partial t} \sum_{a\in
J_g(k)}\left[\frac12\log(\frac{\partial}{\partial x}
e_{t\varphi})(g(a+))+ \frac12\log(\frac{\partial}{\partial x}
e_{t\varphi})(g(a-))- \log\frac{\delta (e_{t\varphi}\circ g)}{\delta
g}(a) \right]\right|_{t=0} \le\epsilon$$ provided $k$ is chosen
large enough (i.e. such that $C\cdot \sum_{a\in
J_g(k)}|g(a+)-g(a-)|\le\epsilon$). This justifies the above
interchange of differentiation and summation.

Now for each $x\in S^1$
$$\left.\frac{\partial}{\partial t}\left(\log\frac{\partial}{\partial x} e_{t\varphi}(x)\right)\right|_{t=0}=\varphi'(x)$$
since the linearization of $e_{t\varphi}$ for small $t$ yields
$$e_{t\varphi}(x)\approx x+t\varphi(x), \quad
\frac{\partial}{\partial x} e_{t\varphi}(x)\approx 1+t\varphi'(x).$$
Similarly, for small $t$ we obtain
$$\frac{\delta
(e_{t\varphi}\circ g)}{\delta g}(a)\approx 1+t\cdot \frac{\delta
(\varphi\circ g)}{\delta g}(a)$$ and thus
$$\left.\frac{\partial}{\partial t}\frac{\delta
(e_{t\varphi}\circ g)}{\delta g}(a)\right|_{t=0}=\frac{\delta
(\varphi\circ g)}{\delta g}(a).$$ Therefore,
$$\left.\frac{\partial}{\partial t}\log
Y^0_{e_{t\varphi}}(g)\right|_{t=0}=V^0_\varphi(g).$$ On the other
hand, obviously
$$\left.\frac{\partial}{\partial t}\log
Y^0_{e_{t\varphi}}(g)\right|_{t=0}=\left.\frac{\partial}{\partial t}
Y^0_{e_{t\varphi}}(g)\right|_{t=0}$$ since $Y^0_{e_{0}}(g)=1$.

Finally, we have to consider the derivative of $X_{e_{t\varphi}}$.
Based on the previous arguments and using the fact that
$\frac{\partial}{\partial t} \log\left(\frac{\partial}{\partial x}
e_{t\varphi}\right)(x)$ is uniformly bounded in $t\in[-1,1]$ and
$x\in S^1$ we immediately see
\begin{eqnarray*}
\left.\frac{\partial}{\partial t}\log
X_{e_{t\varphi}}(g)\right|_{t=0} &=& \left.\frac{\partial}{\partial
t}
\int_{S^1}\log\left(\frac{\partial}{\partial x} e_{t\varphi}\right)(g(y))dy\right|_{t=0}\\
&{=}& \int_{S^1}\left.\frac{\partial}{\partial t}
\log\left(\frac{\partial}{\partial x}
e_{t\varphi}\right)\right|_{t=0}(g(y))\, dy\ = \
\int_{S^1}\varphi'(g(y))dy.
\end{eqnarray*}
Again $X_{e_0}(g)=1$. Therefore,
$$\left.\frac{\partial}{\partial t}
\left[X_{e_{t\varphi}}\right]^\beta(g)\right|_{t=0}=\beta\cdot\int_{S^1}\varphi'(g(y))dy$$
and thus
$$\left.\frac{\partial}{\partial
t}Y^\beta_{e_{t\varphi}}(g)\right|_{t=0}=V^\beta_\varphi(g).$$ this
proves the first identity in (\ref{drift-ident}). The proof of the
second one $V^\beta_\varphi(g)=\left.\frac{\partial}{\partial
t}Y^\beta_{e+{t\varphi}}(g)\right|_{t=0}$ is similar (even slightly
easier).
\end{proof}

\end{subsection}

\begin{subsection}{Directional Derivatives}

For  functions
 $u:\G\to\R$ we will define the {\em directional derivative} along $\varphi\in \C^\infty(S^1,\R)$ by
\begin{equation}\label{dir-der-def}
D_\varphi u(g):=\lim_{t\to0}\frac1t\left[u(e_{t\varphi}\circ
g)-u(g)\right]
\end{equation}
provided this limit exists. In particular, this will be the case for
the following  'cylinder functions'.

\begin{definition}
 We
say that $u:\G\to\R$ belongs to the class $\Syl^k(\G)$ if it can be
written as
\begin{equation}\label{syl-def}
u(g)=U(g(x_1),\ldots, g(x_m))
\end{equation}
for some $m\in\N$, some $x_1,\ldots,x_m\in S^1$ and some
$\C^k$-function $U: (S^1)^m\to\R$.
\end{definition}

It should be mentioned that functions $u\in\Syl^k(\G)$ are in
general not continuous on $\G$.

\begin{lemma}\label{Lie-deriv}
The directional derivative exists for all $u\in\Syl^1(\G)$. In
particular, for $u$ as above
\begin{eqnarray*}
D_\varphi u(g)&=&\lim_{t\to0}\frac1t\left[u(g+t\cdot\varphi\circ g)-u(g)\right]\\
&=&\sum_{i=1}^m {\partial_i}U(g(x_1),\ldots,g(x_m))\cdot
\varphi(g(x_i))
\end{eqnarray*}  with $\partial_iU:=\frac{\partial}{\partial y_i}U$.
Moreover, $D_\varphi: \Syl^k(\G)\to\Syl^{k-1}(\G)$ for all
$k\in\N\cup\{\infty\}$ and
$$\|D_\varphi u\|_{L^2(\Q^\beta)}\le \sqrt{m}\cdot\|\nabla U\|_\infty\cdot
\|\varphi\|_{L^2(S^1)}.$$
\end{lemma}

\begin{proof} The first claim follows from
\begin{eqnarray*}
D_\varphi u(g)&=&\left.\frac{\partial}{\partial t}
U(e_{t\varphi}(g(x_1)),\ldots,e_{t\varphi}(g(x_m)))\right|_{t=0}\\
&=&\sum_{i=1}^m
\left.{\partial_i}U(e_{t\varphi}(g(x_1)),\ldots,e_{t\varphi}(g(x_m)))\cdot
\frac{\partial}{\partial t}e_{t \varphi}(g(x_i))\right|_{t=0}\\
&=&\sum_{i=1}^m {\partial_i}U(g(x_1),\ldots,g(x_m))\cdot
\varphi(g(x_i))\\
&=&\left.\frac{\partial}{\partial t}
U(g(x_1)+t\varphi(g(x_1)),\ldots, g(x_m)+t\varphi(g(x_m)))\right|_{t=0}\\
&=&\lim_{t\to0}\frac1t\left[u(g+t\cdot\varphi\circ g)-u(g)\right].
\end{eqnarray*}
For the second claim,
\begin{eqnarray*}
\|D_\varphi u\|^2_{L^2(\Q^\beta)}&=& \int_\G \left( \sum_{i=1}^m
{\partial_i}U(g(x_1),\ldots,g(x_m))\cdot
\varphi(g(x_i))\right)^2\, d\Q^\beta(g)\\
&\le& \int_\G  \left(\sum_{i=1}^m
({\partial_i}U)^2(g(x_1),\ldots,g(x_m))\cdot
\sum_{i=1}^m\varphi^2(g(x_i))\right)\, d\Q^\beta(g)\\
&\le&  \|\nabla U\|^2_\infty\cdot
\sum_{i=1}^m\int_\G\varphi^2(g(x_i))\, d\Q^\beta(g)\\
&=& {m}\cdot\|\nabla U\|^2_\infty\cdot \int_{S^1}\varphi^2(y)\, dy.
\end{eqnarray*}
\end{proof}

\end{subsection}

\begin{subsection}{Integration by Parts Formula on $\Pe(S^1)$}
For $\varphi\in \C^\infty(S^1,\R)$ let $D^*_\varphi$ denote the
 operator  in $L^2(\G,\Q^\beta)$ adjoint to $D_\varphi$ with domain
$\Syl^1(\G)$.

\begin{proposition}\label{adjoint}
$\Dom(D^*_\varphi)\supset \Syl^1(\G)$ and for all $u\in\Syl^1(\G)$
\begin{equation}D^*_\varphi u=-D_\varphi u - V^\beta_\varphi\cdot u.\end{equation}
\end{proposition}

\begin{proof}
Let $u,v\in\Syl^1(\G)$. Then
\begin{eqnarray*}
\int D_\varphi u\cdot v \, d\Q^\beta&=& \lim_{t\to0}
\frac1t\int  \left[u(e_{t\varphi}\circ g)-u(g)\right]\cdot v(g) \, d\Q^\beta(g)\\
&=& \lim_{t\to0} \frac1t\int  \left[u(g)\cdot v(e_{-t\varphi}\circ
g)\cdot Y^\beta_{e_{-t\varphi}}-u(g)\cdot v(g)\right]
 \, d\Q^\beta(g)\\
&=&\lim_{t\to0}
\frac1t\int  u(g)\cdot \left[v(e_{-t\varphi}\circ g)-v(g)\right] \, d\Q^\beta(g)\\
&&+\lim_{t\to0}
\frac1t\int u(g)\cdot v(g)\cdot\left[Y^\beta_{e_{-t\varphi}}-1\right]\, d\Q^\beta(g)\\
&&+\lim_{t\to0}
\frac1t\int u(g)\cdot \left[v(e_{-t\varphi}\circ g)-v(g)\right]\cdot\left[Y^\beta_{e_{-t\varphi}}-1\right]\, d\Q^\beta(g)\\
&=&-\int u\cdot D_\varphi v\, d\Q^\beta(g)-\int u\cdot v\cdot
V^\beta_\varphi\, d\Q^\beta(g)+0.
\end{eqnarray*}
To justify the last equality, note that according to Lemma
\ref{finiteness} $|\log Y^\beta_{e_{t\varphi}}|\le C\cdot |t|$ for
$|t|\le1$. Hence, the claim follows with dominated convergence and
Lemma \ref{drift-ident}.
\end{proof}

\begin{corollary} The operator $(D_\varphi, \Syl^1(\G))$ is closable in $L^2(\Q^\beta)$. Its closure will
be denoted by  $(D_\varphi,\Dom(D_\varphi))$.
\end{corollary}

In other words, $\Dom(D_\varphi)$ is the closure (or completion) of
$\Syl^1(\G)$ with respect to the norm
$$u\mapsto \left(\int [u^2+ (D_\varphi u)^2]\,
d\Q^\beta\right)^{1/2}.$$ Of course, the space $\Dom(D_\varphi)$
will depend on $\beta$ but we assume $\beta>0$ to be fixed for the
sequel.

\begin{remark}\label{operator}\rm The bilinear form
\begin{equation}
\E_\varphi(u,v):=\int D_\varphi u \cdot D_\varphi v\,
d\Q^\beta,\qquad \Dom(\E_\varphi):=\Dom(D_\varphi)
\end{equation}
is a Dirichlet form on $L^2(\G,\Q^\beta)$ with form core
$\Syl^\infty(\G)$. Its generator $(L_\varphi,\Dom(L_\varphi))$ is
the Friedrichs extension of the
 symmetric operator $$(-D^*_\varphi\circ D_\varphi, \ \Syl^2(\G)).$$
\end{remark}

\end{subsection}

\begin{subsection}{Derivatives and Integration by Parts Formula on  $\Pe(\eI)$}

Now let us have a look on flows on $\eI$. To do so, let a function
$\varphi\in\C^\infty(\eI,\R)$ with  $\varphi(0)=\varphi(1)=0$ be
given. (Note that each such function can be regarded as
$\varphi\in\C^\infty(S^1,\R)$ with $\varphi(0)=0$.) The flow
equation (\ref{flow}) now defines a flow $e_{t\varphi}$, $t\in\R$,
of order preserving $\C^\infty$ diffeomorphisms of $\eI$. In
particular, $e_{t\varphi}(0)=0$ and $e_{t\varphi}(1)=1$ for all
$t\in\R$.

\medskip

Lemma \ref{drift-lemma} together with  Theorem \ref{CoV-I}
immediately yields
\begin{lemma}\label{drift-lemma-i}
For $\varphi\in\C^\infty(\eI,\R)$ with $\varphi(0)=\varphi(1)=0$ and
each $\beta\ge0$
\begin{equation}\label{drift-ident-i}
\left.\frac{\partial}{\partial
t}Y^\beta_{e_{t\varphi},0}(g)\right|_{t=0} =
V^\beta_\varphi(g)-\frac{\varphi'(0)+\varphi'(1)}2=:
V^\beta_{\varphi,0}(g).
\end{equation}
\end{lemma}

\medskip

For  functions
 $u:\G_0\to\R$ we will define the {\em directional derivative} along
 $\varphi\in \C^\infty(\eI,\R)$ with $\varphi(0)=\varphi(1)=0$
as before  by
\begin{equation}\label{dir-der-def-i}
D_{\varphi} u(g):=\lim_{t\to0}\frac1t\left[u(e_{t\varphi}\circ
g)-u(g)\right]
\end{equation}
provided this limit exists. We will consider three classes of
'cylinder functions' for which the existence of this limit is
guaranteed.

\begin{definition}
{\bf (i)}\
 We
say that a function $u:\G_0\to\R$ belongs to the class
$\Cyl^k(\G_0)$ (for $k\in\N\cup\{0,\infty\}$) if it can be written
as
\begin{equation}\label{cyl}
u(g)=U\left(\smint \vec f(t) g(t)dt \right)
\end{equation}
for some $m\in\N$, some $\vec f=(f_1,\ldots,f_m)$ with $f_i\in
L^2(\eI,\leb)$ and some $\C^k$-function $U: \R^m\to\R$. Here and in
the sequel, we write $\int \vec f(t)g(t) dt =\left( \int_0^1
f_1(t)g(t) dt, \ldots, \int_0^1 f_m(t)g(t) dt\right)$.

{\bf (ii)}\
 We
say that $u:\G_0\to\R$ belongs to the class $\Syl^k(\G_0)$ if it can
be written as
\begin{equation}\label{syl}
u(g)=U\left(g(x_1),\ldots, g(x_m)\right)
\end{equation}
for some $m\in\N$, some $x_1,\ldots,x_m\in \eI$ and some
$\C^k$-function $U: \R^m\to\R$.

{\bf (iii)}\
 We
say that $u:\G_0\to\R$ belongs to the class $\Zyl^k(\G_0)$ if it can
be written as
\begin{equation}\label{zyl}
u(g)= U\left(\smint \vec\alpha(g_s) ds  \right)
\end{equation}
 with $U$ as above, $\vec\alpha=(\alpha_1,\ldots,\alpha_m)\in \C^k(\eI,\R^m)$ and
$\int \vec\alpha(g_s) ds =\left( \int_0^1 \alpha_1(g_s) ds, \ldots,
\int_0^1 \alpha_m(g_s) ds \right)$.

\end{definition}

\begin{remark}\rm For each $\varphi\in \C^\infty(S^1,\R)$ with
$\varphi(0)=0$ (which can be regarded as
 $\varphi\in \C^\infty(\eI,\R)$ with $\varphi(0)=\varphi(1)=0$),
  the definitions of $D_{\varphi}$ in (\ref{dir-der-def}) and (\ref{dir-der-def-i}) are
consistent in the following sense. Each cylinder function $u\in
\Syl^1(\G_0)$ defines by $v(g):=u(g-g_0)$  ($\forall g\in\G$) a
cylinder function $v\in \Syl^1(\G)$ with $D_{\varphi} v=D_{\varphi}
u$ on $\G_0$. Conversely, each cylinder function $v\in \Syl^1(\G)$
defines by $u(g):=v(g)$ ($\forall g\in\G_0$) a cylinder function
$u\in \Syl^1(\G_0)$ with $D_{\varphi} v=D_{\varphi} u$ on $\G_0$.
\end{remark}

\begin{lemma}\label{Lie-deriv-i} 
{\bf (i)} \ The directional derivative $D_{\varphi} u(g)$ exists for
all $u\in\Cyl^1(\G_0)\cup \Syl^1(\G_0)\cup \Zyl^1(\G_0)$ (in each
point $g\in\G_0$ and in each direction $\varphi\in
\C^\infty(\eI,\R)$ with $\varphi(0)=\varphi(1)=0$) and $D_{\varphi}
u(g)=\lim_{t\to0}\frac1t\left[u(g+t\cdot\varphi\circ
g)-u(g)\right]$. Moreover,
\begin{eqnarray*}
D_{\varphi} u(g) &=&\sum_{i=1}^m {\partial_i}U\left(\smint \vec
f(t)g(t)dt\right) \cdot \smint f_i(t)\varphi(g(t))dt
\end{eqnarray*}
 for each $u\in\Cyl^1(\G_0)$
 as in (\ref{cyl}),
\begin{eqnarray*}
D_{\varphi} u(g) &=&\sum_{i=1}^m
{\partial_i}U(g(x_1),\ldots,g(x_m))\cdot \varphi(g(x_i))
\end{eqnarray*}
 for each $u\in\Syl^1(\G_0)$
 as in (\ref{syl}), and
 \begin{eqnarray*}
D_{\varphi} u(g) &=&\sum_{i=1}^m {\partial_i}U\left(\smint
\vec\alpha(g_s) ds  \right) \cdot\smint \alpha_i'(g_s)
\varphi(g_s)ds
\end{eqnarray*}
 for each $u\in\Zyl^1(\G_0)$
 as in (\ref{zyl}).

 {\bf(ii)} \ \label{adjoint-i}
For $\varphi\in \C^\infty(\eI,\R)$ with $\varphi(0)=\varphi(1)=0$
let $D^*_{\varphi,0}$ denote the
 operator  in $L^2(\G_0,\Q_0^\beta)$ adjoint to $D_{\varphi}$.
Then for all $u\in\Cyl^1(\G_0)\cup \Syl^1(\G_0)\cup \Zyl^1(\G_0)$
\begin{equation}D^*_{\varphi,0} u=-D_{\varphi} u -  V^\beta_{\varphi,0}\cdot u.\end{equation}
\end{lemma}

\begin{proof} See the proof of the analogous results in Lemma
\ref{Lie-deriv} and Proposition \ref{adjoint}.
\end{proof}

\begin{remark}\rm
The operators $(D_{\varphi}, \Cyl^1(\G_0))$, $(D_{\varphi},
\Syl^1(\G_0))$, and
 $(D_{\varphi}, \Zyl^1(\G_0))$ are closable in $L^2(\Q_0^\beta)$.
The  closures of $(D_{\varphi}, \Cyl^1(\G_0))$, $(D_{\varphi}, \Zyl^1(\G_0))$ and $(D_{\varphi},
\Syl^1(\G_0))$ coincide. They will be denoted by
$(D_{\varphi},\Dom(D_{\varphi}))$.  See (proof of) Corollary \ref{dir-form-G0}.
\end{remark}

\end{subsection}

\end{section}

\begin{section}{Dirichlet Form and Stochastic Dynamics on on $\G$}

At each point $g\in\G$, the directional derivative $D_\varphi u(g)$
of any 'nice' function $u$ on $\G$ defines a linear form
$\varphi\mapsto D_\varphi u(g)$ on $\C^\infty(S^1)$. If we specify a
pre-Hilbert norm $\|.\|_g$ on $\C^\infty(S^1)$ for which this linear
form is continuous then there exists a unique element $Du(g)\in
T_g\G$ with $D_\varphi u(g)=\langle Du(g),\varphi\rangle_g$ for all
$\varphi\in\C^\infty(S^1)$. Here $T_g\G$ denotes the completion of
$\C^\infty(S^1)$ w.r.t. the norm $\|.\|_g$.

The canonical choice of a Dirichlet form on $\G$ will then be (the
closure of)
\begin{equation}\label{dir-g}\E(u,v)=
\int_\G \langle Du(g), Dv(g)\rangle_g\, d\Q^\beta(g), \qquad
u,v\in\Syl^1(\G).
\end{equation}
Given such a Dirichlet form, there is a straightforward procedure to
construct an operator ('generalized Laplacian') and a Markov process
('generalized Brownian motion'). Different choices of  $\|.\|_g$ in
general will lead to completely different Dirichlet forms, operators
and Markov processes.

We will discuss in detail two choices: in this chapter we will
choose
 $\|.\|_g$ (independent of $g$) to be the Sobolev norm
 $\|.\|_{H^s}$ for some $s>1/2$;
 in the remaining chapters,  $\|.\|_g$ will always be the $L^2$-norm
 $\varphi\mapsto(\int_{S^1} \varphi(g_t)^2dt)^{1/2}$ of
 $L^2(S^1,g_*\leb)$.

 \medskip

For the sequel,  fix -- once for ever  -- the number $\beta>0$ and
drop it from the notations, i.e. $\Q:=\Q^\beta$,
$V_\varphi:=V_\varphi^\beta$ etc.

\begin{subsection}{The Dirichlet Form on $\G$}

Let $(\psi_k)_{k\in\N}$ denote the standard Fourier basis of
$L^2(S^1)$. That is,
$$\psi_{2k}(x)=\sqrt2\cdot \sin(2\pi kx),\quad \psi_{2k+1}(x)=\sqrt2 \cdot\cos(2\pi kx)$$
for $k=1,2,\ldots$ and $\psi_1(x)=1$. It constitutes a complete
orthonormal system in $L^2(S^1)$: each $\varphi\in L^2(S^1)$ can
uniquely be written as $\varphi(x)=\sum_{k=1}^\infty c_k
\cdot\psi_k(x)$ with Fourier coefficients of $\varphi$ given by
$c_k:=\int_{S^1} \varphi(y) \psi_k(y)dy$. In terms of these Fourier
coefficients we define for each $s\ge0$ the norm
\begin{equation}\label{sobolev-norm}
\|\varphi\|_{H^s}:=\left(c_1^2+\sum_{k=1}^\infty
 k^{2s}\cdot(c_{2k}^2+c_{2k+1}^2)\right)^{1/2}\end{equation}
on $\C^\infty(S^1)$. The Sobolev space $H^s(S^1)$ is the completion
of $\C^\infty(S^1)$ with respect to the norm
 $\|.\|_{H^s}$.
It  has a complete orthonormal system consisting of smooth functions
$(\varphi_k)_{k\in\N}$. For instance, one may choose
\begin{equation}\label{fourier-s}
\varphi_{2k}(x)=\sqrt2\cdot k^{-s}\cdot \sin(2\pi kx),\quad
\varphi_{2k+1}(x)=\sqrt2 \cdot k^{-s}\cdot\cos(2\pi
kx)\end{equation} for $k=1,2,\ldots$ and $\varphi_1(x)=1$.

A linear form $A:\C^\infty(S^1)\to\R$ is {\em continuous} w.r.t.
$\|.\|_{H^s}$ --- and thus can be represented as $A(\varphi)=\langle
\psi,\varphi\rangle_{H^s}$ for some $\psi\in H^s(S^1)$  with
$\|\psi\|_{H^s}=\|A\|_{H^s}$ --- if and only if
\begin{equation}\|A\|_{H^s}:=\left(|A(\psi_1)|^2+\sum_{k=1}^\infty
 k^{2s}\cdot(|A(\psi_{2k})|^2+|A(\psi_{2k+1})|^2)\right)^{1/2}<\infty.\end{equation}

\begin{proposition} Fix a number $s>1/2$. Then for each cylinder function $u\in\Syl(\G)$
 and each $g\in\G$, the directional derivative
defines a continuous linear form $\varphi\mapsto D_\varphi u(g)$ on
$\C^\infty(S^1)\subset H^s(S^1)$. There exists a unique tangent
vector $Du(g)\in H^s(S^1)$ such that
 $D_\varphi u(g)=\langle Du(g),\varphi\rangle_{H^s}$ for all
$\varphi\in\C^\infty(S^1)$.

In terms of the family  $\Phi=(\varphi_k)_{k\in \N}$ from
(\ref{fourier-s})
$$Du(g)=\sum_{k=1}^\infty D_{\varphi_k}u(g) \cdot \varphi_k(.)$$ and
\begin{equation}\label{def-dd}
\|Du(g)\|_{H^s}^2=\sum_{k=1}^\infty |D_{\varphi_k}u(g)|^2.
\end{equation}
\end{proposition}
\begin{proof}
It remains to prove that the RHS of (\ref{def-dd}) is finite for
each $u$ and $g$ under consideration. According to Lemma
\ref{Lie-deriv}, for any $u\in\Syl(\G)$ represented as in
(\ref{syl})
\begin{eqnarray*}\sum_{k=1}^\infty  |D_{\varphi_k}u(g)|^2
&=& \sum_{k=1}^\infty \left( \sum_{i=1}^m
{\partial_i}U(g(x_1),\ldots,g(x_m))\cdot
\varphi_k(g(x_i))\right)^2\\
&\le& m\cdot \|\nabla
U\|_\infty^2\cdot\|\sum_{k=1}^\infty\varphi_k^2\|_{\infty}\ =\
m\cdot \|\nabla U\|_\infty^2\cdot(1+4\sum_{k=1}^\infty k^{-2s}).
\end{eqnarray*}
And, indeed, the latter is finite for each $s>1/2$.
\end{proof}

For the sequel, let us now fix a number $s>1/2$ and define
\begin{equation}\label{e-d}\E(u,v)=
 \int_\G \langle Du(g), Dv(g)\rangle_{H^s}\, d\Q(g)
\end{equation}
 for
$u,v\in\Syl^1(\G)$. Equivalently, in terms  of the family
$\Phi=(\varphi_k)_{k\in \N}$ from (\ref{fourier-s})
\begin{equation}\label{sum-dir}
\E(u,v)=\sum_{k=1}^\infty\int_\G D_{\varphi_k}u(g)\cdot
D_{\varphi_k}v(g)\, d\Q(g).
\end{equation}

\begin{theorem}\label{dir-form}
{\bf (i)} \ $(\E,\Syl^1(\G))$ is closable. Its closure $(\E,
\Dom(\E))$ is a regular Dirichlet form on $L^2(\G,\Q)$ which is
strongly local and recurrent (hence, in particular, conservative).

{\bf (ii)} \ For $u\in\Syl^1(\G)$ with representation
(\ref{syl-def})
\begin{eqnarray*}\E(u,u)&=&
\sum_{k=1}^\infty \int_\G \left( \sum_{i=1}^m
{\partial_i}U(g(x_1),\ldots,g(x_m))\cdot
\varphi_k(g(x_i))\right)^2\, d\Q(g).
\end{eqnarray*}
The generator of the Dirichlet form is the Friedrichs  extension of
the operator $L$ given on
 $\Syl^2(\G)$ by
\begin{eqnarray*}
L u(g)&=&
\sum_{i,j=1}^m\sum_{k=1}^\infty{\partial_i\partial_j}U\left(g(x_1),\ldots,g(x_m)\right)
  \varphi_k(g(x_i))\varphi_k(g(x_j))\\
  &&+
\sum_{i=1}^m\sum_{k=1}^\infty{\partial_i}U\left(g(x_1),\ldots,g(x_m)\right)
[\varphi'_k(g(x_i))+V_{\varphi_k}(g)]\varphi_k(g(x_i)).
\end{eqnarray*}

{\bf (iii)} \
 $\Zyl^1(\G)$ is a core for $\Dom(\E)$ (i.e. it is contained in the latter as a dense subset). For
$u\in\Zyl^1(\G)$ with representation (\ref{zyl})
\begin{eqnarray*}\E(u,u)&=&
\sum_{k=1}^\infty \int_\G \left( \sum_{i=1}^m {\partial_i}U(\smint
\vec\alpha(g_t)dt)\cdot \smint
\alpha_i'(g_t)\varphi_k(g_t)dt\right)^2\, d\Q(g).
\end{eqnarray*}
The generator of the Dirichlet form is the Friedrichs extension of
the operator $L$ given on
 $\Zyl^2(\G)$ by
\begin{eqnarray*}
L u(g)&=&
\sum_{i,j=1}^m\sum_{k=1}^\infty{\partial_i\partial_j}U\left(\smint
\vec \alpha(g_t)dt\right) \cdot \smint
\alpha_i'(g_t)\varphi_k(g_t)dt\cdot \smint
\alpha_j'(g_t)\varphi_k(g_t)dt
\\
  &&+
\sum_{i=1}^m\sum_{k=1}^\infty{\partial_i}U\left(\smint
\vec \alpha(g_t)dt\right)
\{ V_{\varphi_k}(g)+\smint [\alpha_i''(g_t)\varphi_k^2(g_t) + \alpha_i'(g_t)\varphi_k'(g_t)\varphi_k(g_t)]dt\}.
\end{eqnarray*}

{\bf (iv)} \ The intrinsic metric $\rho$ can be estimated from below
in terms of the $L^2$-metric:
$$\rho(g,h)\ge \frac1{\sqrt{C}} \|g-h\|_{L^2}.$$
\end{theorem}

\begin{remark}\rm
All assertions of the above Theorem remain valid for any $\E$
defined as in (\ref{sum-dir}) with any choice of a sequence
$\Phi=(\varphi_k)_{k\in \N}$  of smooth functions on $S^1$ with
 \begin{equation}\label{summable}
 C:=\|\sum_{k=1}^\infty
\varphi_k^2\|_{\infty}<\infty.\end{equation} (This condition is
satisfied for the sequence from (\ref{fourier-s}) if and only if
$s>1/2$.)
\end{remark}

The proof of the Theorem will make use of the following

\begin{lemma}
{\bf (i)} \ $\Dom(\E)$ contains all functions $u$ which can be
represented as
\begin{equation}\label{u-dist}
u(g)=U(\|g-f_1\|_{L^2},\ldots, \|g-f_m\|_{L^2})
\end{equation}
with some $m\in\N$, some $f_1,\ldots, f_m\in\G$ and some
$U\in\C^1(\R^m,\R)$.

For each $u$ as above, each $\varphi\in\C^\infty(S^1)$ and
$\Q$-a.e. $g\in\G$
$$D_\varphi u(g)= \sum_{i=1}^m
{\partial_i}U(\|g-f_1\|_{L^2},\ldots, \|g-f_m\|_{L^2})
\cdot
\int_{S^1}\mbox{\rm sign}(g(t)-f_i(t))\frac{|g(t)-f_i(t)|}{\|g-f_i\|_{L^2}}\varphi(g(t))dt
$$
where $\mbox{\rm sign}(z):=+1$ for $z\in S^1$ with $|[0,z]|\le 1/2$ and
$\mbox{\rm sign}(z):=-1$ for $z\in S^1$ with $|[z,0]|< 1/2$.
{\bf (ii)} \ Moreover, $\Dom(\E)$ contains all functions $u$ which
can be represented as
\begin{equation}\label{u-moll}
u(g)=U(g_{\epsilon_1}(x_1),\ldots, g_{\epsilon_m}(x_m))
\end{equation}
with some $m\in\N$, some  $x_1,\ldots, x_m\in S^1$, some
$\epsilon_1,\ldots,\epsilon_m\in\,]0,1[$ and some
$U\in\C^1((S^1)^m,\R)$.

Here $g_\epsilon(x):=\int_x^{x+\epsilon}g(t)dt\in S^1$ for $x\in
S^1$ and $0<\epsilon<1$. More precisely,
$$g_\epsilon(x):=\pi(\int_x^{x+\epsilon}\pi^{-1}g(t)dt)$$
where $\pi: \G(\R)\to\G$ (cf. section 2.2) denotes the projection
and $\pi^{-1}: \G\to\G(\R)$ the canonical lift with $\pi^{-1}(g)(t)\in [g(x), g(x)+1]\subset\R$
for $t\in[x,x+1]\subset\R$.

For each $u$ as above, each $\varphi\in\C^\infty(S^1)$ and each $g\in\G$
$$D_\varphi u(g)=
\sum_{i=1}^m
{\partial_i}U(g_{\epsilon_1}(x_1),\ldots,g_{\epsilon_m}(x_m))\cdot
\frac1{\epsilon_i}\int_{x_i}^{x_i+\epsilon_i}\varphi(g(t))dt.$$

{\bf (iii)} \ The set  of all $u$ of the form (\ref{u-moll}) is dense in $\Dom(\E)$.
\end{lemma}

\begin{proof}
(i) Let us first prove that for each $f\in\G$, the map $u(g)=\|g-f\|_{L^2}$ lies in $\Dom(\E)$.
For $n\in\N$, let $\pi_n:\G\to\G$ be the map which replaces each $g$ by the piecewise constant map:
$$\pi_n(g)(t):=g(\frac in)\qquad\mbox{for }t\in[\frac in,\frac{i+1}n[.$$
Then by right continuity $\pi_n(g)\to g$ as $n\to\infty$ and thus
$$\frac1n\sum_{i=0}^{n-1}|g(\frac in)-f(\frac in)|^2\longrightarrow
\int_{S^1}|g(t)-f(t)|^2dt.$$
Therefore, for each $g\in\G$ as $n\to\infty$
\begin{equation}\label{moll-appr}
u_n(g):=U_n(g(0), g(\frac1n),\ldots, g(\frac{n-1}n))\longrightarrow u(g)
\end{equation}
where $U_n(x_1,\ldots,x_{n}):=\left(\frac1n\sum_{i=0}^{n-1}d_n(x_{i+1}-f(\frac in))^2\right)^{1/2}$
and $d_n$ is a smooth approximation of the distance function $x\mapsto |x|$ on $S^1$ (which itself is non-differentiable
at $x=0$ and $x=\frac12$) with $|d_n'|\le1$ and $d_n(x)\to|x|$ as $n\to\infty$.
Obviously, $u_n\in\Syl^1(\G)$.

By dominated convergence, (\ref{moll-appr}) also implies that $u_n\to u$ in $L^2(\G,\Q)$.
Hence, $u\in\Dom(\E)$ if (and only if) we can prove that
$$ \sup_n\E(u_n)<\infty.$$
But
\begin{eqnarray*}
\E(u_n)&=&
\sum_{k=1}^\infty \int_\G \left|\sum_{i=1}^{n}\partial_iU_n(g(0), g(\frac1n),\ldots, g(\frac{n-1}n))\cdot
\varphi(g(\frac{i-1}n))\right|^2d\Q(g)\\
&\le&\sum_{k=1}^\infty\int_\G\frac1n\sum_{i=1}^{n} \varphi_k^2(g(\frac {i-1}n))\, d\Q(g)\ = \
\sum_{k=1}^\infty\|\varphi_k\|_{L^2}^2<\infty,
\end{eqnarray*}
uniformly in $n\in\N$.
This proves the claim for the function $u(g)=\|g-f\|_{L^2}$.

From this, the general claim follows immediately:
if $v_n$, $n\in\N$, is a sequence of $\Syl^1(\G)$ approximations of $g\mapsto \|g-0\|_{L^2}$ then
$u_n(g):=U(v_n(g-f_1),\ldots, v_n(g-f_m))$ defines a sequence of $\Syl^1(\G)$ approximations of
$u(g)=U( \|g-f_1\|_{L^2},\ldots, \|g-f_m\|_{L^2})$.

\medskip

(ii)
Again it suffices to treat the particular case $m=1$ and $U=id$, that is, $u(g)=g_\epsilon(x)$ for some
$x\in S^1$ and some $0<\epsilon<1$. Let $\tilde g\in\G(\R)$ be the lifting of $g$ and recall that
$u(g)=\pi(\frac1\epsilon\int_x^{x+\epsilon} \tilde g(t)dt)$.
Define $u_n\in\Syl^1(\G)$ for $n\in\N$ by
$u_n(g)=\pi(\frac1n\sum_{i=0}^{n-1}\tilde g(x+\frac in\epsilon))$.
Right continuity of $\tilde g$ implies $u_n\to u$ as $n\to\infty$ pointwise on $\G$ and thus also in $L^2(\G,\Q)$.
To see the  boundedness of $\E(u_n)$ note that
$D_\varphi u_n(g)=\frac1n\sum_{i=0}^{n-1}\varphi( g(x+\frac in\epsilon))$. Thus
$$\E(u_n)\le
\sum_{k=1}^\infty\int_\G\frac1n\sum_{i=0}^{n-1}\varphi_k^2( g(x+\frac in\epsilon))d\Q(g)
=\sum_{k=1}^\infty\|\varphi_k\|_{L^2}^2<\infty.$$

\medskip

(iii) We have to prove that each $u\in\Syl^1(\G)$ can be approximated in the norm $(\|.\|^2+\E(.))^{1/2}$
by functions $u_n$ of type (\ref{u-moll}). Again it suffices to treat the particular case $u(g)=g(x)$ for some $x\in S^1$.
Choose $u_n(g)=g_{1/n}(x)$. Then by right continuity of $g$, $u_n\to u$ pointwise on $\G$ and thus also
in $L^2(\G,\Q)$. Moreover, $D_\varphi u_n(g)=n\int_x^{x+1/n}\varphi(g(t))dt$ (for all $\varphi$ and $g$)
and therefore
$$\E(u_n)\le\sum_{k=1}^\infty n\int_x^{x+1/n}\varphi_k^2(g(t))dtd\Q(g)=
\sum_{k=1}^\infty\|\varphi_k\|_{L^2}^2<\infty.$$
\end{proof}

\begin{proof}[Proof of the Theorem]
{\bf (a)} \ The sum $\E$ of closable bilinear forms with common
domain $\Syl^1(\G)$ is closable, provided it is still finite on this
domain. The latter will follow by means of Lemma \ref{Lie-deriv}
which implies for all $u\in\Syl^1(\G)$ with representation
(\ref{cyl})
\begin{eqnarray*}\E(u,u)
&=&\sum_{k=1}^\infty \int_\G \left( \sum_{i=1}^m
{\partial_i}U(g(x_1),\ldots,g(x_m))\cdot
\varphi_k(g(x_i))\right)^2\, d\Q(g)\\
&\le& m\cdot \|\nabla
U\|_\infty^2\cdot\sum_{k=1}^\infty\|\varphi_k\|^2_{L^2(S^1)}<\infty.
\end{eqnarray*}
Hence, indeed $\E$ is finite on $\Syl^1(\G)$.

{\bf (b)} \ The Markov property for $\E$ follows from that of the
$\E_{\varphi_k}(u,v)=\int_\G D_{\varphi_k}u\cdot D_{\varphi_k}v\,
d\Q$.

{\bf (c)} \ According to the previous Lemma, the class of {\em continuous} functions of type
(\ref{u-moll}) is dense in $\Dom(\E)$. Moreover, the class of {\em finite energy} functions of type
(\ref{u-dist}) is dense in $\C(\G)$ (with the
$L^2$ topology of $\G\subset L^2(S^1)$, cf. Proposition \ref{g-top}). Therefore, the
Dirichlet form $\E$ is regular.

{\bf (e)} \ The estimate for the intrinsic metric is an immediate consequence of
the following estimate for the norm of the gradient of the function $u(g)=\|g-f\|_{L^2}$
(which holds for each $f\in\G$ uniformly in $g\in\G$):
\begin{eqnarray*}
\|Du(g)\|^2&=&\sum_{k=1}^\infty
\left(
\int_{S^1}\mbox{\rm sign}(g(t)-f_i(t))\frac{|g(t)-f_i(t)|}{\|g-f_i\|_{L^2}}\varphi_k(g(t))dt
\right)^2\\
&\le&\sum_{k=1}^\infty\int_{S^1}\varphi_k^2(g(t))dt\le\|\sum_{k=1}^\infty\varphi_k^2\|_\infty=:C.
\end{eqnarray*}

{\bf (f)} \ The {\em locality} is an immediate consequence of the
previous estimate: Given functions $u,v \in\Dom(\E)$ with disjoint
supports, one has to prove that $\E(u,v)=0$. Without restriction,
one may assume that $\supp[u]\subset B_r(g)$  and $\supp[v]\subset
B_r(h)$ with $\|g-h\|_{L^2}>2r+2\delta$. (The general case will
follow by a simple covering argument.) Without restriction, $u,v$
can be assumed to be bounded. Then $|u|\le C w_{\delta,g}$ and
$|v|\le C w_{\delta,h}$ for some constant $C$ where
$$w_{\delta,g}(f)=\left[\frac1\delta(r+\delta-\|f-g\|_{L^2}) \wedge 1 \right]\vee
0.$$ Given $u_n\in\Syl^1(\G)$ with $u_n\to u$ in $\Dom(\E)$ put
$$\overline u_n=(u_n\wedge w_{\delta,g})\vee(-w_{\delta,g}).$$
Then  $\overline u_n\to u$ in $\Dom(\E)$. Analogously,
 $\overline v_n\to v$ in $\Dom(\E)$ for
$\overline v_n=(v_n\wedge w_{\delta,h})\vee(-w_{\delta,h})$. But
obviously, $\E(\overline u_n,\overline v_n)=0$ since $\overline
u_n\cdot \overline v_n=0$. Hence, $\E(u,v)=0$.

{\bf (g)} \ In order to prove that $\Zyl^1(\G)$ is contained in
$\Dom(\E)$  it suffices to prove that each $u\in\Zyl^1(\G)$ of the
form $u(g)=\int\alpha(g_t)dt$ can be approximated in $\Dom(\E)$ by
$u_n\in\Syl^1(\G)$. Given $u$ as above with $\alpha\in\C^1(S^1,\R)$
put $u_n(g)=\frac1n\sum_{i=1}^n \alpha(g_{i/n})$. Then
$u_n\in\Syl^1(\G)$, $u_n\to u$ on $\G$ and
$$D_\varphi u_n(g)=\frac1n\sum_{i=1}^n
\alpha'(g_{i/n})\varphi(g_{i/n})\to
\int\alpha'(g_t)\varphi(g_t)dt=D_\varphi u(g).$$
Moreover,
\begin{eqnarray*}
\E(u_n,u_n)&=&
\int_\G\sum_k \left|\frac1n\sum_{i=1}^n
\alpha'(g_{i/n})\varphi(g_{i/n})\right|^2\, d\Q(g)\\
&\le&
C\cdot\int_\G\sum_k \frac1n\sum_{i=1}^n
\alpha'(g_{i/n})^2\, d\Q(g)=C\cdot \int_{S^1} \alpha'(t)^2dt
\end{eqnarray*}
uniformly in $n\in\N$.
 Hence,
$u\in\Dom(\E)$ and $$\E(u,u)=\lim_{n\to\infty}\E(u_n,u_n)=
\int_\G\sum_k \left|\int_{S^1}
\alpha'(g_t)\varphi_k(g_t)dt\right|^2d\Q(g).$$

 {\bf (h)} \ The set $\Zyl^1(\G)$ is dense in $\Dom(\E)$ since according to assertion (ii) of the previous Lemma already
 the subset of all $u$ of the form (\ref{u-moll}) is dense in $\Dom(\E)$.

 Finally, one easily verifies that $\Zyl^2(\G)$ is dense in $\Zyl^1(\G)$ and (using the integration by parts formula)
 that $L$ is a symmetric operator on
 $\Zyl^2(\G)$
 with the given representation.
\end{proof}

\begin{corollary} There exists a strong Markov process $(g_t)_{t\ge0}$ on $\G$, associated with the Dirichlet form
$\E$.
It has continuous trajectories and it is reversible w.r.t. the measure $\Q$.
Its generator has the form
$$\frac12 L=\frac12 \sum_k D_{\varphi_k}D_{\varphi_k}+\frac12
\sum_k V_{\varphi_k}\cdot D_{\varphi_k}$$
with $\{\varphi_k\}_{k\in\N}$ being the Fourier basis
of $H^s(S^1)$.
\end{corollary}

\begin{remark}\rm
This process
$(g_t)_{t\ge0}$ is closely related to the stochastic processes on the diffeomorphism group of $S^1$ and to the
'Brownian motion' on the homeomorphism group of $S^1$,

studied by Airault, Fang, Malliavin, Ren, Thalmaier and others \cite{MR2082490,AM06, AR02, Fang02,MR2091355,MR1713340}.
These are processes
with generator $\frac12 L_0=\frac12\sum_k D_{\varphi_k}D_{\varphi_k}$. For instance, in the case $s=3/2$ our process from the previous Corollary may be regarded as 'Brownian motion plus drift'.
All the previous approaches are restricted to $s\ge 3/2$.
The main improvements of our approach are:
\begin{itemize}
\item identification of a probability measure $\Q$ such that these processes --- after adding a suitable drift --- are reversible;
\item construction of such processes in all cases $s>1/2$.
\end{itemize}
\end{remark}
\end{subsection}

\begin{subsection}{Finite Dimensional Noise Approximations}

In the previous section, we have seen the construction of the
diffusion process on $\G$ under minimal assumptions. However, the
construction of the process is rather abstract. In this section, we
try to construct explicitly a diffusion process associated with the
generator of the Dirichlet form $\E$ from Theorem
\ref{dir-form}. Here we do not aim for greatest generality.

Let a {\em finite} family $\Phi=(\varphi_k)_{k=1,\ldots,n}$ of
smooth functions on $S^1$ be given and let $(W_t)_{t\ge0}$ with
$W_t=(W^1_t,\ldots, W^n_t)$ be a $n$-dimensional Brownian motion,
defined on some probability space $(\Omega,\mathcal{F},\mathbf{P})$.
For each $x\in S^1$ we define a stochastic processes
$(\eta_t(x))_{t\ge0}$ with values in $S^1$ as the strong solution of
the Ito differential equation
\begin{equation}\label{ito}
d\eta_t(x)=\sum_{k=1}^n\varphi_k(\eta_t(x))dW^k_t+\frac12\sum_{k=1}^n\varphi'_k(\eta_t(x))\varphi_k(\eta_t(x))dt
\end{equation}
with initial condition $\eta_0(x)=x$. Equation (\ref{ito}) can be
rewritten in Stratonovich form as follows
\begin{equation}\label{strato}
d\eta_t(x)=\sum_{k=1}^n\varphi_k(\eta_t(x))\diamond dW^k_t.
\end{equation}
Obviously, for every $t$  and for $\mathbf{P}$-a.e.
$\omega\in\Omega$, the function $x\mapsto \eta_t(x,\omega)$ is an
element of the semigroup $\G$. (Indeed, it is a
$\C^\infty$-diffeomorphism.) Thus (\ref{strato}) may also be
interpreted as a
 Stratonovich SDE  on the semigroup $\G$:
\begin{equation}\label{strato-2}
d\eta_t=\sum_{k=1}^n\varphi_k(\eta_t)\diamond dW^k_t,\quad \eta_0=e.
\end{equation}
This process on $\G$ is right invariant: if $g_t$ denotes the
solution to (\ref{strato-2}) with initial condition $g_0=g$ for some
initial condition $g\in\G$ then $g_t=\eta_t\circ g$. One easily
verifies that the generator of this process $(g_t)_{t\ge0}$ is given
on $\Syl^2(\G)$ by $\frac12\sum_{k=1}^n D_{\varphi_k}D_{\varphi_k}$.
What we aim for, however, is a process with generator
$$-\frac12\sum_{k=1}^n D^*_{\varphi_k}D_{\varphi_k}=
\frac12\sum_{k=1}^n D_{\varphi_k}D_{\varphi_k}+\frac12\sum_{k=1}^n
V_{\varphi_k}\cdot D_{\varphi_k}.$$ Define a new probability measure
$\mathbf{P}^g$ on $(\Omega, \mathcal{F})$, given on  $\mathcal{F}_t$
by
\begin{equation}\label{girsanov}
d\mathbf{P}^g=\exp\left(\sum_{k=1}^n\int_0^tV_{\varphi_k}(\eta_s\circ
g)dW^k_s-\frac12\sum_{k=1}^n\int_0^t |V_{\varphi_k}(\eta_s\circ
g)|^2ds\right)d\mathbf{P}
\end{equation}
and a semigroup $(P_t)_{t\ge0}$ acting on bounded measurable
functions $u$ on $\G$ as follows
$$P_tu(g)=\int_\Omega u(\eta_t(g(.),\omega))\, d\mathbf{P}^g(\omega).$$

\begin{proposition}
$(P_t)_{t\ge0}$ is a strongly continuous Markov semigroup on $\G$.
Its generator is an extension of the operator
$\frac12 L=-\frac12\sum_{k=1}^n D^*_{\varphi_k}D_{\varphi_k}$ with
domain $\Syl^2(\G)$. That is, for all $u\in\Syl^2(\G)$ and all
$g\in\G$
\begin{equation}
\lim_{t\to0}\frac1t \left( P_tu(g)-u(g)\right)=\frac12 L u(g).
\end{equation}

\end{proposition}
\begin{proof}
The strong continuity  follows easily from the fact that
$\eta_t(x,.)\to x$ a.s. as $t\to0$ which implies by dominated
convergence
$$P_tu(g)=\int_\Omega u(\eta_t\circ g)\, d\mathbf{P}^g\to  u(g)$$
for each continuous $u:\G\to\R$.

Now we aim for identifying the generator. According to Girsanov's
theorem, under the measure $\mathbf{P}^g$ the processes
$$\tilde W^k_t=W^k_t-\frac12\int_0^tV_{\varphi_k}(\eta_s\circ g)ds$$
for $k=1,\ldots,n$ will define $n$ independent Brownian motions. In
terms of these driving processes, (\ref{ito}) can be reformulated as
\begin{equation}\label{ito-2}
dg_t(x)=\sum_{k=1}^n\varphi_k(g_t(x))d\tilde W^k_t+
\frac12\sum_{k=1}^n[\varphi'_k(g_t(x))+V_{\varphi_k}(g_t)]\varphi_k(g_t(x))dt
\end{equation}
(recall that $g_s=\eta_s\circ g$). The chain rule applied to a
smooth function $U$ on $(S^1)^m$, therefore, yields
\begin{eqnarray*}
\lefteqn{dU\left(g_t(y_1),\ldots,g_t(y_m)\right)}\\
&=&
\sum_{i=1}^m\frac{\partial}{\partial x_i}U\left(g_t(y_1),\ldots,g_t(y_m)\right)dg_t(y_i)\\
&&+\frac12\sum_{i,j=1}^m\frac{\partial^2}{\partial x_i\partial
x_j}U\left(g_t(y_1),\ldots,g_t(y_m)\right)
d\langle g_.(y_i), g_.(y_j)\rangle_t\\
&=& \sum_{i=1}^m\sum_{k=1}^n\frac{\partial}{\partial
x_i}U\left(g_t(y_1),\ldots,g_t(y_m)\right)
\varphi_k(g_t(y_i))d\tilde{W}^k_t\\
&&+ \frac12\sum_{i=1}^m\sum_{k=1}^n\frac{\partial}{\partial
x_i}U\left(g_t(y_1),\ldots,g_t(y_m)\right)
[\varphi'_k(g_t(y_i))+V_{\varphi_k}(g_t)]\varphi_k(g_t(y_i))dt\\
&&+ \frac12\sum_{i,j=1}^m\sum_{k=1}^n\frac{\partial^2}{\partial
x_i\partial x_j}U\left(g_t(y_1),\ldots,g_t(y_m)\right)
  \varphi_k(g_t(y_i))\varphi_k(g_t(y_j))dt.
\end{eqnarray*}
Hence, for a cylinder function of the form $u(g)=U(g(y_1),\ldots,
g(y_m))$ we obtain
\begin{eqnarray*}
\lefteqn{\lim_{t\to0}\frac1t\left( P_tu(g)-u(g)\right)}\\
&=&\lim_{t\to0}\frac1t \int_\Omega
\left[U\left(g_t(y_1),\ldots,g_t(y_m)\right)
-U\left(g_0(y_1),\ldots,g_0(y_m)\right)\right]\, d\mathbf{P}^g\\
&=&\lim_{t\to0}\frac1t \int_\Omega \int_0^t\left[
\frac12\sum_{i=1}^m\sum_{k=1}^n\frac{\partial}{\partial
x_i}U\left(g_s(y_1),\ldots,g_s(y_m)\right)
[\varphi'_k(g_s(y_i))+V_{\varphi_k}(g_s)]\varphi_k(g_s(y_i))\right.\\
&&+
\left.\frac12\sum_{i,j=1}^m\sum_{k=1}^n\frac{\partial^2}{\partial
x_i\partial x_j}U\left(g_s(y_1),\ldots,g_s(y_m)\right)
  \varphi_k(g_s(y_i))\varphi_k(g_s(y_j))\right]\, ds \,  d\mathbf{P}^g\\
&\stackrel{(*)}=&
\frac12\sum_{i=1}^m\sum_{k=1}^n\frac{\partial}{\partial
x_i}U\left(g(y_1),\ldots,g(y_m)\right)
[\varphi'_k(g(y_i))+V_{\varphi_k}(g)]\varphi_k(g(y_i))\\
&&+ \frac12\sum_{i,j=1}^m\sum_{k=1}^n\frac{\partial^2}{\partial
x_i\partial x_j}U\left(g(y_1),\ldots,g(y_m)\right)
  \varphi_k(g(y_i))\varphi_k(g(y_j))\\
&=&\frac12\sum_{k=1}^n\left[
D_{\varphi_k}D_{\varphi_k}u(g)+V_{\varphi_k}(g)\cdot
D_{\varphi_k}u(g)\right] =-\frac12\sum_{k=1}^n
D_{\varphi_k}^*D_{\varphi_k}u(g).
\end{eqnarray*}
In order to justify ($*$), we have to verify continuity in $s$ in all
the expressions preceding ($*$). The only term for which this is not
obvious is $V_{\varphi_k}(g_s)$. But $g_s=\eta_s\circ g$ with a
function $\eta_s(x,\omega)$ which is continuous in $x$ and in $s$.
Thus $V_{\varphi_k}(\eta_s(.,\omega)\circ g)$ is continuous in $s$.
\end{proof}

\begin{remark}\label{rem-summ}\rm
All the previous argumentations in principle also apply to infinite
families of $(\varphi_k)_{k=1,2,\ldots}$, provided they  have
sufficiently good integrability properties. For instance, the family
(\ref{fourier-s}) with $s>\frac52$ will do the job. There are three
key steps which require a careful verification:
\begin{itemize}
\item the solvability of the Ito equation (\ref{ito}) and the fact that the solutions are
homeomorphisms of $S^1$; here $s\ge\frac32$ suffices, cf. \cite{MR1713340};
\item the boundedness of the quadratic variation of the drift to justify Girsanov's transformation in
(\ref{girsanov}); for $s>\frac52$ this will be satisfied since Lemma
\ref{drift-lemma} implies (uniformly in $g$)
\begin{eqnarray*}
\sum_{k=1}^\infty |V_{\varphi_k}(g)|^2&\le& (\beta+1)^2
\sum_{k=1}^\infty \int_0^1|\varphi''_k(x)|^2dx\le 4(\beta+1)^2
\sum_{k=1}^\infty k^{4-2s};
\end{eqnarray*}
\item the finiteness of the generator and Ito's chain rule for $\C^2$-cylinder functions;
here $s>\frac32$ will be sufficient.
\end{itemize}
\end{remark}

\begin{remark}\rm
Another completely different approximation of the process
$(g_t)_{t\ge0}$ in terms of finite dimensional SDEs is obtained as
follows. For $N\in\N$, let $\Syl^1_N$ denote the set of cylinder
functions $u:\G\to\R$ which can be represented as $u(g)=U(g(1/N),
g(2/N),\ldots, g(1))$ for some $U\in\C^1((S^1)^N)$. Denote the
closure of $(\E,\Syl^1_N)$ by $(\E^N,\Dom(\E^N))$. It is the image
of the Dirichlet form $(E^N,\Dom(E^N))$ on $\Sigma_N\subset(S^1)^N$ given by
\begin{equation}
E^N(U)=\int_{\Sigma_N} \sum_{i,j=1}^N\partial_i U(x) \partial_jU(x) \, a_{ij}(x)\rho(x)\, dx
\end{equation}
with
$$a_{ij}(x)=\sum_{k=1}^\infty \varphi_k(x_i)\varphi_k(x_j),\qquad
\rho(x)=\frac{\Gamma(\beta)}{\Gamma(\beta/N)^N}
\prod_{i=1}^N (x_{i+1}-x_i)^{\beta/N-1} dx.$$
and (as before)
$\Sigma_N=\left\{(x_1,\ldots,x_N)\in (S^1)^N: \ \sum_{i=1}^N
|[x_i,x_{i+1}]|=1\right\}$.
That is,
$$\E^N(u)=E^N(U)$$
for cylinder functions $u\in\Syl^1_N$ as above.
Let $(X_t,\mathbf{P}_x)_{t\ge0, x\in\Sigma_N}$ be the Markov process on $\Sigma_N$ associated with $E^N$. Then the semigroup
associated with $\E^N$ is given by
$$T^N_tu(g)=\mathbf{E}_{g(1/N),\ldots,g(1)}\left[U(X_t)\right].$$

Now let $(g_t,\mathbf{P}_g)_{t\ge0, g\in\G}$ and $(T_t)_{t\ge0}$ denote the Markov process and the $L^2$-semigroup  associated
with $\E$. Then as $N\to\infty$
$$T^{2^N}_t\to T_t\quad\mbox{strongly in }L^2$$
since
$$\E^{2^N}\searrow \E$$
in the sense of quadratic forms, \cite{Reed/Simon}, Theorem S.16. (Note that $\cup_{N\in\N}\Syl^1_{2^N}$ is dense in $\Dom(\E)$.)
\end{remark}

\end{subsection}
 \begin{subsection}{Dirichlet Form and Stochastic Dynamics on $\G_1$ and $\Pe$}

In order to define the derivative of a function $u: \G_1\to\R$ we
regard it as a function $\tilde u$ on $\G$ with the property $\tilde
u(g)=\tilde u(g\circ \theta_z)$ for all $z\in S^1$. This implies
that $D_\varphi \tilde u(g)=(D_\varphi \tilde u)(g\circ \theta_z)$
whenever one of these expressions is well-defined. In other words,
$D_\varphi \tilde u$ defines a function on $\G_1$ which will be
denoted by $D_\varphi u$ and called {\em the directional derivative
of $u$ along $\varphi$}.

\begin{corollary}
{\bf (i)} \ Under assumption (\ref{summable}), with the notations
from above,
$$\E(u,u)=\sum_{k=1}^\infty \int_{\G_1}|D_{\varphi_k}u|^2 \, d\Q.$$
defines a regular, strongly local, recurrent Dirichlet form on
$L^2(\G_1,\Q)$.

{\bf (ii)} \ The Markov process on $\G$ analyzed in the previous
section extends to a (continuous, reversible) Markov process on $\G_1$.
\end{corollary}

In order to see the second claim, let $g,\tilde g\in \G$ with
$\tilde g=g\circ \theta_z$ for some $z\in S^1$. Then obviously,
$$\tilde g_t(.,\omega)=\eta_t(\tilde g(.),\omega)=\eta_t(g(.+z),\omega)=g_t(.,\omega)\circ \theta_z.$$
Moreover,
$$\mathbf{P}^{\tilde g}=\mathbf{P}^{g}$$
since $V_\varphi(g\circ \theta_z)=V_\varphi(g)$ for all $\varphi$
under consideration and all $z\in S^1$.

\medskip

\medskip

\medskip

\medskip

The objects considered previously -- derivative,
Dirichlet form and Markov process on $\G_1$ --
 have canonical counterparts on $\Pe$.
The key to these new objects is the bijective map $\chi:\G_1\to\Pe$.

The {flow} generated by a smooth 'tangent vector' $\varphi:S^1\to\R$
through the point
 $\mu\in\Pe$ will be given by
$((e_{t\varphi})_*\mu)_{t\in\R}$. In these terms, the directional
derivative of a function $u:\Pe\to\R$ at the point
 $\mu\in\Pe$ in direction $\varphi\in\C^\infty(S^1,\R)$ can be expressed as
$$D_\varphi
u(\mu)=\lim_{t\to0}\frac1t\left[u((e_{t\varphi})_*\mu)-u(\mu)\right],$$
provided this limit exists.
 The {adjoint operator} to
$D_\varphi$ in $L^2(\Pe,\Pp)$ is given (on a suitable dense
subspace) by
$$
D^*_\varphi u(\mu)=-D_\varphi(\mu)-V_\varphi(\chi^{-1}(\mu))\cdot
u(\mu).$$ The {drift term} can be represented as
$$V_\varphi(\chi^{-1}(\mu))=
\beta\int_0^1 \varphi'(s) \, \mu(ds) \ + \
\sum_{I\in\Gaps(\mu)}\left[\frac{\varphi'(I_-)+\varphi'(I_+)}2-\frac{\varphi(I_+)-\varphi(I_-)}{|I|}\right].
$$

Given a sequence $\Phi=(\varphi_k)_{k\in \N}$  of smooth functions
on $S^1$ satisfying (\ref{summable}), we obtain a (regular, strongly
local, recurrent) Dirichlet form $\E$ on $L^2(\Pe,\Pp)$ by
\begin{equation}\label{p-dir-form}
\E(u,u)=\sum_k\int_\Pe |D_{\varphi_k} u(\mu)|^2 d\Pp(\mu).
\end{equation}
It is the image of the
Dirichlet form defined in (\ref{sum-dir}) under the map $\chi$.
The generator of $\E$ is given
on an appropriate dense subspace of $L^2(\Pe,\Pp)$ by
\begin{equation}\label{gen}
L=-\sum_{k=1}^\infty D^*_{\varphi_k} D_{\varphi_k}.
\end{equation}
For $\Pp$-a.e. $\mu_0\in\Pe$, the associated {Markov process} $(\mu_t)_{t\ge0}$ on $\Pe$ starting in $\mu_0$
 is given as
$$\mu_t(\omega)=g_t(\omega)_*\leb$$
where $(g_t)_{t\ge0}$ is the process on $\G$, starting in $g_0:=\chi^{-1}(\mu_0)$.
(As mentioned before, $(g_t)_{t\ge0}$ admits a more direct construction provided we
 restrict ourselves to a finite sequence $\Phi=(\varphi_k)_{k=1,\ldots,n}$.)
\end{subsection}

\begin{subsection}{Dirichlet Form and Stochastic Dynamics on $\G_0$ and $\Pe_0$}
For $s>0$ and $\varphi:[0,1]\to\R$ let the Sobolev norm
$\|\varphi\|_{H^s}$ be defined as in (\ref{sobolev-norm})
and let $H^s_0([0,1])$ denote the closure of $\C_c^\infty(]0,1[)$, the space of smooth
$\varphi:[0,1]\to\R$ with compact support in $]0,1[$. If $s\ge1/2$ (which is the only case we are interested in)
$H^s_0([0,1])$ can be identified with $\{\varphi\in H^s([0,1]):\ \varphi(0)=\varphi(1)=0\}$ or equivalently with
$\{\varphi\in H^s(S^1):\ \varphi(0)=0\}$. For the sequel, fix $s>1/2$ and a complete
 orthonormal basis $\Phi=\{\varphi_k\}_{k\in\N}$
 of $H^s_0([0,1])$ with $C:=\|\sum_k \varphi_k^2\|_\infty<\infty$, and define
 $$\E_0(u,u)=\sum_{k=1}^\infty \int_{\G_0}|D_{\varphi_k,0}u(g)|^2\, d\Q_0(g).$$

 \begin{corollary}\label{dir-form-G0}
$(\E_0,\Syl^1(\G_0))$, $(\E_0,\Zyl^1(\G_0))$ and $(\E_0,\Cyl^1(\G_0))$ are closable. Their closures coincide
and define a regular, strongly local, recurrent Dirichlet form  $(\E_0,
\Dom(\E_0))$ on $L^2(\G_0,\Q_0)$.
\end{corollary}

\begin{proof}
For the closability  (and the equivalence of the respective closures) of
$(\E_0,\Syl^1(\G_0))$ and  $(\E_0,\Zyl^1(\G_0))$, see the proof of Theorem \ref{dir-form}.
Also all the assertions on the closure are deduced in the same manner.
For the closability   of
$(\E_0,\Cyl^1(\G_0))$ (and the equivalence of its closure with the previously defined closures),
see the proof of Theorem \ref{dir-form-cyl} below.

\end{proof}

As explained in the previous subsection, these objects (invariant measure, derivative,
Dirichlet form and Markov process) on $\G_0$  have canonical counterparts on $\Pe_0$
defined by means of the bijective map $\chi:\G_0\to\Pe_0$.

 \end{subsection}
\end{section}

\begin{section}{The Canonical Dirichlet Form on the Wasserstein Space}
\begin{subsection}{Tangent Spaces and Gradients}
The aim of this chapter is to construct a canonical Dirichlet form on the $L^2$-Wasserstein space $\Pe_0$.
Due to the isometry $\chi:\G_0\to\Pe_0$ this is equivalent to construct a canonical Dirichlet form on the metric space
$(\G_0, \|.\|_{L^2})$. This can be realized in two geometric settings which seem to be completely different:
\begin{itemize}
\item
Like in the preceding two chapters, $\G_0$ can be considered as a group, with composition of functions as group operation.
The tangent space $T_g\G_0$ is the closure (w.r.t. some norm) of the space of smooth functions $\varphi:[0,1]\to \R$ with $\varphi(0)=\varphi(1)=0$.
Such a function $\varphi$ induces a flow on $\G_0$ by
$(g,t)\mapsto e_{t\varphi}\circ g\approx g+t\,\varphi\circ g$ and it defines a directional derivative
 by
$D_\varphi u(g)=\lim_{t\to0}\frac1t[u(e_{t\varphi}\circ g)-u(g)]$
for $u:\G_0\to\R$.
The norm on $T_g\G_0$ we now choose
to be
$\|\varphi\|_{T_g}:=(\int\varphi(g_s)^2ds)^{1/2}.$
That is, $$T_g\G_0:=L^2([0,1], g_*\leb).$$
For given $u$ and $g$ as above, a gradient $Du(g)\in T_g\G_0$ exists with
$$D_\varphi u(g)=\langle Du(g),\varphi\rangle_{T_g}\qquad (\forall \varphi\in T_g)$$
if and only if
$\sup_{\varphi}\frac{D_\varphi u(g)}{\|\varphi\circ g\|_{L^2}}<\infty$.

\item
Alternatively, we can regard $\G_0$ as a closed subset of the space $L^2([0,1],\leb)$. The linear structure of
the latter (with the pointwise addition of functions as group operation)
suggests to choose as tangent space
 $$\T_g\G_0:=L^2([0,1], \leb).$$
 An element $f\in\T_g\G_0$
 induces a flow by
$(g,t)\mapsto  g+tf$ and it defines a directional derivative ({\em 'Frechet derivative'})
 by
$\D_f u(g)=\lim_{t\to0}\frac1t[u(g+tf)-u(g)]$
for $u:\G_0\to\R$, provided $u$ extends to a neighborhood of $\G_0$ in $L^2([0,1],\leb)$ or the flow (induced by $f$)
stays within $\G_0$.
A gradient $\D u(g)\in \T_g\G_0$ exists with
$$\D_f u(g)=\langle \D u(g),f\rangle_{L^2}\qquad (\forall \varphi\in L^2)$$
if and only if
$\sup_{f}\frac{\D_f u(g)}{\|f\|_{L^2}}<\infty$.
In this case, $\D u(g)$ is the usual $L^2$-gradient.
\end{itemize}
Fortunately, both geometric settings lead to the same result.
\begin{lemma}
(i) For each $g\in\G_0$, the map $\iota_g:\varphi\mapsto \varphi\circ g$ defines an isometric embedding
of $T_g\G_0=L^2([0,1], g_*\leb)$ into $\T_g\G_0=L^2([0,1], \leb)$. For each (smooth)
cylinder function $u:\G_0\to\R$
$$D_\varphi u(g)=\D_{\varphi\circ g}u(g).$$
If $\D u\in L^2(\leb)$ exists  then $Du\in L^2(g_*\leb)$ also exists.

(ii) For $\Q_0$-a.e. $g\in\G_0$, the above map $\iota_g: T_g\G_0\to\T_g\G_0$ is even bijective.
For each $u$ as above
$D u(g)=\D u(g) \circ g^{-1}$
and
$$\|D u(g)\|_{T_g}=\|\D u(g)\|_{\T_g}.$$
\end{lemma}

\begin{proof}
(i) is obvious, (ii) follows from the fact that for  $\Q_0$-a.e. $g\in\G_0$ the generalized inverse $g^{-1}$ is continuous
and thus $g^{-1}(g_t)=t$ for all $t$ (see sections 3.5 and 2.1). Hence,
the  map $\iota_g: T_g\G_0\to\T_g\G_0$ is surjective: for each $f\in \T_g\G_0$
$$\iota_g(f\circ g^{-1})=f\circ  g^{-1}\circ g=f.$$
\end{proof}

\begin{example}\rm\label{ex-grad}
(i) For each $u\in \Zyl^1(\G_0)$ of the form $u(g)=U(\int_0^1 \vec \alpha (g_t)dt)$ with $U\in \C^1(\R^m,\R)$ and
$\vec \alpha=(\alpha_1,\ldots,\alpha_m)\in \C^1([0,1],\R^m)$,
the gradients $Du(g)\in T_g\G_0=L^2(\eI, g_*\leb)$ and $\D u(g)\in \T_g\G_0=L^2(\eI, \leb)$ exist:
$$\D u(g)=\sum_{i=1}^m \partial_i U(\smint \vec \alpha (g_t)dt)\cdot \alpha_i'(g(.)),\qquad
D u(g)=\sum_{i=1}^m \partial_i U(\smint \vec \alpha (g_t)dt)\cdot \alpha_i'(.)$$
and their norms coincide:
$$\|D u(g)\|_{T_g}^2=\|\D u(g)\|_{\T_g}^2=
\int_0^1\left|\sum_{i=1}^m \partial_i U(\smint \vec \alpha (g_t)dt)\cdot \alpha_i'(g(s))\right|^2ds.$$
(ii) For each $u\in \Cyl^1(\G_0)$ of the form $u(g)=U(\int_0^1 \vec f(t) g(t)dt)$ with $U\in \C^1(\R^m,\R)$ and
$\vec f=(f_1,\ldots,f_m)\in L^2([0,1],\R^m)$,
the gradient
$$\D u(g)=\sum_{i=1}^m \partial_i U(\smint \vec f(t)  g(t)dt)\cdot \alpha_i(.)\ \in\  L^2(\eI,\leb)$$
exists
and
$$\|\D u(g)\|_{\T_g}^2=
\int_0^1\left|\sum_{i=1}^m \partial_i U(\smint \vec f(t) g(t)dt)\cdot f_i(s)\right|^2ds.$$
\end{example}

For $u\in\Cyl^1(\G_0)\cup\Zyl^1(\G_0)$, the gradient $\D u$ can be regarded as a map $\G_0\times\eI\to\R$,
$(g,t)\mapsto \D u(g)(t)$.
More precisely,
$$\D: \ \Cyl^1(\G_0)\cup\Zyl^1(\G_0)\ \to \ L^2(\G_0\times\eI, \Q_0\otimes\leb).$$

\begin{proposition}\label{close-d}
The operator
$\D: \ \Zyl^1(\G_0)\ \to \ L^2(\G_0\times\eI, \Q_0\otimes\leb)$
is closable in $L^2(\G_0,\Q_0)$.
\end{proposition}

\begin{proof}
Let $W\in L^2(\G_0\times\eI, \Q_0\otimes\leb)$ be
of the form
$W(g)=w(g)\cdot \varphi(g_t)$ with some $w\in\Zyl^1(\G_0)$  and some
$\varphi\in\C^\infty(\eI)$ satisfying $\varphi(0)=\varphi(1)=0$.
Then
 according to the integration by parts formula
 for each $u\in \Zyl^1(\G_0)$ with $u(g)=U(\int_0^1 \vec \alpha (g_s)ds)$
\begin{eqnarray*}
\int_{\G_0\times \eI} \D u\cdot W\, d(\Q_0\otimes\leb)&=&
\int_{\G_0}\int_0^1 \sum_{i=1}^m \partial_i U(\smint \vec \alpha (g_s)ds)\alpha'_i(g_t)
w(g)\varphi(g_t)dtd\Q_0(g)\\
&=&
\int_{\G_0} D_{\varphi}u(g) w(g)\, d\Q_0(g)=
\int_{\G_0} u(g) D^*_{\varphi} w(g)\, d\Q_0(g).
\end{eqnarray*}
To prove the closability of $\D$, consider a sequence $(u_n)_n$ in $\Zyl^1(\G_0)$ with $u_n\to0$ in
$L^2(\Q_0)$ and $\D u_n\to V$ in $L^2(\Q_0\otimes\leb)$. Then
\begin{eqnarray}\label{int-d}
\int V\cdot W\, d(\Q_0\otimes\leb)=\lim_n\int \D u_n \cdot W\, d(\Q_0\otimes\leb)
=\lim_n\int u_nD_\varphi^*  w\, d\Q_0=0
\end{eqnarray}
for all $W$ as above. The linear hull of the latter is dense in $L^2(\Q_0\otimes\leb)$.
Hence, (\ref{int-d}) implies $V=0$ which proves the closability of $\D$.
\end{proof}

The closure of $(\D, \Zyl^1(\G_0))$ will be denoted by $(\DD,\Dom(\DD)$. Note that a priori it is not clear whether $\DD$ coincides
with $\D$ on $\Cyl^1(\G_0)$. (See, however, Theorem \ref{dir-form-cyl} below.)

\end{subsection}

\begin{subsection}{The Dirichlet Form}
\begin{definition}
For $u,v\in \Zyl^1(\G_0)\cup \Cyl^1(\G_0)$ we define the 'Wasserstein Dirichlet integral'
\begin{equation}\label{Wasserstein-dir-int}
\EE(u,v)=\int_{\G_0} \langle\D u(g),  \D v(g)\rangle_{L^2} \, d\Q_0(g).
\end{equation}
\end{definition}

\begin{theorem}\label{wasserstein-dir-form}
{\bf (i)} \ $(\EE,\Zyl^1(\G_0))$ is closable. Its closure $(\EE,
\Dom(\EE))$ is a regular, recurrent Dirichlet form on $L^2(\G_0,\Q_0)$.

$\Dom(\EE)=\Dom(\DD)$ and for all $u,v\in\Dom(\DD)$
$$\EE(u,v)=\int_{\G_0\times\eI}\DD u\cdot \DD v\, \ d(\Q_0\otimes\leb).$$

{\bf (ii)} \ The set $\Zyl^\infty_0(\G_0)$ of all cylinder functions $u\in \Zyl^\infty(\G_0)$
of the form $u(g)=U(\int\vec\alpha(g_s)ds)$ with
$U\in\C^\infty(\R^m,\R)$ and $\vec\alpha=(\alpha_1,\ldots,\alpha_m)\in\C^\infty(\eI, \R^m)$ satisfying $\alpha_i'(0)=\alpha_i'(1)=0$
is a core for $(\EE, \Dom(\EE))$.

{\bf (iii)} \
The generator $(\LL, \Dom(\LL)$ of $(\EE, \Dom(\EE))$ is the Friedrichs  extension of the operator
$(\LL, \Zyl^\infty_0(\G_0)$ given  by
\begin{eqnarray*}
\LL u(g)&=&
-\sum_{i=1}^m D_{\alpha_i}^*u_i(g)\\
&=&
\sum_{i,j=1}^m{\partial_i\partial_j}U\left(\smint\vec\alpha(g_s)ds\right)\cdot\int_0^1\alpha_i'(g_s)\alpha_j'(g_s)ds
\ +
\
\sum_{i=1}^m{\partial_i}U\left(\smint\vec\alpha(g_s)ds\right)\cdot
V^\beta_{\alpha_i'}(g)
\end{eqnarray*}
where
$u_i(g):=\partial_i U(\smint \vec\alpha(g_s) ds)$ and
 $V_{\alpha_i'}^\beta(g)$ denotes the drift term defined in section 5.1 with $\varphi=\alpha_i'$;
 $\beta>0$ is the parameter of the entropic measure fixed throughout the whole chapter.

{\bf (iv)} \ The Dirichlet form $(\EE,\Dom(\EE))$ has a square field operator given by
$$\Gamma(u,v):=\langle \DD u, \DD v\rangle_{L^2(Leb)}\quad \in L^1(\G_0,\Q_0)$$
with $\Dom(\Gamma)=\Dom(\EE)\cap L^\infty(\G_0,\Q_0)$. That is, for all $u,v,w\in
\Dom(\EE)\cap L^\infty(\G_0,\Q_0)$
\begin{equation}
\label{gamma-op}
2\int w\cdot \Gamma(u,v)\, d\Q_0=\EE(u,vw)+\EE(uw,v)-\EE(uv,w).
\end{equation}
\end{theorem}

\begin{proof}{\bf (a)} \
The closability of the form $(\EE,\Zyl^1(\G_0))$ follows immediately from the previous Proposition
\ref{close-d}. Alternatively, we can deduce it from assertion (iii) which we are going to prove first.

{\bf (b)} \ Our first claim is that $\EE(u,w)=-\int u \cdot \LL w\, d\Q_0$ for all $u,w\in \Zyl_0^\infty(\G_0)$.
Let $u(g)=U(\int \vec\alpha(g_s) ds) $ and
$w(g)=W(\int \vec\gamma(g_s) ds) $
with
$U, W\in\C^\infty(\R^m,\R)$ and $\vec\alpha=(\alpha_1,\ldots,\alpha_m), \vec\gamma=(\gamma_1,\ldots,\gamma_m)
\in\C^\infty(\eI, \R^m)$ satisfying $\alpha_i'(0)=\alpha_i'(1)=\gamma_i'(0)=\gamma_i'(1)=0$.
Observe that
\begin{eqnarray*}\langle \D u(g),\D w(g)\rangle_{L^2}&=&\sum_{i,j=1}^m \partial_i U(\smint \vec\alpha(g_s) ds)\cdot
\partial_j W(\smint \vec\gamma(g_s) ds)\cdot \int_0^1 \alpha'_i(g_s)\gamma'_j(g_s)ds\\
&=&
\sum_{i=1}^m u_i(g)\cdot D_{\alpha'_i}w(g).
\end{eqnarray*}
Hence, according to the integration by parts formula from Proposition \ref{adjoint-i}
\begin{eqnarray*}
\EE(u,w)&=&\int_{\G_0}\langle \D u(g),\D w(g)\rangle_{L^2}\, d\Q_0(g)\\
&=&\sum_{i=1}^m \int_{\G_0}u_i(g)\cdot D_{\alpha'_i}w(g)\, d\Q_0(g)\\
&=&\sum_{i=1}^m \int_{\G_0} D^*_{\alpha'_i}u_i(g) \cdot w(g)\, d\Q_0(g)\\
&=&-\int_{\G_0} \LL u(g)\cdot w(g)\, d\Q_0(g).
\end{eqnarray*}
This proves our first claim. In particular, $(\LL,\Zyl_0^\infty(\G_0)) $ is a symmetric operator.
Therefore, the form $(\EE,\Zyl_0^\infty(\G_0))$ is closable and its generator coincides with the Friedrichs extension of
$\LL$.

{\bf (c)} \ Now let us prove that $\Zyl_0^\infty(\G_0)$ is dense in $\Zyl^1(\G_0)$. That is, let us
 prove that each function $u\in\Zyl^1(\G_0)$ can be approximated by functions $u_\epsilon\in\Zyl^\infty_0(\G_0)$.
For simplicity, assume that $u$ is of the form $u(g)=U(\int \alpha(g_s) ds) $ with $U\in C^1(\R)$ and $\alpha
\in C^1(\eI)$. (That is, for simplicity, $m=1$.)
Let $U_\epsilon\in C^\infty(\R)$ for $\epsilon>0$ be smooth approximations of $U$ with
$\|U-U_\epsilon\|_\infty+\|U'-U'_\epsilon\|_\infty\to0$ as $\epsilon\to0$ and let
 $\alpha_\epsilon\in C^\infty(\R)$  with $\alpha'_\epsilon(0)=\alpha'_\epsilon(1)=0$
 be smooth approximations of $\alpha$ with
$\|\alpha-\alpha_\epsilon\|_\infty\to0$  and
$\alpha'_\epsilon(t)\to\alpha'(t)$ for all $t\in]0,1[$ as $\epsilon\to0$.
Moreover, assume that $\sup_\epsilon\|\alpha'\|_\infty<\infty$.

Define $u_\epsilon\in\Zyl^\infty_0(\G_0)$ as
$u_\epsilon(g)=U_\epsilon(\int \alpha_\epsilon(g_s) ds) $.
Then
  $u_\epsilon \to u$ in $L^2(\G_0,\Q_0)$ by dominated convergence relative $\Q_0$.

   Since
\[ \sup_\epsilon \sup_{g\in \G}
\bigl(U'_\epsilon(\smint \alpha_\epsilon(g(s)) ds)\bigr)^2 \smint_{\eI} \alpha_\epsilon '(g_s)^2 ds \leq C,\]
\[ (\alpha_\epsilon ')^2(g(s))\stackrel{\epsilon \to 0}{\longrightarrow}  \alpha'(g_s)^2 \quad \forall s \in \eI\setminus\bigl(  \{g=0\}\cap \{g=1\}\bigr),\]
and
\[
\eI\setminus\bigl(  \{g=0\}\cap \{g=1\}\bigr) =]0,1[ \mbox{ for } \Q_0\mbox{-almost all } g \in \G_0\]
one finds by dominated convergence in $L^2(\eI, \leb)$, for $ \Q_0\mbox{-almost all } g \in \G_0$
\[ \bigl(U'_\epsilon(\smint \alpha_\epsilon(g_s) ds)\bigr)^2 \smint_{\eI} \alpha_\epsilon '(g_s)^2 ds
\stackrel{\epsilon \to 0}{\longrightarrow}  \bigl(U'(\smint \alpha (g_s) ds)\bigr)^2 \smint_{\eI} \alpha'(g_s)^2 ds.  \]
Hence also with
\begin{eqnarray*}
\EE(u_\epsilon,  u_\epsilon) &=&
\int_{\G_0}  \bigl(U'_\epsilon(\smint \alpha_\epsilon(g_s) ds)\bigr)^2\cdot \smint \alpha_\epsilon '(g_s)^2ds \Q_0(dg)\\
& \stackrel{\epsilon \to 0}{\longrightarrow}&
 \int_{\G_0} \bigl(U'(\smint \alpha(g_s) ds)\bigr)^2 \cdot\smint \alpha  '(g_s)^2ds \Q_0(dg)
 \end{eqnarray*}
by  dominated convergence in $L^2(\G_0, \Q_0)$. In particular,   $\{u_\epsilon\}_\epsilon$ constitutes a
 Cauchy sequence relative to the norm   $\|v\|_{\EE,1}^2 :=\|v\|_{L^2(\G,\Q)} ^2 + \EE(v,v)$.
 In fact, since the sequence $u_\epsilon $ is uniformly bounded w.r.t.  to $\|.\|_{\EE,1}$,
 by weak compactness there is a weakly converging subsequence in $(\Dom(\EE), \|.\|_{\EE,1})$.
 Since the associated norms converge, the convergence is actually strong in $(\Dom(\EE), \|.\|_{\EE,1})$.
 Moreover, since $u_\epsilon \to u $ in $L^2(\G_0, \Q_0)$, this limit is unique.
 Hence the entire sequence converges to $u \in (\Dom(\EE),\|.\|_{\EE,1})$,
 such that in particular $\EE(u,u) = \lim_{\epsilon \to 0}\EE(u_\epsilon,u_\epsilon)$.

 This proves our second claim. In particular,
it implies that also
$(\EE,\Zyl^1(\G_0))$ is closable and that
 the closures of
 $\Zyl_0^\infty(\G_0)$ and $\Zyl^1(\G_0)$ coincide.

{\bf (d)} \ Obviously, $(\EE, \Dom(\EE))$ has the Markovian property. Hence, it is a Dirichlet form.
Since the constant functions belong to $\Dom(\EE)$, the form is recurrent.
Finally, the set $\Zyl^1(\G_0)$ is dense in $(\C(\G_0),\|.\|_\infty)$ according to the theorem of Stone-Weierstrass
since it separates the points in the compact metric space $\G_0$. Hence,  $(\EE, \Dom(\EE))$ is regular.

{\bf (e)} \  According to Leibniz' rule, (\ref{gamma-op}) holds true for all $u,v,w\in\Zyl^1(\G_0)$.
Arbitrary  $u,v,w\in
\Dom(\EE)\cap L^\infty(\G_0,\Q_0)$ can be approximated in $(\EE(.)+\|.\|^2)^{1/2}$ by $u_n,v_n,w_n\in\Zyl^1(\G_0)$
which are uniformly bounded on $\G_0$.
Then
$u_nv_n\to uv$, $u_nw_n\to uw$ and $v_nw_n\to vw$ in $(\EE(.)+\|.\|^2)^{1/2}$.
Moreover, we may assume that $w_n\to w$ $\Q_0$-a.e. on $\G_0$ and thus
$$\int\left|
w\Gamma(u,v)-w_n\Gamma(u_n,v_n)\right|d\Q_0\le
\int|w-w_n|\Gamma(u,v)d\Q_0+\int|w_n|\cdot|\Gamma(u,v)-\Gamma(u_n,v_n)|d\Q_0\to0$$
by dominated convergence. Hence, (\ref{gamma-op}) carries over from $\Zyl^1(\G_0)$ to
$\Dom(\EE)\cap L^\infty(\G_0,\Q_0)$.
\end{proof}

\begin{lemma}\label{l2-diff} For each $f\in \G_0$ the function $u:g\mapsto \langle f,g\rangle_{L^2}$ belongs to $\Dom(\EE)$.
\end{lemma}

\begin{proof}
{\bf (a)} For $f,g\in \G_0$ put $\mu_f=f_*\leb$ and $\mu_g=g_*\leb$. Recall that by Kantorovich duality
\begin{eqnarray*}
\frac12\|f-g\|_{L^2}^2&=&\frac12 d_W^2(\mu_f,\mu_g)\\
&=&
\sup_{\varphi,\psi}\left\{\int_0^1 \varphi d\mu_f+\int_0^1 \psi d\mu_g\right\}=
\sup_{\varphi,\psi}\left\{\int_0^1 \varphi(f_t) dt+\int_0^1 \psi(g_t) dt\right\}
\end{eqnarray*} where the $\sup_{\varphi,\psi}$ is taken over all (smooth, bounded) $\varphi\in L^1(\eI,\mu_f)$,
$\psi\in L^1(\eI,\mu_g)$ satisfying
$\varphi(x)+\psi(y)\le \frac12 |x-y|^2$
for $\mu_f$-a.e. $x$ and $\mu_g$-a.e. $y$ in $\eI$.
Replacing $\varphi(x)$ by $|x|^2/2-\varphi(x)$ (and $\psi(y)$ by $\ldots$) this can be restated as
\begin{eqnarray}\label{duality}\langle f, g\rangle_{L^2}=\inf_{\varphi,\psi}\left\{\int_0^1 \varphi(f_t) dt+\int_0^1 \psi(g_t) dt\right\}
\end{eqnarray} where the $\inf_{\varphi,\psi}$ now is taken over all (smooth, bounded) $\varphi\in L^1(\eI,\mu_f)$,
$\psi\in L^1(\eI,\mu_g)$ satisfying
$\varphi(x)+\psi(y)\ge \langle x,y\rangle$
for $\mu_f$-a.e. $x$ and $\mu_g$-a.e. $y$ in $\eI$.
If $g$ is strictly increasing then $\psi$ can be chosen as
$$\psi'=f\circ g^{-1},$$
cf. \cite{MR1964483}, sect. 2.1 and 2.2.

{\bf (b)} Now fix a countable dense set $\{g_n\}_{n\in\N}$ of strictly increasing functions in $\G_0$ and an arbitrary
function $f\in\G_0$.
Let $(\varphi_n,\psi_n)$ denote a minimizing pair for $(f,g_n)$ in (\ref{duality}) and define $u_n:\G_0\to\R$ by
$$u_n(g):=\min_{i=1,\ldots,n}\left\{\int_0^1 \varphi(f_i(t)) dt+\int_0^1 \psi_i(g(t)) dt\right\}.$$
Note that $\psi'_i=f\circ g_i^{-1}$ and thus
$u_n(g_i)=\langle f,g_i\rangle$ for all $i=1,\ldots,n$.
Therefore,
\begin{eqnarray*}
|u_n(g)-u_n(\tilde g)|&\le&
\max_i \int_0^1 |\psi_i(g(t))dt-\psi_i(\tilde g(t))|dt
\le \max_i\|\psi_i'\|_\infty\cdot \int_0^1 |g(t)-\tilde g (t)|dt\le
   \|g-\tilde g\|_{L^1}
\end{eqnarray*}
for all $g,\tilde g\in\G_0$. Hence, $u_n\to u$ pointwise on $\G_0$ and in $L^2(\G_0,\Q_0)$ where $u(g):=\langle f,g\rangle$.

{\bf (c)} The function $u_n$ is in the class $\Zyl^0(\G_0)$:
$$u_n(g)=U_n\left( \smint\vec\alpha(g_t)dt\right)$$
with $U_n(x_1,\ldots,x_n)=\min\{c_1+x_1,\ldots,c_n+x_n\}$, $c_i=\int\varphi_i(f(t))dt$ and
$\alpha_i=\psi_i$. The function $U_n$ can be easily approximated by $\C^1$ functions in order to verify that
$u_n\in\Dom(\EE)$ and
$$\DD u_n(g)=\sum_{i=1}^n 1_{A_i}(g)\cdot \psi_i'(g(.))$$
with a suitable disjoint decomposition $\G_0=\cup_i A_i$. (More precisely,
 $A_i$ denotes the set of all $g\in\G_0$ satisfying
$\int_0^1 \varphi(f_i(t)) dt+\int_0^1 \psi_i(g(t)) dt
<\int_0^1 \varphi(f_j(t)) dt+\int_0^1 \psi_j(g(t)) dt$ for all $j<i$ and
$\int_0^1 \varphi(f_i(t)) dt+\int_0^1 \psi_i(g(t)) dt
\le\int_0^1 \varphi(f_i(t)) dt+\int_0^1 \psi_i(g(t)) dt$ for all $j>i$.)
Thus
$$\|\DD u_n(g)\|^2=\sum_i1_{A_i}(g)\cdot \int_0^1\psi_i'(g(t))^2dt$$
and $$
\EE(u_n)\le \max_{i\le n} \int_{\G_0} \|\psi_i'\circ g\|_{L^2}^2d\Q_0(g).$$
In particular, since $|\psi_i'|\le 1$,
$$\sup_n\EE(u_n)\le1$$ and thus $u\in\Dom(\EE)$.
\end{proof}

\begin{lemma}\label{grad-repr}
For all $u\in\Zyl^1(\G_0)$ and all $w\in \Cyl^1(\G_0)\cap \Dom(\EE)$
\begin{equation}\label{e=d}\EE(u,w)=\int_{\G_0} \langle\D u(g),\D w(g)\rangle_{L^2} d\Q_0(g)
\end{equation}
(with $\D u(g)$ and $\D w(g)$ given explicitly as in Example \ref{ex-grad}).
\end{lemma}

\begin{proof}
Recall that for $u\in\Zyl^\infty_0(\G_0)$ of the form $u(g)=U(\smint\vec\alpha(g_t)dt)$
$$L u(g)=\sum_{i=1}^m D^*_{\alpha_i'} u_i(g)$$
with $u_i(g)=\partial_iU(\smint\vec\alpha(g_t)dt)$.
Hence, for
$w\in \Cyl^1(\G_0)$ of the form $w(g)=W(\langle \vec h, g\rangle)$
\begin{eqnarray*}
\EE(u,w)&=&
-\int \LL u(g)\, w(g)\, d\Q_0(g)\\
&=&\sum_{i=1}^m \int_{\G_0} D^*_{\alpha_i'} u_i(g)\, w(g)\, d\Q_0(g)=
\sum_{i=1}^m \int_{\G_0} u_i(g)\,  D_{\alpha_i'} u_i(g) w(g)\, d\Q_0(g)\\
&=&\sum_{i,j=1}^m\int_{\G_0}\partial_iU(\smint\vec\alpha(g_t)dt)\cdot
\partial_j W(\smint\vec h(t) g(t)dt)\cdot \smint \alpha'_i(g(t))h_j(t)dt\  d\Q_0(g)\\
&=&\int_{\G_0}  \langle \D u(g), \D w(g)\rangle d\Q_0(g).
\end{eqnarray*}
This proves the claim provided $u\in\Zyl^\infty_0(\G_0)$. By density this extends to all $u\in\Zyl^1(\G_0)$.
\end{proof}

\begin{theorem}\label{dir-form-cyl}
{\bf (i)} \ $(\EE,\Cyl^1(\G_0))$ is closable and its closure coincides with $(\EE,\Dom(\EE))$.
Similarly, $(\D,\Cyl^1(\G_0))$ is closable and its closure coincides with $(\DD,\Dom(\DD))$.

{\bf (ii)} \
For all $u,w\in\Zyl^1(\G_0)\cup \Cyl^1(\G_0)$
\begin{equation}\label{ee=dd}
\Gamma(u,w)(g)= \langle\D u(g),\D w(g)\rangle_{L^2},
\end{equation}
in particular,
$\EE(u,w)=\int_{\G_0} \langle\D u(g),\D w(g)\rangle_{L^2} d\Q_0(g)$
(with $\D u(g)$ and $\D w(g)$ given explicitly as in Example \ref{ex-grad}).

{\bf (iii)} \
For each $f\in\G_0$ the function $u_f: g\mapsto \|f-g\|_{L^2}$ belongs to $\Dom(\E)$ and
$\Gamma(u_f,u_f)\le 1$ $\Q_0$-a.e. on $\G_0$.

{\bf (iv)} \
$(\EE,\Dom(\EE))$ is strongly local.
\end{theorem}

\begin{proof}
{\bf (a)} Claim: \ {\it
 For each $f\in L^2(\eI,\leb)$ the function $u_f:g\mapsto \langle f,g\rangle_{L^2}$ belongs to $\Dom(\EE)$
 and $\EE(u_f,u_f)=\|f\|_{L^2}^2$.}

Indeed, if $f\in L^2\cap\C^1$ then $f=c_0+c_1 f_1+c_2f_2$ with $f_1,f_2\in \G_0$ and $c_0,c_1,c_2\in\R$.
Hence, $u_f\in\Dom(\EE)$ according to Lemma \ref{l2-diff}
and $\EE(u_f,u_f)=\int\|\D u_f\|^2d\Q_0=\|f\|^2$ according to Lemma \ref{grad-repr}.
Finally, each $f\in L^2$ can be approximated by $f_n\in L^2\cap\C^1$ with $\|f-f_n\|\to0$.
Hence, $u_f\in\Dom(\EE)$
and $\EE(u_f,u_f)=\|f\|^2$.

{\bf (b)} Claim: \ {\it $\C^1(\G_0)\subset\Dom(\EE)$.}

Let $u\in\C^1(\G_0)$ be given with $u(g)=U(\langle \vec f,g\rangle)$, $U\in\C^1(\R^m,\R)$, $\vec f=(f_1,\ldots,f_m)\in L^2(\eI,\R^m)$.
For each $i=1,\ldots,m$ let $(w_{i,n})_{n\in\N}$ be an approximating sequence in $(\Zyl^1(\G_0),(\EE+\|.\|^2)^{1/2})$
for $w_i: g\mapsto \langle f_i,g\rangle$.
Put $u_n(g)=U(w_{1,n}(g),\ldots, w_{m,n}(g))$. Then $u_n\in\Zyl^1(\G_0)$, $u_n\to u$ pointwise on
$\G_0$ and in $L^2(\G_0,\Q_0)$. Moreover,
\begin{eqnarray*}
\EE(u_n,u_n)&=&
\int \| \sum_i \partial_i U(w_{1,n}(g),\ldots, w_{m,n}(g)) \D w_{i,n}(g)\|^2_{L^2}\ d\Q_0(g)\\
&\to &
\int \| \sum_i \partial_i U(\langle f_1,g\rangle, \ldots,\langle f_m,g\rangle) \D w_{i}(g)\|^2_{L^2}\ d\Q_0(g)\\
&= &
\int\|\D u(g)\|^2d\Q_0(g).
\end{eqnarray*}
Hence, $u\in\Dom(\EE)$ and $\EE(u,u)=\int\|\D u(g)\|^2d\Q_0(g)$.

{\bf (c)}
Assertion (ii) then follows via polarization and bi-linearity.
Assertion (iii) is an immediate consequence of assertion (ii).
Assertion (iii) allows to prove the locality of the Dirichlet form $(\EE,\Dom(\EE))$ in the same manner
as in the proof of Theorem \ref{dir-form}.

{\bf (d)} Claim: \ {\it
 $\Cyl^1(\G_0)$ is dense in $\Dom (\EE)$.}

We have to prove that each $u\in\Zyl^1(\G_0)$ can be approximated by $u_n\in\Cyl^1(\G_0)$. As usual, it suffices to treat
the particular case $u(g)=\int_0^1\alpha(g_t)dt$ for some $\alpha\in\C^1(\eI)$.
Put $U_n(x_1,\ldots,x_n)=\frac1n\sum_{i=1}^n\alpha(x_i)$ and
$f_{n,i}(t)=n\cdot 1_{[\frac{i-1}n,\frac in[}(t)$. Then
$$u_n(g):=U_n(\langle f_{n,1},g\rangle,\ldots \langle f_{n,n},g\rangle)=\frac1n \sum_{i=1}^n \alpha\left( n\int_{\frac{i-1}n}^{\frac in}g_tdt\right)$$
defines a sequence in $\Cyl^1(\G_0)$ with
$u_n(g)\to u(g)$ pointwise on $\G_0$ and in $L^2(\G_0,\Q_0)$.

Moreover,
\begin{equation}
\D u_n(g)=\sum_{i=1}^n \alpha'\left( n\int_{\frac{i-1}n}^{\frac in}g_tdt\right)\cdot1_{[\frac{i-1}n,\frac in[}(.)
\end{equation}
and therefore
\begin{equation}
\EE(u_n)=\int_{\G_0}\frac1n\sum_{i=1}^n \alpha'\left( n\int_{\frac{i-1}n}^{\frac in}g_tdt\right)^2
d\Q_0(g)\longrightarrow\int_{\G_0}\int_0^1\alpha'(g_t)^2dtd\Q_0(g)=\EE(u).
\end{equation}
Thus $(u_n)_n$ is Cauchy in $\Dom(\EE)$ and $u_n\to u$ in $\Dom(\EE)$.

\end{proof}
\end{subsection}

\begin{subsection}{Rademacher Property and Intrinsic Metric}

We say that a function $u:\G_0\to\R$ is 1-Lipschitz if
$$|u(g)-u(h)|\le \|g-h\|_{L^2}\qquad (\forall g,h\in\G_0).$$
\begin{theorem}\label{lip-is-diff}
Every 1-Lipschitz function $u$ on $\G_0$ belongs to $\Dom(\EE)$ and $\Gamma(u,u)\le 1$ $\Q_0$-a.e. on $\G_0$.
\end{theorem}

Before proving the theorem in full generality, let us first consider the following particular case.

\begin{lemma} Given $n\in\N$, let $\{h_1,\ldots,h_n\}$ be a orthonormal system in $L^2(\eI,\leb)$ and let $U$ be a 1-Lipschitz function on $\R^n$.
Then the function $u(g)=U(\langle h_1,g\rangle,\ldots,\langle h_n,g\rangle)$ belongs
to $\Dom(\EE)$ and $\Gamma(u,u)\le 1$ $\Q_0$-a.e. on $\G_0$.
\end{lemma}

\begin{proof}
Let us first assume that in addition $U$ is $\C^1$. Then according to Theorem \ref{dir-form-cyl}, $u$ is in
$\Dom(\EE)$ and
$\D u(g)=\sum_{i=1}^n\partial_i U(\langle \vec h,g\rangle)\cdot h_i$.
Thus
$$\Gamma(u,u)(g)=\|\D u(g)\|_{L^2}=\sum_{i=1}^n |\partial_i U(\langle\vec h,g\rangle)|^2\le 1.$$
In the case of a general 1-Lipschitz continuous $U$ on $\R^n$ we choose an approximating sequence of 1-Lipschitz functions $U_k$, $k\in\N$, in $\C^1(\R^n)$
with $U_k\to U$ uniformly on $\R^n$ and put $u_k(g)=U_k((\langle\vec h,g\rangle)$ for $g\in\G_0$.
Then $u_k\to u$ pointwise and in $L^2(G_0,\Q_0)$. Hence, $u\in\Dom(\EE)$ and
$\Gamma(u,u)\le 1$ $\Q_0$-a.e. on $\G_0$.
\end{proof}

\begin{proof}[Proof of Theorem \ref{lip-is-diff}]
Every 1-Lipschitz function $u$ on $\G_0$ can be extended to a 1-Lipschitz function $\tilde u$ on $L^2(\eI,\leb)$
('Kirszbraun extension').
Hence, without restriction, assume that $u$ is a 1-Lipschitz function on $L^2(\eI,\leb)$.
Choose a complete orthonormal system $\{h_i\}_{i\in\N}$ of the separable Hilbert space  $L^2(\eI,\leb)$ and define for each $n\in\N$
the function $U_n:\R^n\to\R$ by
$$U_n(x_1,\ldots, x_n)=u\left(\sum_{i=1}^n x_ih_i\right)$$
for $x=(x_1,\ldots,x_n)\in\R^n$.
This function $U_n$ is 1-Lipschitz on $\R^n$:
$$\left|U_n(x)-U_n(y)\right|\le
\left\|\sum_{i=1}^n x_ih_i-\sum_{i=1}^n y_ih_i\right\|_{L^2}\le
|x-y|.$$
Hence, according to the previous Lemma
the function $$u_n(g)=U_n(\langle h_1,g\rangle,\ldots,\langle h_n,g\rangle)$$ belongs belongs to $\Dom(\EE)$ and $\Gamma(u_n,u_n)\le 1$ $\Q_0$-a.e. on $\G_0$.

\smallskip

Note that
$$u_n(g)=u\left(\sum_{i=1}^n\langle h_i,g\rangle h_i\right)$$
for each $g\in L^2(\eI,\leb)$.
Therefore, $u_n\to u$ on $L^2(\eI,\leb)$ since $\sum_{i=1}^n\langle h_i,g\rangle h_i \to g$ on $L^2(\eI,\leb)$ and since $u$ is continuous on
$L^2(\eI,\leb)$.
Thus, finally, $u\in\Dom(\EE)$ and $\Gamma(u,u)\le 1$ $\Q_0$-a.e. on $\G_0$.
\end{proof}

Our next goal is the converse to the previous Theorem.

\begin{theorem}\label{invrademacher}
Every continuous function $u\in\Dom(\EE)$ with $\Gamma(u,u)\le 1$ $\Q_0$-a.e. on $\G_0$ is 1-Lipschitz on $\G_0$.
\end{theorem}

\begin{lemma}
For each $u\in\Cyl^1(\G_0)\cup \Zyl^1(\G_0)$ and all $g_0,g_1\in\G_0$
\begin{equation}\label{line-integral}
u(g_1)-u(g_0)=\int_0^1 \langle \D u\left((1-t)g_0+tg_1\right),g_1-g_0\rangle_{L^2}dt.
\end{equation}
\end{lemma}

\begin{proof} Put $g_t=(1-t)g_0+tg_1$ and consider the $\C^1$ function $\eta:[0,1]\to\R$ defined by
$\eta_t=u(g_t)$. Then $$\dot\eta_t=\D_{g_1-g_0}u(g_t)=\langle\D u(g_t),g_1-g_0\rangle$$
and thus
$$\eta_1-\eta_0=\int_0^1\dot\eta_t dt=\int_0^1 \langle\D u(g_t),g_1-g_0\rangle dt.$$
\end{proof}

\begin{lemma}
Let $g_0,g_1\in\G_0\cap\C^3$ and put $g_t=(1-t)g_0+t g_1$. Then for each $u\in\Dom(\EE)$ and each bounded measurable
$\Psi:\G_0\to\R$
\begin{equation}
\int_{\G_0}[ u(g_1\circ h)-u(g_0\circ h)]\Psi(h)\, d\Q_0(h)=
\int_0^1
\int_{\G_0} \langle\DD u(g_t\circ h, (g_1-g_0)\circ h\rangle \Psi(h)\, \Q_0(h)dt.
\end{equation}

\end{lemma}

\begin{proof}
Given $g_0,g_1$, $\Psi$ and $u\in\Dom(\EE)$ as above, choose an approximating sequence in $\Zyl^1(\G_0)\cup\Cyl^1(\G_0)$ with
$u_{n}\to u$ in $\Dom(\EE)$
as $n\to\infty$.
According to the previous Lemma for each $n$
\begin{equation}\label{path-int}
\int_{\G_0}[ u_{n}(g_1\circ h)-u_{n}(g_0\circ h)]\Psi(h)\, d\Q_0(h)
=
\int_0^1\int_{\G_0}
 \langle \D u_n\left(g_t\circ h\right),(g_1-g_0)\circ h\rangle\Psi(h)\, d\Q_0(h)dt.
\end{equation}
By assumption $u_n\to u$ in $L^2(\G_0,\Q_0)$ and
$\D u_n\to \DD u$ in $L^2(\G_0\times\eI, \Q_0\otimes \leb)$ as $n\to\infty$.
Using the quasi-invariance of $\Q_0$ (Theorem \ref{CoV-I}) this implies
$$
\int_{\G_0}| u(g_t\circ h)-u_n(g_t\circ h)|\Psi(h)\, d\Q_0(h)=
\int_{\G_0}[ u(h)-u_n(h)|\Psi(g_t^{-1}\circ h)\cdot Y^\beta_{g_t^{-1}}(h)\, d\Q_0(h)\to0$$
as $n\to\infty$ as well as
\begin{eqnarray*}
\lefteqn{\int_{\G_0} \|\DD u(g_t\circ h) -\D u_n(g_t\circ h)|^2_{L^2} \Psi(h)\, \Q_0(h)}\\
&=&
\int_{\G_0} \|\DD u( h) -\D u_n( h)|^2_{L^2} \Psi(g_t^{-1}\circ h)\cdot Y^\beta_{g_t^{-1}}(h)\, \Q_0(h)\to0
\end{eqnarray*}
Hence, we may pass to the
limit $n\to\infty$ in (\ref{path-int}) which yields the claim.
\end{proof}

\begin{proof}[Proof of  Theorem \ref{invrademacher}]
Let a continuous  $u\in\Dom(\EE)$ be given with $\Gamma(u,u)\le 1$ $\Q_0$-a.e. on $\G_0$.
We want to prove that
$u(g_1)-u(g_0)\le\|g_1-g_0\|_{L^2}$ for all $g_0,g_1\in\G_0$. By density of $\G_0\cap\C^3$ in $\G_0$ and by continuity of $u$
it suffices to prove the claim for $g_0,g_1\in\G_0\cap\C^3$.

Choose a sequence of bounded measurable $\Psi_k:\G_0\to\R_+$ such that the probability measures $\Psi_kd\Q_0$ on $\G_0$ converge weakly to
$\delta_{e}$, the Dirac mass in the identity map $e\in\G_0$.
Then according to the previous Lemma and the assumption  $\|\DD u\|\le1$
\begin{eqnarray*}
\lefteqn{\int_{\G_0}[ u(g_1\circ h)-u(g_0\circ h)]\Psi_k((h)d\Q_0(h)}\\
&=&
\int_0^1
\int_{\G_0} \langle\DD u(g_t\circ h, (g_1-g_0)\circ h\rangle \Psi_k(h)\, d\Q_0(h)dt\\
&\le&
\int_0^1
\int_{\G_0} \|\DD u(g_t\circ h)\|_{L^2}\cdot \|(g_1-g_0)\circ h\|_{L^2}\cdot \Psi_k(h)\, d\Q_0(h)dt\\
&\le&
\int_{\G_0}  \|(g_1-g_0)\circ h\|_{L^2}\cdot \Psi_k(h)\, d\Q_0(h).
\end{eqnarray*}
Now the integrands on both sides, $h\mapsto u(g_1\circ h)-u(g_0\circ h)$ as well as $h\mapsto\|(g_1-g_0)\circ h\|_{L^2}$,
are continuous in $h\in\G_0$. Hence, as $k\to\infty$ by weak convergence $\Psi_kd\Q_0\to\delta_e$ we obtain
$$u(g_1)-u(g_0)\le\|g_1-g_0\|_{L^2}.$$
\end{proof}

\begin{corollary} \label{identintrinsicmetric}
 The intrinsic metric for the Dirichlet form $(\EE,\Dom(\EE))$ is the $L^2$-metric:
$$\|g_1-g_0\|_{L^2}=\sup\left\{ u(g_1)-u(g_0): \ u \in\C(\G_0)\cap\Dom(\EE), \ \Gamma(u,u)\le 1\, \Q_0\mbox{-a.e. on }\G_0\right\}$$
 for all $g_0,g_1\in\G_0$.
\end{corollary}
\end{subsection}

\begin{subsection}{Finite Dimensional Noise Approximations}

The goal of this section is to present representations -- and finite dimensional approximations -- of the Dirichlet form
\begin{equation*}
\EE(u,v)=\int_{\G_0} \langle\D u(g),  \D v(g)\rangle_{L^2} \, d\Q_0(g)
\end{equation*}
in terms of globally defined vector fields.

If $(\varphi_i)_{i\in\N}$ is a complete orthonormal system in $T_g=L^2(\eI,g_*\leb)$ for a given $g\in\G_0$ then
obviously
\begin{equation}\label{on-repr}
\langle\D u(g),\D v(g)\rangle_{L^2}=\sum_{i=1}^\infty D_{\varphi_i}u(g)D_{\varphi_i}v(g).
\end{equation}
Unfortunately, however, there exists no family  $(\varphi_i)_{i\in\N}$ which is {\em simultaneously} orthonormal in all $T_g=L^2(\eI,g_*\leb)$, $g\in\G_0$.
For a general family,
the representation (\ref{on-repr}) should be replaced by
\begin{equation}\label{gen-repr}
\langle\D u(g),\D v(g)\rangle_{L^2}=\sum_{i,j=1}^\infty D_{\varphi_i}u(g)\cdot a_{ij}(g)\cdot D_{\varphi_j}v(g)
\end{equation}
where $a(g)=(a_{ij}(g))_{i,j\in\N}$ is the 'generalized inverse' to  $\Phi(g)=(\Phi_{ij}(g))_{i,j\in\N}$ with
$$\Phi_{ij}(g):=\langle\varphi_i,\varphi_j\rangle_{T_g}=\int_0^1\varphi_i(g_t)\varphi_j(g_t)dt.$$

In order to make these concepts rigorous, we have to introduce some notations.

\medskip

For fixed $n\in\N$ let $S_+(n)\subset\R^{n\times n}$ denote the set of symmetric nonnegative definite real $(n\times n)$-matrices. For each $A\in S_+(n)$ a unique element
$A^{-1}\in S_+(n)$, called {\em generalized inverse to $A$}, is defined by
$$A^{-1}x:=\left\{\begin{array}{ll}0&\ \mbox{ if }x\in \mbox{Ker}(A),\\
y&\ \mbox{ if }x\in\mbox{Ran}(A) \mbox{ with }x=Ay\end{array}\right.$$
 This definition makes sense since (by the symmetry of $A$) we have an orthogonal decomposition
$\R^n=\mbox{Ker}(A)\oplus\mbox{Ran}(A)$.
Obviously,
$$A^{-1}\cdot A=A\cdot A^{-1}=\pi_A$$
where $\pi_A$ denotes the projection onto $\mbox{Ran}(A)$.

Moreover, for each $A\in S_+(n)$ there exists a unique element
$A^{1/2}\in S_+(n)$, called {\em nonnegative square root of $A$}, satisfying $$A^{1/2}\cdot A^{1/2}=A.$$

Let $\Psi^{(n)}$ denote the map $A\mapsto A^{-1}$, regarded as a map from $S_+(n)\subset\R^{n\times n}$ to $\R^{n\times n}$, with
$\Psi^{(n)}_{ij}(A)=(A^{-1})_{ij}$ for $i,j=1,\ldots,n$. Similarly, put
$$\Xi^{(n)}: S_+(n)\to S_+(n), \ A\mapsto (A^{1/2})^{-1}=(A^{-1})^{1/2}.$$
Note that $\Psi^{(n)}(A)=\Xi^{(n)}(A)\cdot \Xi^{(n)}(A)$ for all $A\in S_+(n)$.

The maps $\Psi^{(n)}$ and $\Xi^{(n)}$ are smooth on the subset of positive definite matrices $A\in S_+(n)$ but unfortunately not on the whole set $S_+(n)$.
However, they can be approximated
from below (in the sense of quadratic forms) by smooth maps: there exists a sequence of $\C^\infty$ maps  $\Xi^{(n,l)}:\R^{n\times n}\to\R^{n\times n}$ with
$$\xi\cdot\Xi^{(n,k)}(A)\cdot\xi\le \xi\cdot\Xi^{(n,l)}(A)\cdot\xi$$
for all $A\in S_+(n), \xi\in\R^n$ and all $k,l\in\N$ with $k\le l$
and
$$\Xi_{ij}^{(n,l)}(A)\to\Xi_{ij}^{(n)}(A)=(A^{-{1/2}})_{ij}$$
for all $A\in S_+(n), i,j\in\{1,\ldots,n\}$ as $ l\to\infty$.
Put $\Psi^{(n,l)}(A)=\Xi^{(n,l)}(A)\cdot \Xi^{(n,l)}(A)$ for  $A\in \R^{n\times n}$. Then the sequence $(\Psi^{(n,l)})_{l\in\N}$ approximates $\Psi^{(n)}$ from below in the sense of quadratic forms.

\medskip

Now let us choose a family  $\{\varphi_i\}_{i\in\N}$ of smooth functions $\varphi_i:\eI\to\R$ which is total in $\C^0(\eI)$ w.r.t. uniform convergence (i.e. its linear hull is dense). Put
$$\Phi_{ij}(g):=\langle\varphi_i,\varphi_j\rangle_{T_g}=\int_0^1\varphi_i(g_x)\varphi_j(g_x)dx$$
and
$$a^{(n,l)}_{ij}(g)=\Psi^{(n,l)}_{ij}\left(\Phi(g)\right),\qquad \sigma^{(n,l)}_{ij}(g)=\Xi^{(n,l)}_{ij}(\Phi(g)).$$
Note that the maps $g\mapsto a^{(n,l)}_{ij}(g)$ and $g\mapsto \sigma^{(n,l)}_{ij}(g)$  (for each choice of $n,l,i,j$) belong to the class $\Zyl^\infty(\G_0)$.
Moreover, put
$$a^{(n)}_{ij}(g)=\Psi^{(n)}_{ij}\left(\Phi(g)\right).$$
Then obviously
the orthogonal projection $\pi_n$ onto the linear span of $\{\varphi_1,\ldots,\varphi_n\}\subset T_g=L^2(\eI,g_*\leb)$ is given by
$$\pi_nu=\sum_{i,j=1}^n a_{ij}^{(n)}(g)\cdot\langle u,\varphi_i\rangle_{T_g} \cdot\varphi_j$$
and
$$\langle \pi_nu,\pi_nv\rangle_{T_g}=
\sum_{i,j=1}^n \langle u,\varphi_i\rangle_{T_g} \cdot a_{ij}^{(n)}(g)\cdot\langle v,\varphi_j\rangle_{T_g}$$
for all $u,v\in T_g$.

\begin{theorem}
{\bf (i)}\
For each $n,l\in\N$
the form $(\EE^{(n,l)},\Zyl^1(\G_0))$ with
$$\EE^{(n,l)}(u,v)=
\sum_{i,j=1}^n \int_{\G_0} D_{\varphi_i}u(g)\cdot a_{ij}^{(n,l)}(g)\cdot D_{\varphi_j}v(g)\, d\Q_0(g)$$
is closable. Its closure is a Dirichlet form with generator being
 the Friedrichs extension of the symmetric operator
$(\LL^{(n,l)},\Zyl^2(\G_0))$ given by
\begin{equation}\label{fin-dim-op}
\LL^{(n,l)}=\sum_{i,j=1}^n a^{(n,l)}_{ij}\cdot D_{\varphi_i}D_{\varphi_j}+
\sum_{i,j=1}^n \left[D_{\varphi_i} a^{(n,l)}_{ij}+ a^{(n,l)}_{ij}\cdot V^\beta_{\varphi_i}\right]D_{\varphi_j}.
\end{equation}

{\bf (ii)}\
As $l\to\infty$
$$\EE^{(n,l)}\nearrow\EE^{(n)}$$
where
$$\EE^{(n)}(u,v)=
\sum_{i,j=1}^n \int_{\G_0} D_{\varphi_i}u(g)\cdot a^{(n)}_{ij}(g)\cdot D_{\varphi_j}v(g)\, d\Q_0(g).$$
for $u,v\in\Zyl^1(\G_0)$. Hence, in particular, $\EE^{(n)}$ is a Dirichlet form.

{\bf (iii)}\
As $n\to\infty$
$$\EE^{(n)}\nearrow\EE$$
(which provides an alternative proof for the closability of the form $(\EE,\Zyl^1(\G_0))$).
\end{theorem}

\begin{proof}
(i) The function $a^{(n,l)}_{i,j}$ on $\G_0$ is a cylinder function in the class $\Zyl^1(\G_0)$.
The integration by parts formula for the $D_{\varphi_i}$, therefore,  implies that
for all $u,v\in\Zyl^2(\G_0)$
\begin{eqnarray*}\EE^{(n,l)}(u,v)&=&\sum_{i,j}\int D_{\varphi_i}u(g)D_{\varphi_j}v(g)
a^{(n,l)}_{ij}(g)d\Q_0(g)\\
&=&\sum_{i,j}\int u(g)\cdot D^*_{\varphi_i}\left(a^{(n,l)}_{ij}D_{\varphi_j}v\right)(g)
\ d\Q_0(g)=-\int u(g)\cdot \LL^{(n,l)}v(g)\ d\Q_0(g).
\end{eqnarray*}
with
$$\LL^{(n,l)}=-\sum_{i,j=1}^nD^*_{\varphi_i}\left(a^{(n,l)}_{ij}D_{\varphi_j}\right).$$
Hence, $(\EE^{(n,l)}, \Zyl^2(\G_0))$ is closable and the generator of its closure is the Friedrichs extension of $(\LL^{(n,l)},\Zyl^2(\G_0))$.

(ii) The monotone convergence $\EE^{(n,l)}\nearrow\EE^{(n)}$ of the quadratic forms is an immediate consequence of the fact that $a^{(n,l)}(g)\nearrow a^{(n)}(g)$
(in the sense of symmetric matrices)
for each $g\in\G_0$ which in turn follows from  the defining properties of the approximations $\Psi^{(n,l)}$ of the generalized inverse $\Psi^{(n)}$.

The limit of an increasing sequence of Dirichlet forms is itself again a Dirichlet form provided it is densely defined which in our case is guaranteed since it is
finite on
$\Zyl^2(\G_0)$.

(iii)
Obviously, the $\EE^{n}, n\in\N$ constitute an increasing sequence of Dirichlet forms with $\EE^{n}\le\EE$ for all $n$. Moreover,
$\Zyl^1(\G_0)$ is a core for all the forms under consideration.
Hence, it suffices to prove that for each $u\in\Zyl^1(\G_0)$ and each $\epsilon>0$ there exists an $n\in\N$ such that
$$\left|\EE^{(n)}(u,u)-\EE(u,u)\right|\le\epsilon.$$
To simplify notation, assume that $u$ is of the form $u(g)=U(\int\alpha(g_t)dt)$ for some $U\in\C^1_c(\R)$ and some $\alpha\in\C^1(\eI)$.
By assumption, the set $\{\varphi_i,i\in\N\}$ is total in $\C^0(\eI)$ w.r.t. uniform convergence. Hence, for each $\delta>0$ there exist $n\in\N$ and
 $\varphi\in\mbox{span}(\varphi_1,\ldots,\varphi_n)$
with
$\|\alpha'-\varphi\|_{\sup}\le\delta$ which implies
$$\frac{\langle\alpha',\varphi\rangle_{T_g}}{\|\varphi\|_{T_g}}\ge{\|\varphi\|_{T_g}}-\delta\ge{\|\alpha'\|_{T_g}}-2\delta.$$
Thus
\begin{eqnarray*}
\EE(u,u)\ge\EE^{(n)}(u,u)&\ge&\int_{\G_0} U'(\smint\alpha(g_t)dt)^2\cdot\langle\alpha',\varphi\rangle^2_{T_g}\cdot\frac1{\|\varphi\|^2_{T_g}}d\Q_0(g)\\
&\ge&
\int_{\G_0} U'(\smint\alpha(g_t)dt)^2\cdot\left(\|\alpha'\|_{T_g}-2\delta\right)^2d\Q_0(g)\\
&\ge&
\int_{\G_0} U'(\smint\alpha(g_t)dt)^2\cdot\left(\frac1{1+\delta}\|\alpha'\|^2_{T_g}-4\delta\right)d\Q_0(g)\\
&\ge&
\frac1{1+\delta}\EE(u,u)-4\delta\|U'\|^2_{sup}.
\end{eqnarray*}
Hence, for $\delta$ sufficiently small, $\EE(u,u)$ and $\EE^{(n)}(u,u)$ are arbitrarily close to each other.
\end{proof}

\begin{remark} For any given $g_0\in\G_0$, let $(g_t)_{t\ge0}$ with $g_t: (x,\omega)\mapsto g_t^x(\omega)$ be the solution to the SDE
\begin{eqnarray*}
dg_t^x&=&\sum_{i,j=1}^n\sigma_{ij}^{(n,l)}(g_t)\cdot \varphi_j(g_t^x) \, dW_t^i\\
&&+\frac12\sum_{i,j=1}^n
a_{ij}^{(n,l)}(g_t)\cdot \varphi_j(g_t^x)\cdot\left(\varphi'_i(g_t^x)+V^\beta_{\varphi_i}(g_t)\right)dt\\
&&+\frac12\sum_{i,j=1}^n
\sum_{k,m=1}^n \partial_{km}\Psi_{ij}^{(n,l)}(\Phi(g_t))\cdot
\langle (\varphi_k\varphi_m)',\varphi_i\rangle_{T_g}\cdot\varphi_j(g_t^x)dt
\end{eqnarray*}
where $\partial_{km}\Psi_{ij}^{(n,l)}$ for $(k,m)\in \{1,\ldots,n\}^2$ denotes the 1st order partial derivative of the function $\Psi_{ij}^{(n,l)}:\R^{n\times n}\to\R$
with respect to the coordinate $x_{km}$.
Then the generator of the process coincides on $\Zyl^2(\G_0)$ with  the operator $\frac12\LL^{(n,l)}$ from (\ref{fin-dim-op}), the generator of the Dirichlet form $\EE^{(n,l)}$.

\medskip \rm Let us briefly comment on the various terms in the SDE from above:
\begin{itemize}
\item
The first one, $\sum_{i,j=1}^n\sigma_{ij}^{(n,l)}(g_t)\cdot \varphi_j(g_t^x) \, dW_t^i$ is the diffusion term, written in Ito form;
\item
the second one, $\frac12\sum_{i,j=1}^n
a_{ij}^{(n,l)}(g_t)\cdot \varphi_j(g_t^x)\cdot\varphi'_i(g_t^x)dt$ is a drift which comes from the transformation between Stratonovich and Ito form
(it would disappear if we wrote the diffusion term in Stratonovich form).
\item
The next one,
$\frac12\sum_{i,j=1}^n
a_{ij}^{(n,l)}(g_t)\cdot \varphi_j(g_t^x)\cdot V^\beta_{\varphi_i}(g_t)dt$ is
a drift which arises from our change of variable formula. Actually,
since
$$V^\beta_{\varphi_i}(g)=\beta\int_0^1\varphi'_i(g(y))dy+
\sum_{a\in J_g}\left[
\frac{\varphi'_i(g(a+))+\varphi'_i(g(a-))}2-\frac{\varphi_i(g(a+))-\varphi_i(g(a-))}{g(a+)-g(a-)}\right],
$$
it consists of two parts, one originates in the logarithmic derivative of the entropy of the $g$'s
(which finally will force the process to evolve as a stochastic perturbation of the heat equation),
the other one is created by the jumps of the $g$'s.
\item
The last term,
$\frac12\sum_{i,j=1}^n
\sum_{k,m=1}^n \partial_{km}\Psi_{ij}^{(n,l)}(\Phi(g_t))\cdot
\langle (\varphi_k\varphi_m)',\varphi_i\rangle_{T_g}\cdot\varphi_j(g_t^x)dt$
involves the derivative of the diffusion matrix. It
arises from the fact that the generator is originally given in divergence form.
\end{itemize}
\end{remark}
\end{subsection}
\begin{subsection}{The Wasserstein Diffusion $(\mu_t)$ on $\Pe_0$}

The objects considered previously -- derivative,
Dirichlet form and Markov process on $\G_0$ --
 have canonical counterparts on $\Pe_0$.
The key to these objects is the bijective map $\chi:\G_0\to\Pe_0$, $g\mapsto g_*\leb$.

\medskip

We denote by $\Zyl^k(\Pe_0)$ the set of all ('cylinder') functions $u: \Pe_0\to\R$ which can be written as
\begin{equation}
u(\mu)=U\left(\int_0^1 \alpha_1d\mu,\ldots,\int_0^1\alpha_md\mu\right)
\end{equation}
with some $m\in\N$, some $U\in\C^k(\R^m)$ and some $\vec\alpha=(\alpha_1,\ldots,\alpha_m)\in\C^k(\eI,\R^m)$ .
The subset of  $u\in\Zyl^k(\Pe_0)$ with
$\alpha'_i(0)=\alpha'_i(1)=0$ for all $i=1,\ldots,m$ will be denoted by
 $\Zyl^k_0(\Pe_0)$.
For $u\in\Zyl^1(\Pe_0)$ represented as above we define its gradient $D u(\mu)\in L^2(\eI,\mu)$ by
$$D u (\mu)=\sum_{i=1}^m \partial_i U(\smint \vec \alpha d\mu)\cdot \alpha_i'(.)$$
with norm
$$\|D u(\mu)\|_{L^2(\mu)}=\left[
\int_0^1\left|\sum_{i=1}^m \partial_i U(\smint \vec \alpha d\mu)\cdot \alpha_i'\right|^2d\mu\right]^{1/2}.$$
The tangent space at a given point $\mu\in\Pe_0$ can be identified with $L^2(\eI,\mu)$. The action of a tangent vector $\varphi\in L^2(\eI,\mu)$
on $\mu$ ('exponential map') is given by the push forward
$\varphi_*
\mu$.

\begin{theorem}
{\bf(i)} \
The image of the
Dirichlet form defined in (\ref{Wasserstein-dir-int}) under the map $\chi$ is the regular, strongly
local, recurrent  {\em Wasserstein Dirichlet form} $\EE$ on $L^2(\Pe_0,\Pp_0)$ defined on its core $\Zyl^1(\Pe_0)$ by
\begin{equation}\label{p-wasserstein-form}
\EE(u,v)=\int_{\Pe_0} \langle D u(\mu), D v(\mu)\rangle_{L^2(\mu)}^2 d\Pp_0(\mu).
\end{equation}
The Dirichlet form has a square field operator, defined on $\Dom(\EE)\cap L^\infty$, and given on $\Zyl^1(\Pe_0)$ by
$$\Gamma(u,v)(\mu)=
\langle D u(\mu), D v(\mu)\rangle_{L^2(\mu)}^2.$$
The intrinsic metric for the Dirichlet form is the $L^2$-Wasserstein distance $d_W$. More precisely, a continuous function $u:\Pe_0\to\R$ is 1-Lipschitz w.r.t. the $L^2$-Wasserstein distance if and only if it belongs to $\Dom(\EE)$ and $\Gamma(u,u)(\mu)\le1$ for
$\Pp_0$-a.e. $\mu\in\Pe_0$.

{\bf(ii)}\
The generator of the Dirichlet form is the Friedrichs extension of the symmetric operator $(\LL,\Zyl^2_0(\Pe_0)$ on $L^2(\Pe_0,\Pp_0)$ given as
$\LL=\LL_1+\LL_2+\beta\cdot\LL_3$ with
\begin{eqnarray*}
\LL_1 u(\mu)&=&
\sum_{i,j=1}^m{\partial_i\partial_j}U(\smint\vec\alpha d\mu)\cdot\int_0^1\alpha_i'\alpha_j'd\mu;
\\
\LL_2 u(\mu)&=&
\sum_{i=1}^m{\partial_i}U(\smint\vec\alpha d\mu)\cdot
\left(\sum_{I\in\Gaps(\mu)}\left[\frac{\alpha_i''(I_-)+\alpha_i''(I_+)}2-\frac{\alpha_i'(I_+)-\alpha_i'(I_-)}{|I|}\right] -\frac{\alpha_i''(0)+\alpha_i''(1)}2\right)
\\
\LL_3u(\mu)&=&
\sum_{i=1}^m{\partial_i}U(\smint\vec\alpha d\mu)\cdot\int_0^1 \alpha_i'' d\mu.
\end{eqnarray*}
Recall that $\Gaps(\mu)$ denotes the set of intervals
$I=\,]I_-,I_+[\,\subset \eI$ of maximal length with $\mu(I)=0$ and
$|I|$ denotes the length of such an interval.

{\bf (iii)}\
For $\Pp_0$-a.e. $\mu_0\in\Pe_0$, the associated {Markov process} $(\mu_t)_{t\ge0}$ on $\Pe_0$ starting in $\mu_0$, called {\em Wasserstein diffusion},
with generator $\frac12\LL$
 is given as
$$\mu_t(\omega)=g_t(\omega)_*\leb$$
where $(g_t)_{t\ge0}$ is the Markov process on $\G_0$ associated with the Dirichlet form of Theorem \ref{wasserstein-dir-form},
starting in $g_0:=\chi^{-1}(\mu_0)$.

For each $u\in\Zyl^2_0(\Pe_0)$ the process
$$u(\mu_t)-u(\mu_0)-\frac12\int_0^t \LL u(\mu_s)ds$$
is a martingale whenever the distribution of $\mu_0$ is chosen to be absolutely continuous w.r.t. the entropic measure $\Pp_0$.
Its quadratic variation process is
$$\int_0^t \Gamma(u,u)(\mu_s)ds.$$
\end{theorem}

\begin{remark}  \rm  $\LL_1$ is the second order part ('diffusion part') of the generator $\LL$, $\LL_2$ and $\LL_3$ are first order operators ('drift parts').
The operator $\LL_1$ describes the diffusion on $\Pe_0$ in all directions of the respective tangent spaces. This means that the process $(\mu_t)$
at each time $t\geq 0$ experiences the full 'tangential' $L^2(\eI,\mu_t)$-noise.

  $\LL_3$ is the generator of the deterministic  semigroup ('Neumann heat flow') $(H_t)_{t\ge0}$ on $L^2(\Pe_0,\Pp_0)$ given by
$$H_tu(\mu)=u(h_t\mu).$$
Here $h_t$ is the heat kernel on $\eI$ with reflecting ('Neumann') boundary conditions and $h_t\mu(.)=\int_0^1 h_t(.,y)\mu(dy)$.
Indeed,
for each $u\in\Zyl^1_0(\Pe_0)$ given as $u(g)=U(\int\vec\alpha d\mu)$ we obtain
$H_tu(\mu)= U\left(\smint\smint\vec \alpha(x)h_t(x,y)\mu(dy)dx\right)$ and thus
\begin{eqnarray*}
\partial_t H_tu(\mu)&=&\sum_{i=1}^m \partial_iU(h_t\mu)\cdot\partial_t\smint\smint \alpha_i(x)h_t(x,y)\mu(dy)dx\\
&=&\sum_{i=1}^m \partial_iU(h_t\mu)\cdot\smint\smint \alpha_i(x)h_t''(x,y)\mu(dy)dx\\
&=&\sum_{i=1}^m \partial_iU(h_t\mu)\cdot\smint\smint \alpha_i''(x)h_t(x,y)\mu(dy)dx
= \LL_3 H_tu(\mu).
\end{eqnarray*}
Note that $\LL$ depends on $\beta$ only via the drift term $\LL_3$ and $\frac1\beta \LL\to\LL_3$
 as $\beta\to\infty$.
 \end{remark}

The following statement, which in the finite dimensional case is known as  Varadhan's formula, exhibits another  close relationship between $(\mu_t)$ and the geometry of $(\Pe(\eI), d_W)$.   The Gaussian short time asymptotics of the process $(\mu_t)_{t\ge0}$ are governed by the $L^2$-Wasserstein distance. 
\begin{corollary} \label{gaussianbehaviour}
For measurable sets $A,B\in\Pe_0$  with positive $\Pe_0$-measure,  let $d_W(A,B) = \inf \{ d_W(\nu,\tilde\nu)\,| \, \nu\in A , \tilde\nu\in B\}$
and  $p_t(A,B)= \int_A\int_B p_t(\nu,d\tilde\nu)\Pe_0(d\nu)$
where $p_t(\nu,d\tilde\nu)$ denotes the transition semigroup for the process $(\mu_t)_{t\ge0}$.

Then
\begin{equation}
  \lim _{t \to 0}  t \log p_t(A,B) =  - \frac {d_W(A,B)^2}{2}.
\label{varform}
\end{equation}
\end{corollary}

\begin{proof} This type of result is known as Varadhan's formula. Its respective form  for $(\EE,\Dom(\EE)$ on $L^2(\Pe_0,\Pp_0)$
holds true by the very general results of \cite{MR1988472} for conservative symmetric diffusions,
 and the identification of the intrinsic metric as $d_W$ in our previous Theorem.
 \end{proof}

Due to the sample path continuity of $(\mu_t)$ the Wasserstein diffusion is equivalently characterized by the following martingale problem. Here  we  use the notation $\langle\alpha,\mu_t\rangle=\int_0^1\alpha(x)\mu_t(dx)$.

\begin{corollary}\label{simplemartingaleproblem}
For each $\alpha\in\C^2(\eI)$ with $\alpha'(0)=\alpha'(1)=0$ the process
\begin{eqnarray*}M_t&=&\langle\alpha,\mu_t\rangle-\frac\beta2\int_0^t
\langle\alpha'',\mu_s\rangle ds\\
&&-\frac12\int_0^t\left(
\sum_{I\in\Gaps(\mu_s)}\left[\frac{\alpha''(I_-)+\alpha''(I_+)}2-\frac{\alpha'(I_+)-\alpha'(I_-)}{|I|}\right] -\frac{\alpha''(0)+\alpha''(1)}2\right)ds
\end{eqnarray*}
is a continuous martingale with quadratic variation process
$$[ M]_t=\int_0^t\langle(\alpha')^2,\mu_s\rangle ds.$$

\end{corollary}

\begin{remark}\rm
For illustration one may compare corollary \ref{simplemartingaleproblem}  for $(\mu_t)$  in the case  $\beta=1$
to the respective martingale problems for four other well-known measure valued process, say on the real line,
namely the so-called super-Brownian motion  or Dawson-Watanabe process $(\mu^{DW}_t)$, the Fleming-Viot process $(\mu^{FW})$,
both of which we can consider with the Laplacian as drift, the Dobrushin-Doob process $(\mu^{DD}_t)$ which is the empirical
measure of independent Brownian motions with locally finite Poissonian starting distribution,
 cf. \cite{MR1612725}, and finally simply the empirical measure process of a single Brownian motion $(\mu_t^{BM}= \delta_{X_{t}})$.
 For each $i \in \{ DW, FV, DD, BM\}$ and sufficiently regular $\alpha: \R \to \R$ the process
 $M^i_t := \langle \alpha , \mu^i_t\rangle - \frac12\int _0^t \langle  \alpha'', \mu^i_s\rangle ds $ is a continuous martingale with quadratic variation process
\begin{eqnarray*}
&[M^{DW}]_t &= \ \int_0^t \langle \alpha^2,\mu^{DW}_s\rangle ds,  \\
&[M^{FV}]_t &= \ \int_0^t [\langle \alpha^2,\mu^{FV}_s\rangle - (\langle \alpha,\mu^{FV}_s\rangle)^2]ds, \\
&[M^{DD}]_t &= \ \int_0^t \langle (\alpha')^2,\mu^{DD}_s\rangle ds, \\
&[M^{BM}]_t &= \ \int_0^t \langle (\alpha')^2,\mu^{BM}_s\rangle ds.
\end{eqnarray*}
\end{remark}
\end{subsection}
In view of corollary \ref{gaussianbehaviour} the apparent similarity of $\mu^{DD}$ and $\mu^{BM}$ to the Wasserstein diffusion $\mu$ is no suprise. However, the effective state spaces of $\mu^{DD}$, $\mu^{BM}$ and $\mu_t$ are as much different as their invariant measures.
\end{section}

\nocite{MR1457615, MR1242575,MR1612725,MR1457615,MR1989437,debzamb,MR1172940}

\bibliographystyle{alpha}

\begin{thebibliography}{CEMS01}

\bibitem[AKR98]{MR1612725}
S.~Albeverio, Yu.~G. Kondratiev, and M.~R{\"o}ckner.
\newblock Analysis and geometry on configuration spaces.
\newblock {\em J. Funct. Anal.}, 154(2):444--500, 1998.


\bibitem[AM06]{AM06}
H{\'e}l{\`e}ne Airault and Paul Malliavin.
Quasi-invariance of Brownian measures on the group of circle homeomorphisms and infinite-dimensional Riemannian geometry.
{\em J. Funct. Anal.} 241 (1): 99-142, 2006.



\bibitem[AMT04]{MR2082490}
H{\'e}l{\`e}ne Airault, Paul Malliavin, and Anton Thalmaier.
\newblock Canonical {B}rownian motion on the space of univalent functions and
  resolution of {B}eltrami equations by a continuity method along stochastic
  flows.
\newblock {\em J. Math. Pures Appl. (9)}, 83(8):955--1018, 2004.

\bibitem[AR02]{AR02}
 H{\'e}l{\`e}ne Airault and Jiagang Ren.
Modulus of continuity of the canonic Brownian motion ``on the group of diffeomorphisms of the circle".
{\em J. Funct. Anal.} 196 (2): 395-426, 2002.


\bibitem[Ber99]{MR1746300}
Jean Bertoin.
\newblock Subordinators: examples and applications.
\newblock In {\em Lectures on probability theory and statistics (Saint-Flour,
  1997)}, volume 1717 of {\em Lecture Notes in Math.}, pages 1--91. Springer,
  Berlin, 1999.

\bibitem[Bre91]{MR1100809}
Yann Brenier.
\newblock Polar factorization and monotone rearrangement of vector-valued
  functions.
\newblock {\em Comm. Pure Appl. Math.}, 44(4):375--417, 1991.


\bibitem[CEMS01]{MR2002k:58038}
Dario Cordero-Erausquin, Robert~J. McCann, and Michael Schmuckenschl{\"a}ger.
\newblock A {R}iemannian interpolation inequality \`a la {B}orell, {B}rascamp
  and {L}ieb.
\newblock {\em Invent. Math.}, 146(2):219--257, 2001.

\bibitem[Daw93]{MR1242575}
Donald~A. Dawson.
\newblock Measure-valued {M}arkov processes.
\newblock In {\em \'Ecole d'\'Et\'e de Probabilit\'es de Saint-Flour
  XXI---1991}, volume 1541 of {\em Lecture Notes in Math.}, pages 1--260.
  Springer, Berlin, 1993.

\bibitem[DZ06]{debzamb}
Arnaud Debussche and Lorenzo Zambotti.
\newblock {Conservative Stochastic Cahn-Hilliard equation with reflection}.
\newblock 2006.
\newblock Preprint.

\bibitem[{\'E}Y04]{MR2074696}
Michel {\'E}mery and Marc Yor.
\newblock A parallel between {B}rownian bridges and gamma bridges.
\newblock {\em Publ. Res. Inst. Math. Sci.}, 40(3):669--688, 2004.


\bibitem[Fan02]{Fang02} Shizan Fang.
Canonical Brownian motion on the diffeomorphism group of the circle.
{\em J. Funct. Anal.} 196 (1): 162-179, 2002.

\bibitem[Fan04]{MR2091355}
Shizan Fang.
\newblock Solving stochastic differential equations on {${\rm Homeo}(S\sp 1)$}.
\newblock {\em J. Funct. Anal.}, 216(1):22--46, 2004.

\bibitem[FOT94]{MR96f:60126}
Masatoshi Fukushima, Y{\=o}ichi Oshima, and Masayoshi Takeda.
\newblock {\em Dirichlet forms and symmetric {M}arkov processes}.
\newblock Walter de Gruyter \& Co., Berlin, 1994.

\bibitem[Han02]{MR1902190}
Kenji Handa.
\newblock Quasi-invariance and reversibility in the {F}leming-{V}iot process.
\newblock {\em Probab. Theory Related Fields}, 122(4):545--566, 2002.

\bibitem[HR03]{MR1988472}
Masanori Hino and Jos{\'e}~A. Ram{\'{\i}}rez.
\newblock Small-time {G}aussian behavior of symmetric diffusion semigroups.
\newblock {\em Ann. Probab.}, 31(3):1254--1295, 2003.

\bibitem[JKO98]{MR1617171}
Richard Jordan, David Kinderlehrer, and Felix Otto.
\newblock The variational formulation of the {F}okker-{P}lanck equation.
\newblock {\em SIAM J. Math. Anal.}, 29(1):1--17 (electronic), 1998.

\bibitem[Kin93]{MR1207584}
J.~F.~C. Kingman.
\newblock {\em Poisson processes}, volume~3 of {\em Oxford Studies in
  Probability}.
\newblock The Clarendon Press Oxford University Press, New York, 1993.
\newblock , Oxford Science Publications.

\bibitem[Mal99]{MR1713340}
Paul Malliavin.
\newblock The canonic diffusion above the diffeomorphism group of the circle.
\newblock {\em C. R. Acad. Sci. Paris S\'er. I Math.}, 329(4):325--329, 1999.

\bibitem[McC97]{MR98e:82003}
Robert~J. McCann.
\newblock A convexity principle for interacting gases.
\newblock {\em Adv. Math.}, 128(1):153--179, 1997.

\bibitem[NP92]{MR1172940}
D.~Nualart and {\'E}.~Pardoux.
\newblock White noise driven quasilinear {SPDE}s with reflection.
\newblock {\em Probab. Theory Related Fields}, 93(1):77--89, 1992.

\bibitem[Ott01]{MR2002j:35180}
Felix Otto.
\newblock The geometry of dissipative evolution equations: the porous medium
  equation.
\newblock {\em Comm. Partial Differential Equations}, 26(1-2):101--174, 2001.

\bibitem[OV00]{MR2001k:58076}
F.~Otto and C.~Villani.
\newblock Generalization of an inequality by {T}alagrand and links with the
  logarithmic {S}obolev inequality.
\newblock {\em J. Funct. Anal.}, 173(2):361--400, 2000.

\bibitem[RS80]{Reed/Simon}
Michael Reed and Barry Simon. {\em Functional Analysis I.} Academic Press 1980.

\bibitem[vRS05]{MR2142879}
Max-K. von Renesse and Karl-Theodor Sturm.
\newblock Transport inequalities, gradient estimates, entropy, and {R}icci
  curvature.
\newblock {\em Comm. Pure Appl. Math.}, 58(7):923--940, 2005.


\bibitem[Sch97]{MR1457615}
Alexander Schied.
\newblock Geometric aspects of {F}leming-{V}iot and {D}awson-{W}atanabe
  processes.
\newblock {\em Ann. Probab.}, 25(3):1160--1179, 1997.

\bibitem[Sta03]{MR1989437}
Wilhelm Stannat.
\newblock On transition semigroups of {$(A,\Psi)$}-superprocesses with
  immigration.
\newblock {\em Ann. Probab.}, 31(3):1377--1412, 2003.

\bibitem[Stu06]{MR2237206}
Karl-Theodor Sturm.
\newblock On the geometry of metric measure spaces. {I}.
\newblock {\em Acta Math.}, 196(1):65--131, 2006.

\bibitem[TVY01]{MR1853759}
Natalia Tsilevich, Anatoly Vershik, and Marc Yor.
\newblock An infinite-dimensional analogue of the {L}ebesgue measure and
  distinguished properties of the gamma process.
\newblock {\em J. Funct. Anal.}, 185(1):274--296, 2001.

\bibitem[Vil03]{MR1964483}
C{\'e}dric Villani.
\newblock {\em Topics in optimal transportation}, volume~58 of {\em Graduate
  Studies in Mathematics}.
\newblock American Mathematical Society, Providence, RI, 2003.
\end{thebibliography}

\def\cprime{$'$} \def\cprime{$'$}

\end{document}